%% file: 20230122_arXiv_nonidentifiablemodel.tex
\documentclass[a4paper,10pt]{article}

\usepackage{dsfont}
\usepackage{amssymb}
\usepackage{latexsym}
\usepackage{amsmath}
\usepackage{xcolor}
\usepackage{comment}
\usepackage{amsthm}
\usepackage[dvips]{graphicx}
\usepackage{layout} 
\usepackage{ulem}
\usepackage{enumerate}
\usepackage{bm}
\usepackage{bbm} 
\usepackage[bbgreekl]{mathbbol} 
\usepackage{authblk}
\usepackage{setspace} 
\newif\ifrs
\rstrue
\ifrs \usepackage{mathrsfs} \fi  
\newif\ifcol
\coltrue 

\newtheorem{theorem}{Theorem}[section]

\newtheorem{lemma}[theorem]{Lemma}

\newtheorem{proposition}[theorem]{Proposition}

\newtheorem{remark}[theorem]{Remark}
\newtheorem{example}[theorem]{Example}
\numberwithin{equation}{section}
\newtheorem{rem}[theorem]{Remark}
\newtheorem{theorem*}{Theorem}
\newtheorem{ass*}[theorem*]{Assumption}
\newtheorem{note*}[theorem*]{Note}
\newtheorem{lemma*}[theorem*]{Lemma}
\newtheorem{definition*}[theorem*]{Definition}
\newtheorem{proposition*}[theorem*]{Proposition}
\newtheorem{corollary*}[theorem*]{Corollary}
\newtheorem{remark*}[theorem*]{Remark}
\newtheorem{example*}[theorem*]{Example}
\numberwithin{equation}{section}
\newif\ifcol
\colfalse
\ifcol
\newcommand{\colorr}{\color[rgb]{0.8,0,0}}

\newcommand{\colorn}{\color[rgb]{1,1,1}}

\else
\newcommand{\colr}{\color{black}}
\newcommand{\colorr}{\color{black}}
\newcommand{\colb}{\color[rgb]{0,0,0.8}}
\newcommand{\colorn}{\color{black}}

\fi
\newif\ifcol
\colfalse
\ifcol

\else
\fi
\newif\ifcol
\colfalse
\ifcol

\else
\fi
\newif\ifcol
\colfalse
\ifcol

\else
\fi
\newif\ifcol
\colfalse
\ifcol

\else
\fi
\newif\ifcol
\colfalse
\ifcol

\else
\fi
\excludecomment{en-text}
\includecomment{jp-text}
\includecomment{comment}
\input nakamacro031024.tex

\begin{document}

\title{
	Penalized estimation for non-identifiable models
\footnote{
This work was in part supported by 
Japan Science and Technology Agency CREST  JPMJCR2115; 
Japan Society for the Promotion of Science Grants-in-Aid for Scientific Research 
No.\,17H01702 (Scientific Research);  
and by a Cooperative Research Program of the Institute of Statistical Mathematics. 
}
}
\author[1,2]{Junichiro Yoshida}
\author[1,2]{Nakahiro Yoshida}
\affil[1]{Graduate School of Mathematical Sciences, University of Tokyo
\footnote{Graduate School of Mathematical Sciences, University of Tokyo: 3-8-1 Komaba, Meguro-ku, Tokyo 153-8914, Japan. e-mail: nakahiro@ms.u-tokyo.ac.jp}
}
\affil[2]{CREST, Japan Science and Technology Agency
}
\maketitle
\ \\
{\it Summary} 
{ 
	We derive   asymptotic properties of   penalized estimators  
	 for singular models  for which identifiability may break and  the true parameter values can lie on the boundary of the parameter space.  
	Selection consistency of the estimators is also validated.
	The  problem  that the true values lie on  the boundary is dealt with by  our previous results that are  applicable to singular models, besides, penalized estimation and  non-ergodic statistics. In order to overcome non-identifiability, 
	 we consider a suitable penalty such as the non-convex Bridge and the adaptive Lasso that  stabilizes the asymptotic behavior of the  estimator and shrinks inactive parameters. Then  the estimator converges to one of the most  parsimonious values among all the  true values. 
	 In particular, the oracle property can also be obtained even if  parametric structure of the singular model  is so complex that likelihood ratio tests for model selection are labor intensive to perform.
	Among many potential applications, 
	the examples handled in the paper  are:  (i) the superposition of parametric proportional hazard models and  (ii) a counting process having intensity with multicollinear covariates.
}
\section{Introduction}
{ The purpose of this article is to obtain  the asymptotic behavior of the penalized estimator  that possesses selection consistency under singular models where identifiability  fails and  the true parameter values may lie on the boundary of the parameter space.  For that, we apply the limit theorem  in Yoshida and Yoshida \cite{yoshida2022quasi},  which is  a generalization of  local asymptotic theory  established by  Ibragimov and Khas'minskii {\colorr \cite{ibragimov1972asymptotic, IbragimovHascprimeminskiui1981}}  and which is valid  {\colr for singular models, besides,  non-ergodic statistics and penalized estimation.}  {\colr The limit theorem, recalled  in Section 2, can overcome  the singularity that the true values lie on the boundary by a more general local approximation of the parameter space  than in previous studies. (See Remark \ref{remarkHT-infty} (i).)}

\begin{en-text}Let us recall the limit theorem in \cite{yoshida2022quasi} briefly. We denote by $\cale^\epsilon =\{\calx, \cala, P_\theta^\epsilon, \theta \in\Theta\}$  a sequence of statistical experiments with $\epsilon \in (0, 1]$ for a parameter space $\Theta \subset \bbR^{\sfp}$. Let $a_T \in GL(\sfp)$ be a positive normalizing factor tending to zero as $T \to \infty$. For a $\theta^* \in \Theta$, define a random field $\bbZ_T$ by 
\bea\label{Z_epsilon}
\bbZ_T (u) = \frac{dP^\epsilon_{\theta^*+\varphi(\epsilon)u}}{dP^\epsilon_{\theta^*}}(X^\epsilon)\qquad(u \in \bbR^m).
\eea
Suppose that $Z_\epsilon$ can be extended on a $\hat C(\bbR^m)$-valued random variable\footnote{$\hat C(\bbR^m)$ is the space of  continuous functions on $\bbR^m$ that tends to zero at the infinity} and that the extended $Z_\epsilon$ converges to $Z$ in $\hat{C}(\bbR^m)$, where $Z$ is some $\hat C (\bbR^m)$-valued random variable. Then we have
\bea\label{Ibragimov}
\hat u_\epsilon =\underset{ u}{\rm argmax\,} Z_\epsilon(u) \overset{d}{\to}   \hat u = \underset{ u\in \bbR^m}{\rm argmax\,} Z(u).
\eea
When the experiment $\cale^\epsilon$ is locally asymptotically normal (LAN), $Z(u)$ takes the form of $Z(u)=\exp\big(\Delta(\theta^*)-2^{-1}I(\theta^*)[u^{\otimes2}]\big)$, where $\Delta(\theta^*) \sim N_m\big(0, I(\theta^*)\big)$ and $I(\theta^*)$ is the Fisher information matrix, and we obviously obtain  $\hat u= I(\theta^*)^{-1}\Delta(\theta^*) \sim N_m\big(0, I(\theta^*)^{-1}\big)$. In this paper, we extend this theory to a more general one  
that  can be  applied even when $\theta^*$ may lie on the boundary of $\Theta$. Moreover, in our extended {\colorr theory},   $Z_\epsilon$  is not necessarily defined as (\ref{Z_epsilon}) and not necessarily asymptotically quadratic. Therefore, as explained below,  penalized estimation can be handled as well as quasi-maximum likelihood estimation. 
Besides, our results can apply to the so-called non-ergodic statistics, where the Fisher information is random in the limit,  including the regular experiment that is locally asymptotically mixed normal (LAMN).

When the true parameter value $\theta^*$  lies on the boundary of $\Theta$,  a cone set locally approximating  $\Theta$ at $\theta^*$ is usually used. This set was 
introduced by Chernoff  \cite{chernoff1954distribution}, and with it, Self and Liang \cite{self1987asymptotic} derived the limit distribution of  the maximum likelihood estimation (MLE)   when $\theta^*$ is on the boundary.  Also, using a cone set that is a generalization of Chernoff's cone set, Andrews \cite{andrews1999estimation} derived the asymptotic distribution of the quasi-maximum likelihood estimator (QMLE)  when $\theta^* $ is on the boundary. These cone sets are denoted by $\Lambda$.
However, if  the rate of  convergence of a given estimator  to $\theta^*$ becomes more  complex, then  $\Theta$  may not be locally approximated by any cone set. In this paper, the asymptotic behavior of the estimator  is derived even in such complex cases  by generalizing the local approximation method for the parameter space.  Instead of $\Lambda$, we denote by $U$ the general set locally approximating $\Theta$ at $\theta^*$, which is not necessarily a cone. (Examples of $U$ are listed in Section 2.3.)  Then our theorem (Theorem \ref{thm1}) shows that (\ref{Ibragimov}) is generalized as
\bea\label{Ibragimovgeneral}
\hat u_\epsilon =\underset{ u}{\rm argmax\,} Z_\epsilon(u) \overset{d}{\to}   \hat u = \underset{ u\in  U}{\rm argmax\,} Z(u).
\eea
\end{en-text}

For regular statistical models, there are many studies on the penalized maximum likelihood estimator (PMLE). 
 Frank and Friedman \cite{frank1993statistical} introduced the Bridge-type penalty 
\beas
p_\lambda (\theta) \yeq 
\lambda \sum_{i=1}^{\sfp} |\theta_i|^q
\qquad\big(\theta = (\theta_1, ..., \theta_\sfp) \in \Theta\big),
\eeas
where $q > 0$  and $\lambda > 0$ are tuning parameters.  In particular, this penalty  gives  the Lasso estimator   when $q=1$ (Tibshirani
\cite{tibshirani1996regression}).   The limit distribution of the PMLE in the Bridge case is obtained by Knight and Fu \cite{fu2000asymptotics}, while  the oracle property of the adaptive Lasso  is shown by Zou \cite{zou2006adaptive}. 
In this article, we treat  penalties that possess model selection consistency, including the non-convex Bridge with $q<1$ and the adaptive Lasso.
{Also, many authors have studied  the penalized quasi-maximum likelihood estimator (PQMLE) in regular  models:   De Gregorio and Iacus \cite{de2012adaptive},  Ga\"{i}ffas and Matulewicz \cite{gaiffas2019sparse}, Masuda and Shimizu \cite{masuda2017moment}, Kinoshita and Yoshida \cite{kinoshita2019penalized} and  Umezu et al. \cite{umezu2019aic}, just to mention a few. 

	{Let us explain the principle of our penalized estimation  that  overcomes  non-identifiability. Let  $\theta \in \Theta$ be the unknown parameter of interest whose true value is not necessarily uniquely determined.  
Denote by $\Theta^*$ the set consisting of the true values of $\theta \in \Theta$.
The target value $\theta^* $ of our penalized estimation is 
one of the true values  which is sufficiently parsimonious in that
\bea\label{parsimonious} \Theta^* \cap \big\{\theta=(\theta_1, ..., \theta_\sfp) \in \Theta ;~\text{for every $j=1, ..., \sfp$, }\,  \theta^*_j=0 \implies \theta_j =0 \big\} \yeq \{\theta^*\}.
\eea
Then by adding  to the quasi-log likelihood function a penalty term  which is uniquely minimized  on $\Theta^*$ by $\theta^*$,  the PQMLE converges to $\theta^*$ in probability, and we can derive its  limit distribution  and selection consistency.  In particular, its oracle property can also be obtained.

{One of  the  simplest example is  the linear regression model
	\bea\label{linearbasic}
	Y \yeq \theta_1X_1 + ... + \theta_\sfp X_\sfp + \epsilon,
	\eea
	where the explanatory variables $X =(X_1, ..., X_\sfp)$ are perfectly multicollinear in that $A X^\prime=0$\footnote{\colr $B^\prime$ denotes the transpose of a matrix $B$.}  
	for some $(\sfp -\sfr) \times \sfp$ matrix $A$. Here $\sfr$ denotes the dimension of the linear space spanned by $X$. Then the true values set $\Theta^*$ is determined as $\{\theta=(\theta_1, ..., \theta_\sfp) \in \Theta ; A\theta^\prime =A\theta_0^\prime\}$, where $\theta_0$ is one of the true values. Then considering   the Bridge-type penalty  with $q<1$  
	that is  uniquely   minimized on $\Theta^*$  by  some element  $\theta^*\in \Theta^*$ satisfying (\ref{parsimonious}), we are expected to obtain the penalized least squares estimator which converges to $\theta^*$ in probability, and can derive its limit distribution and selection consistency as well as its oracle property. {The idea of estimating the unique minimizer $\theta^*$ of the penalty  is  mentioned briefly in the introduction of Knight and Fu \cite{fu2000asymptotics}. 
		Also, this estimation method is certainly valid for time-series data such as a counting process  whose intensity is expressed as the following model. 
		\bea\label{introadditive}
		\lambda_t(\alpha_1, ..., \alpha_\sfa) \yeq \alpha_1X_t^1+\cdots \alpha_\sfa X_t^\sfa\qquad(t\geq 0), 
		\eea
		where $\alpha_1, ..., \alpha_\sfa$ are the non-negative unknown  parameters, and $\{X_t^j\}_{t\geq 0}$ $(j=1, ..., \sfa)$ are non-negative observable  processes that may be multicollinear. 
		{\colr 
			 Without multicollinearity, the additive intensity model as (\ref{introadditive})  is found in the literature of survival analysis (Aalen \cite{aalen1978nonparametric}).
			Lindsey \cite{lindsey1995fitting} shows that   the additive model  gives a better fit than the corresponding multiplicative intensity model in the parametric case, using event or life history data.}
		The additive model  is also treated by Chornoboy et.al. \cite{chornoboy1988maximum} in the context of neural network. 
		Due to the natural boundary constraint $\alpha_j \geq 0$, this example  is somewhat more difficult than usual linear regression models without  the constraint such as (\ref{linearbasic}). Therefore, prior to those models,  we treat (\ref{introadditive}) in Section \ref{pointprocesswithmulticolinear}, but  under the multicollinearity of
		$X^1, ..., X^\sfa$. 
	}
}

Moreover,  this approach using penalized estimation  is valid even when a nuisance  parameter $\tau \in \calt$  exists  with $\theta$. 
Denote by $\Xi^*$ the set of the true values of $(\theta, \tau) \in \Xi= \Theta \times\calt$, and  by $\Theta^*$ the projection  of $\Xi^*$ as $\Theta^*=\big\{\theta \in \Theta ; (\theta, \tau)\in \Xi^*\text{ for some $\tau \in \calt$}\big\}$. 
Then if we impose a suitable penalty  on $\theta$  which  is uniquely minimized  on $\Theta^*$ by some parsimonious element $\theta^* \in \Theta^*$ satisfying (\ref{parsimonious}),  then
the PQMLE converges to $\theta^*$ in probability, and  its  limit distribution and selection consistency can be obtained as well as its oracle property.

An example {\colr of when a nuisance parameter exists} is estimation of the intensity  of a counting process which is expressed as the  superposition of some  parametric proportional hazards models. {\colr That is,}  the intensity $\lambda_t$ is parametrized as 
\bea\label{superpositiongeneral}
\lambda_t \yeq g+ \alpha_1 e^{\beta_1 \cdot X_t^1}+\cdots  +\alpha_\sfa e^{\beta_\sfa \cdot X_t^\sfa} \qquad(t\geq 0),
\eea
where $g \geq 0, \alpha_j\geq 0, \beta_j \in \bbR^{\sfk_j}$  $(j=1, ..., \sfa)$ are the unknown  parameters, and $X_t^{j}$  $(j=1, ..., \sfa)$  are $\bbR^{\sfk_j}$-valued observed predictable processes. 
More generally, we may say that 
the model (\ref{superpositiongeneral}) is the superposition of counting processes whose intensity {are} given by the log-linear models $\log\lambda_t^{(j)}={\beta_j\cdot X_t^j}$.
{\colr The model (\ref{superpositiongeneral}) can be a generalization of  superposed non-homogeneous Poisson processes such as  Mun et al. \cite{mun2013superposed}, Pulcini \cite{pulcini2001modeling} and Xu et al. \cite{xu2018superposed}.  They treat
	\bea
	\lambda_t &=& \alpha_1 e^{\beta_1t} + \alpha_2 e^{-\beta_2t}
	\label{Munlambda}
	\\
	\lambda_t &=& \frac{\beta_1}{\alpha_1}\bigg(\frac{t}{\alpha_1}\bigg)^{\beta_1-1}+\frac{\beta_2}{\alpha_2}\bigg(\frac{t}{\alpha_2}\bigg)^{\beta_2-1} 
	\label{Pulcinilambda}
	\\
	\lambda_t &=& \lambda_0 + \frac{A_1}{\sqrt{2\pi} \omega t}\exp\bigg(\frac{-(\log{\frac{t}{t_1}})^2}{2\omega^2}\bigg) + A_2\exp{\bigg(\frac{t-t_2}{A_3}\bigg)}, 
	\label{Xulamba}
	\eea
	respectively, where $\alpha_1, \alpha_2, \beta_1, \beta_2$ in (\ref{Munlambda}) and (\ref{Pulcinilambda}), and  $\lambda_0, A_1, \omega, t_1, A_2, A_3, t_2$ in (\ref{Xulamba}) denote the unknown parameters.
When treating  the hazard function for mixture model as defined in Zou et al. \cite{zou2016application}, 
 one would divide the right-hand side of (\ref{superpositiongeneral})    by the survival function.} 
{\colr For a variety of parametric regression models  including parametric proportional hazard models without superposition, see Kalbfleisch and Prentice \cite{kalbfleisch2011statistical} and Lawless \cite{lawless2011statistical}.}

Just to convey the idea of our penalized estimation {\colr for the non-identifiable model (\ref{superpositiongeneral})}, we here consider the simplest model as
\beas
\lambda_t \yeq g + \alpha e^{\beta X_t} \qquad(t\geq 0),
\eeas
where $(g, \alpha, \beta) \in [0, M]^3 =:\Xi$ denotes the unknown  parameter, $M>0$ is some constant,  and $X_t$ is an observed one-dimensional predictable process.  
Let the true intensity $\lambda^*$ be a constant $g^*$.
Then  the true values set $\Xi^*$ is determined as $\{(g, \alpha,\beta) \in \Xi ;  (g, \alpha)=(g^*, 0)~~or~~(g+\alpha, \beta)=(g^*, 0)\}$. 
Let us consider $\beta$ as a nuisance, and therefore denote $(g, \alpha)$ and $\beta $ by $\theta$ and $\tau$, respectively.
Then the projection $\Theta^*$  is determined as  $\Theta^*=\{\theta=(g, \alpha) \in [0, M]^2; g+\alpha=g^*\}$, and $\theta^*: =(g^*, 0)$ is parsimonious in the sense of  (\ref{parsimonious}).   
With  some  non-convex Bridge-type penalty that is  uniquely   minimized on $\Theta^*$ by   $\theta^*$, we obtain an estimator which converges to $\theta^*$ in probability, and can derive its limit distribution and selection consistency as well as {its oracle property}. 
}

{There are still many other applications of our penalized estimation such as parameter estimation of mixture distributions. In such examples, however, the assumptions in this paper may not seem easy to verify because of the complexity of their parametric structure. Therefore, we consider to provide much more simple assumptions,  using Hironaka's theorem, resolution of singularities \cite{atiyah1970resolution, hironaka1964resolution}, and to apply our penalized estimation comprehensively to such complex examples  including mixture distributions although those are to be discussed in the subsequent paper due to page limit. For these reasons,    the two examples (\ref{introadditive}) and (\ref{superpositiongeneral}) whose parametric structure has   a relatively simple explicit form  are treated in this article  without using Hironaka's theorem.  The algebraic approach with resolution of singularities to non-identifiable models  has been studied by Watanabe \cite{watanabe2009algebraic} in the context of Bayesian estimation rather than penalized estimation. }

{The main advantage of this penalized estimation in non-identifiable models  is that even if  the true value sets $\Xi^*$ is extremely complex,  it possesses the oracle property under suitable tuning parameters and therefore the asymptotic   distribution of the estimator depends only on $\theta^*$ but on the nuisance $\tau \in \calt$. 
In non-identifiable models, (quasi-)likelihood ratio tests are sometimes labor intensive to perform because of the complexity of their asymptotic behavior. 
{\colr For a review of the asymptotics of the likelihood ratio in non-regular models, see  Brazzale and Mameli \cite{brazzale2022likelihood}. As mentioned in the review,} 
Liu and Shao \cite{liu2003asymptotics} prove that when identifiability fails under the null hypothesis, 
the asymptotic null distribution of the likelihood ratio converges to the distribution of the supremum of  a squared left-truncated centered Gaussian process. Here the derivation of the set over which the supremum is to be taken requires  working out on a case by case basis, and 
simulation for approximation of the asymptotic null distribution may also be needed.
Given those tasks in testing,  penalized estimation  could be  a more handy way to estimate  the true model. 
Also, even if we perform tests, testing   with a small number of models pre-estimated by  penalized estimation with different tuning parameters might be easier than testing with a large number of  less well-founded models. 

The article is organized as follows. In Section 2,  we recall  the theorem described in Yoshida and Yoshida \cite{yoshida2022quasi}. The main result for the PQMLE is stated in Section \ref{PQMLE},  and  some sufficient condition  for its selection consistency is given in Section \ref{selection}.
(Their proofs is given in Section \ref{proofs}.)   Applications are depicted in Sections \ref{countingprocessgab} and \ref{pointprocesswithmulticolinear}.  (Their proofs are given in Section \ref{proof of points}.) 
}

{\colorr
\section{Preliminary}
We recall Theorem 2.1 in Yoshida and Yoshida \cite{yoshida2022quasi}.
Denote by $\Xi = \Theta \times \mathcal{T}$ the unknown parameter space, 
where $\Theta$ is a measurable subset of  ${ \mathbb{R}^\sfp}$, and $ \mathcal{T}$ is a measurable subset of $\mathbb{R}^{\sfq}$. 
We estimate the true value {(or one of the true values)} $\theta^*$ of the unknown parameter $\theta\in\Theta$, while $\tau\in \calt$ is treated as a nuisance parameter.  
Given a probability space $(\Omega, \mathcal{F}, P)$ specifying the distribution of the data, the statistical inference will be carried out based on a continuous random field $\mathbb{H}_T : \Omega \times \Xi \rightarrow \mathbb{R}$ for 
$T\in\bbT\subset\bbR$ satisfying $\sup\bbT=\infty$,  where a continuous random field means that for each $\omega\in \Omega$, $\bbH_T(\omega)$ is continuous on $\Xi$.
Examples of $\bbH_T$ are (quasi-)log likelihood functions and penalized (quasi-)log likelihood functions.

For each $T\in\bbT$, take  an arbitrary $\calt$-valued random variable $\hat\tau_T$, and suppose that   we can take  a $\Theta$-valued random variable $\hat\theta_T$  that   asymptotically maximizes  $\bbH_T(\theta, \hat\tau_T)$ on $\Theta$.  That is, $\hat \theta_T$ satisfies
\beas
P\bigg[\bbH_T(\hat \theta_T, \hat\tau_T) = \sup_{\theta \in \Theta}\bbH_T(\theta, \hat\tau_T)\bigg] \to 1 \quad(T \to \infty).
\eeas
A common example of $(\hat\theta_T, \hat\tau_T)$ is the joint maximizer. Note that all the assumptions [{\bf A1}]-[{\bf A5}] described below do not depend on   $\hat\tau_T$ nor  on how $\hat\theta_T$ is taken and that the result below shows that the asymptotic behavior of $\hat{\theta}_T$ is the same for different sequences $\hat{\tau}_T$.  



Define $U_T$ by 
\begin{eqnarray}\label{U_T}
U_T \yeq a_T^{-1}\big(\Theta -\theta^*\big)\yeq \{u\in \mathbb{R}^\sfp ; \theta^* + a_T u \in \Theta\} 
\end{eqnarray}
for a deterministic sequence $a_T$ in $GL(\sfp)$ with $\lim_{T\to\infty}|a_T| \rightarrow 0$, where  for any real matrix $A$,    denote by $|A|$ the norm $\big\{{\rm Tr}(AA^\prime)\}^{\frac{1}{2}}$.
We mimic the local asymptotic theory to define the random field $\mathbb{Z}_T : \Omega \times U_T \times\calt  \rightarrow \mathbb{R} $ by 
\begin{eqnarray}\label{Z_T}
\mathbb{Z}_T(u, \tau) &=& 
\exp\bigg(\mathbb{H}_T(\theta^* + a_Tu, \tau) - \mathbb{H}_T(\theta^*, \tau)\bigg)
\qquad
(u\in U_T).
\end{eqnarray} 
Let $C(\bbR^\sfp)$ denote the set  of all  continuous functions  defined on $\bbR^\sfp$. We give it the metric topology induced by a metric $d_\infty$ defined as
\beas
d_\infty(f, g)=\sum_{n=1}^\infty2^{-n}\bigg(1\wedge\max_{|x|\leq n}|f(x)-g(x)|\bigg)\qquad\big(f,g \in C(\bbR^\sfp)\big).
\eeas
For each $T\in\calt$, let $\mathbb{V}_T$ be a $C(\mathbb{R}^\sfp)$-valued random variable.
Also, let $\mathbb{Z}$ be a $C(\mathbb{R}^\sfp)$-valued random variable which will be considered as the limit of $\mathbb{Z}_T$. Let $\calg$ be a sub-$\sigma$-algebra of $\calf$.

For a topological space $S$,  a sequence of  $S$-valued random variables $Y_T$ $(T \in \bbT)$ and  a 
$S$-valued random variable $Y$, we say that $Y_T$ converges stably with limit $Y$ and write as $Y_T {\to} ^{d_s(\calg)}Y$ if and only if for any  bounded continuous function  $f$ defined on $S$ and for any bounded $\calg$-measurable random variable $Z$,
$
E\big[f(Y_T)Z\big] \to  E\big[f(Y)Z\big]
$ as $T\to\infty$. 
Also, for $A\subset \mathbb{R}^\sfp$, $\delta >0$, and $R>0$, we define  $A^{\delta}$ and $A(R)$ as
\begin{eqnarray*}  
	A^{\delta} = \big\{ x\in\mathbb{R}^\sfp ; \inf_{a \in A}|x-a| < \delta \big\} 
	&\text{and}&
	\ A(R)=A\cap{\overline{B_R}}, 
\end{eqnarray*}
respectively, where $B_R = \{x \in \mathbb{R}^\sfp ; |x|< R\}$ and  $\overline {B}$ denotes the closure of $B$ for a subset $B$.
Then we  define $U \subset \mathbb{R}^\sfp$ by \begin{eqnarray}\label{U}
U = \bigcap_{\delta >0} \bigcup_{N=1}^{\infty} \bigcap_{T\geq N} 	{U_T}^{\delta} .\end{eqnarray}
%
Consider the following conditions.
\bd
\im[{\bf[A1]}]
$\ds \varlimsup_{R \rightarrow \infty} \varlimsup_{T\rightarrow \infty}
P\bigg[\sup_{U_T \times \mathcal{T},  |u| \geq R} \mathbb{Z}_T(u, \tau)  \geq 1\bigg] \yeq0$.
\im[{\bf[A2]}] For every $R>0$, as $T\to\infty$,
\begin{eqnarray}
\sup_{U_T(R) \times \mathcal{T}}|\mathbb{Z}_T(u, \tau) - \mathbb{V}_T(u)| \overset{P}{\rightarrow} 0, \nonumber\\ 
\mathbb{V}_{T} \overset{d_s(\calg)}{\rightarrow} \mathbb{Z}
\qquad \text{in } C(\overline{B_R}) .\label{2110240341}
\end{eqnarray}
\ed
More precisely, the convergence (\ref{2110240341}) means 
$\bbV_T|_{\sf C}\overset{d_s(\calg)}{\to}\bbZ|_{\sf C}$ for ${\sf C}=C(\overline{B_R} )$. \halflineskip

\bd
\im[{\bf[A3]}] $\ds U \supset \bigcap_{N=1}^{\infty} \overline{\bigcup_{T\geq N} U_T} $.

\im[{\bf[A4]}] 
There exists a $U$-valued random variable $\hat{u}$ such that with probability $1$,  
\begin{eqnarray*}{\mathbb{Z}(\hat{u})} = \sup_{U }{\mathbb{Z}(u)}
\end{eqnarray*}
and such that with probability 1, for all $u \in U$ with $u \neq \hat{u}$,
\begin{eqnarray*}  {\mathbb{Z}(u)} < {\mathbb{Z}(\hat{u})}.
\end{eqnarray*}

\im[{\bf[A5]}] There exist some $ T_0 \in \bbT$ and a sequence of  $U$-valued random variables $\{\hat{v}_T\}_{T\geq T_0, T\in \bbT}$ such that with probability $1$,  
\begin{eqnarray*}{\mathbb{V}_T(\hat{v}_T)} = \sup_{U}{\mathbb{V}_T(u)}
\end{eqnarray*}
and such that $\{\hat v_T\}_{T\geq T_0, T\in\bbT}$ is tight.

\ed
Define $\hat{u}_T$ by $\hat{u}_T= a_T^{-1}(\hat{\theta}_T - \theta^*)$. 

\begin{theorem}[Yoshida and Yoshida \cite{yoshida2022quasi}] \label{thm1} 
	Under {\rm [{\bf A1}]-[{\bf A4}]},
	\begin{eqnarray}\label{hatu}
	\hat{u}_T \overset{d_s(\calg)}{\rightarrow} \hat{u}.
	\end{eqnarray} 
	Moreover, if {\rm [{\bf A5}]} also holds, then
	\begin{eqnarray}
	\hat{u}_T -\hat{v}_T=o_P(1).\label{hatu-hatv}
	\end{eqnarray} 
\end{theorem} 

 \begin{remark}\label{remarkHT-infty}
	{\rm \bd
		\im[(i)] Yoshida and Yoshida \cite{yoshida2022quasi} shows that $U$ is a generalization of the cone sets introduced by Chernoff \cite{chernoff1954distribution} and Andrews \cite{andrews1999estimation}. Some sufficient conditions for [{\bf A3}] and a way to derive the explicit form of $U$ is discussed in \cite{yoshida2022quasi}. 
	\im[(ii)] Consider a situation where     $\bbH_T(\theta, \tau) = - \infty$  for some $(\omega, \theta, \tau) \in \Omega \times \Xi$ while $\bbH_T(\theta^*, \tau) \in \bbR$ for any $(\omega, \tau) \in \Omega \times\calt$. Even in this case, if  $\bbZ_T$ defined as (\ref{Z_T}) is still continuous for each $\omega \in \Omega$ by interpreting $\exp(-\infty)$ as $0$,   then Theorem \ref{thm1} still holds since its proof  depends on the properties of   $\bbZ_T$ including its continuity but not on those of $\bbH_T$. 
	
\ed}
	\end{remark}}

}

	\section{Penalized quasi-maximum likelihood estimation  under
		\\ non-identifiability}
	\subsection{Settings}\label{settings}
	Denote by $\Theta\subset \bbR^\sfp$ the unknown parameter space. As explained in Introduction, in this article, we treat  models where identifiability may fail.
	Thus,  the true value of $\theta \in \Theta$ is not necessarily uniquely determined.  We  denote by $\Theta^*$ the set consisting of  all the true values of $\theta \in \Theta$. 
	Suppose that some $\theta^*\in \Theta^*$ is sufficiently parsimonious in the sense of (\ref{parsimonious}).
		Then  we estimate  $\theta^*$ based on  the model  $\Theta$ by adding    a suitable penalty term whose minimizer is  $\theta^*$ on $\Theta^*$.  (These assumptions including the definitions of $\Theta^*$ 
		are more precisely described  in Section \ref{PQMLE}.)  
	More generally, in what follows,  we also consider a nuisance parameter $\tau$ whose parameter space is denoted by $\calt$.    Then $(\theta, \tau) \in \Xi :=\Theta \times \calt$ are the unknown parameters.   Note that when actually estimating, we do not need  to know which of the unknown parameters corresponds to $\theta$ and which to $\tau$.
		
	Suppose that   $\Theta$ and $\calt$ are compact in $\bbR^\sfp$ and $\bbR^\sfq$, respectively.
	Let $(\Omega, \mathcal{F}, P)$  be a given probability space specifying the distribution of the data.
	Let $\bbT$ be a subset of $\bbR$ satisfying $\sup\bbT=\infty$. We consider the  penalty term
	\beas
	\sum_{j\in \calj}\xi_{j, T} p_j(\theta_j)+s_T(\tau) \qquad\big((\theta_1, ..., \theta_\sfp) \in \Theta, \tau \in \calt\big),
	\eeas 
	where  $s_T : \Omega\times \bbR^\sfq \to [0, \infty)$ is a continuous random field depending on $T \in \bbT$,  $\calj$ is a subset of $\{1, ..., \sfp\}$, $\xi_{j, T} : \Omega \to [0, \infty)$ $(j \in \calj)$ are random variables depending on $T\in \bbT$,
	and $p_j : \bbR \to [0, \infty)$ $(j \in \calj)$ are deterministic continuous functions satisfying the following conditions:
	\bd
	\im[(i)]  $p_{j}\big|_{\bbR\setminus\{0\}}$ is of class $C^1$
	
	\im[(ii)] For any  $x \in \bbR$, $p_{j}(x)=0$ if and only if $x=0$.

	\im[(iii)] There exist some  positive constants $\delta_0$ and $q_j\leq 1$   such that
	\beas
	p_{j}(x) =|x|^{q_j} \qquad\big(x \in [-\delta_0, \delta_0]\big).
	\eeas
	
	\ed
	We consider penalized quasi-maximum likelihood estimation based on the random field $\bbH_T$ given by 
	\begin{eqnarray*}
		\bbH_T(\theta, \tau)=\calh_T(\theta, \tau)-\sum_{j\in \calj}\xi_{j, T} p_j(\theta_j)-s_T(\tau) \qquad\big((\theta, \tau)\in \Xi\big),
	\end{eqnarray*}
	where 
	{for each $T\in\bbT$,  $\calh_T: \Omega\times \Xi\rightarrow \bbR \cup \{- \infty\}$ is  some  random field satisfying that 
  $\calh_T(\omega, \theta^*, \tau)$ is $\bbR$-valued for every $(\omega, \tau) \in \Omega \times \calt$ and that
$\exp\big\{\calh_T(\theta, \tau)- \calh_T(\theta^*, \tau)\big\}$   is    a continuous random field  by interpreting $\exp(-\infty)$ as $0$.
 Note that  $\bbZ_T$ defined as (\ref{Z_T}) becomes a continuous random field. } 
	
	 For each $T \in \bbT$, let $\hat \tau_T$ be an arbitrary $\calt$-valued random variable and $\hat \theta_T$ a $\Theta$-valued random variable that asymptotically maximizes $\bbH_T(\cdot, \hat\tau_T)$ on $\Theta$.  
	Under some conditions in Section \ref{PQMLE}, we derive the asymptotic behavior of the estimator $\hat\theta_T$. 
	As already mentioned,
	the results below show that  the asymptotic behavior of $\hat{\theta}_T$ is the same for different sequences $\hat{\tau}_T$. Indeed, 
	  [{\bf H1}]-[{\bf H5}] do not depend on  $\hat\tau_T$ nor  on the choice of $\hat\theta_T$.

	\subsection{Asymptotic behavior of the PQMLE}\label{PQMLE}
	{ To derive the asymptotic distribution of $\hat \theta_T$, we precisely describe  the assumptions, and  apply Theorem \ref{thm1}.}
	Let $\sfr$ be a positive integer, and  let $h : \Theta \times \calt\to \bbR^{\sfr}$ be a   
	 {continuous} map. 
	Define a set $\Xi^* \subset \Xi$ as 
	\beas
	\Xi^*&=&\big\{(\theta, \tau) \in \Xi ;h(\theta, \tau)=0\big\}.
	\eeas
	Then we give the precise definition of $\Theta^*$  as
	\beas
	\Theta^*&=&\big\{\theta \in \Theta ;   (\theta,  \tau) \in \Xi^* \text{ for some $\tau \in \calt$}\big\}. \eeas 
	
	As in Section \ref{settings}, let $\theta^*$ be an element of $\Theta^*$  which will be specified in [{\bf H2}] as a sufficiently parsimonious value.
	For each $j=1, ..., \sfp$, we denote by $\theta^*_j$ the $j$-th component of $\theta^*$. We decompose $(\theta^*_j)_{j \in \calj}$ as 
	\begin{eqnarray*}
		\theta^*_{i}\neq 0    ~~and~~\theta^*_{k}=0\qquad(i\in \calj_1, k\in\calj_0),
	\end{eqnarray*}
	where $\calj_1$ and $\calj_0$ are disjoint subsets of $\calj$ satisfying $\calj=\calj_1 \cup \calj_0$. 
	For notational simplicity, we assume that 
	\beas
	\calj_0=\{\sfp_1+1, ..., \sfp\},
	\eeas where $\sfp_1$ is some nonnegative integer. For $\theta=(\theta_1, ..., \theta_{\sfp}) \in \bbR^{\sfp}$, we denote $(\theta_k)_{k \in \{1, ..., \sfp\}\setminus \calj_0}=(\theta_{1}, ..., \theta_{\sfp_1})$ and $(\theta_k)_{k \in \calj_0}=(\theta_{\sfp_1+1}, ..., \theta_{\sfp})$ by $\overline\theta$ and $\underline\theta$, respectively. 
	For  each $j=1, ..., \sfr$, 
	let $\frak a_{j, T}>0$  be a deterministic number depending on  $T \in \bbT$ satisfying that $\frak a_{j, T} \to 0$. Define an $\sfr \times \sfr$ matrix $\frak a_T$ as  $\frak a_T={\rm diag}(\frak a_{1, T}, ..., \frak a_{\sfr, T})$.\footnote{\colr For any $a_1, ..., a_n \in \bbR$, ${\rm diag}(a_1, ..., a_n)$  denotes an $n \times n$  diagonal matrix whose $(i, i)$ entry is $a_i$ for every $i=1, ..., n$.}  Let  $\frak b_T>0$ be some deterministic sequence satisfying  that  \bea\label{bT}
	C_0^{-1}\max_{j=1, ..., \sfr}(\frak a_{j, T}^{-2}) ~\leq~  \frak b_T ~\leq~ C_0\min_{j=1, ..., \sfr}(\frak a_{j, T}^{-2})\qquad(T\in \bbT)
	\eea  for some constant $C_0\geq 1$. That is, $\frak a_{j, T}$ $(j=1, ..., \sfr)$   have the same rate  of convergence as $\frak b_T^{-\frac{1}{2}}$.

	\begin{en-text}Let $\{I\}_{I \in \cali}$ be partitions of $\{1, ..., \sfr\}$ satisfying that for each $I$,  
	\beas
	\max_{j \in I}\frak a_{j, T} \leq  C\min_{j \in I}\frak a_{j, T}\qquad(T\in \bbT)
	\eeas  for some constant $C>0$. That is, $\frak a_{j, T}$  $(j \in I)$  have the same rate of convergence.
	\end{en-text}
	We consider the following conditions. 
	{\colr Here the $r$ times tensor product of a vector $v$ is denoted by $v^{\otimes r}$. For a tensor $T=(T_{i_1, ..., i_k})_{i_1, ..., i_k}$ and vectors $v_1=(v_1^{i_1})_{i_1}, ..., (v_k^{i_k})_{i_k}$, we write
		$
		T[v_1 ,..., v_k]\yeq T[v_1 \otimes \cdots \otimes v_k]\yeq \sum_{i_1, ..., i_k}T_{i_1, ..., i_k}v_1^{i_1}\cdots v_{k}^{i_k}.
		$ }
	\bd

	\im[${\bf [H1]}$]  $\calh_T$ satisfies the inequality: for any $(\theta, \tau) \in \Xi$ and any $T \in \bbT$,
	\bea\label{X1}
	\calh_T(\theta, \tau)-\calh_T(\theta^*, \tau)\leq  K_T(\theta, \tau)[\frak a_T^{-1}h(\theta, \tau)]-\frac{1}{2}\big\{G(\theta, \tau)+r_T(\theta, \tau)\big\}[\big(\frak a_T^{-1} h(\theta, \tau)\big)^{\otimes 2}],
	\eea
	where  $K_T$ is a $C(\Xi ;\bbR^{\sfr})$-valued random variable depending on $T \in \bbT$ satisfying that 
	\beas
	\sup_{(\theta, \tau) \in \Xi}|K_T(\theta, \tau)|=O_P(1),
	\eeas
	$G$ is a $C(\Xi ; \bbR^{\sfr}\otimes \bbR^{\sfr})$-valued random variable satisfying that 
	with probability $1$, for any $(\theta, \tau) \in \Xi$, 
	\beas
	G(\theta, \tau) \text{ is positive definite, }
	\eeas
	and $r_T(\theta, \tau)$ is a $C(\Xi ; \bbR^{\sfr}\otimes \bbR^{\sfr})$-valued random variable depending on $T \in \bbT$ satisfying that 
	\beas \sup_{(\theta, \tau)\in \Xi}|r_T(\theta, \tau)|\overset{P}{\to}0.
	\eeas
	
 \im[${\bf [H2]}$] $\theta^*$ satisfies the parsimonious condition
	$\Theta^* \cap\{\theta \in \Theta; \underline\theta=0\}=\{\theta^*\}$. {$\big($Especially, $\theta^*$ satisfies (\ref{parsimonious}).$\big)$}
	{Moreover,  
		 there exists an open ball  $\calw$  in $\bbR^\sfp$ with  center
	at $\theta^*$ such that  
	\bd
	\im[(a)] $h$ can be  extended on $\calw \times \calt$ and for some constant $C>0$,
	\beas
	\big|h(\theta, \tau)-h(\ol\theta, 0, \tau)\big| \leq C|\underline \theta|\qquad\big(\theta \in \calw \cap \Theta, \tau \in  \calt\big),
	\eeas 
	\im[(b)] for some constant $\epsilon_0>0$,  
	\beas
	|h(\overline\theta, 0, \tau)| \geq \epsilon_0 |\ol\theta -\ol\theta^*| \qquad\big(\theta \in \calw \cap \Theta, \tau \in  \calt\big),
	\eeas    
	where $\overline\theta^*=(\theta^*_1, ..., \theta^*_{\sfp_1})$.
	\ed
}
\begin{en-text}
	\beas
	|h(\overline\theta, 0, \tau)| \geq \epsilon_0 |\ol\theta -\ol\theta^*| \qquad\big(\theta \in \calw \cap \Theta, \tau \in  \calt, \big),
	\eeas    
	where $\overline\theta^*=(\theta^*_1, ..., \theta^*_{\sfp_1})$.
\end{en-text}

	\im[${\bf [H3]}$]  For any $r>0$, 
	\beas
	\varlimsup_{\epsilon\to +0}\varlimsup_{T\to\infty}P\bigg[ {E_T(\theta^*)} \geq (1-\epsilon)\inf_{\theta^\dagger \in \Theta^*, |\theta^\dagger-\theta^*| \geq r}{E_T(\theta^\dagger)}  \bigg] =0,
	\eeas
	where $E_T : \Omega \times \bbR^\sfp \to [0, \infty)$ is defined as 
	\beas
	E_T(\omega, \theta)&=&\sum_{j \in \calj}\xi_{j, T}(\omega) p_j(\theta_j) \qquad\big(\omega \in \Omega, \theta=(\theta_1, ..., \theta_\sfp) \in \bbR^\sfp\big).
	\eeas

	\im[${\bf [H4]}$]
 For any $k \in \calj_0$ and any $i \in \calj_1$,
 \beas
 \xi_{i, T} \xi_{k, T}^{-1} =\begin{cases}
 	O_P(1) & (q_k<1)
 	\\ o_P(1)
 	&  (q_k=1) 
 \end{cases}.
 \eeas 
  \big(If $\xi_{k, T}=0$, then we interpret $\xi_{i, T}\xi_{k, T}^{-1}$   as $\infty$.\big)
 \begin{en-text}{\colr there exists a neighborhood $\calw^\prime$ of $\theta^*$ in $\bbR^\sfp$ such that for each $k \in \calj_0$, \beas
 \max_{i \in \calj_1}(\frak a_{i, T}\xi_{i, T})\sup_{\substack{\theta \in \calw^\prime \cap \Theta,\\ \tau \in \calt}}\frac{\big|\frak a_T^{-1}\big\{h(\theta, \tau)-h(\ol \theta, 0, \tau)\big\}\big|}{\big|(\xi_{k, T}\theta_k)_{k \in \calj_0}\big|}
\yeq \begin{cases}
	O_P(1) & (q_k<1) \\ o_P(1) &  (q_k=1)
	\end{cases}.
	\eeas }
\end{en-text}

	\ed
	
	\vspace{4mm}
	Let $\caln$ be a bounded open  set in $\bbR^{\sfp}$ satisfying
	{\bd  
		\im[(i)] for some $\delta>0$,  \bea\label{caln1}
		\Theta\cap\{|\theta-\theta^*|<\delta\} \subset \overline\caln, \eea
		\im[(ii)] for any $\theta \in \caln$ and any $0 < t \leq 1$,  
		\bea\label{caln2} t\theta +(1-t)\theta^* \in \caln.\eea
	\ed }
	\def\ol{\overline} 
	\noindent We suppose that $\calh_T$ {is $\bbR$-valued on $\Omega\times(\ol \caln \cap\Theta) \times \calt$ and } can be extended to an $\bbR$-valued continuous random field defined on $\Omega\times \ol\caln\times\calt$  satisfying that for every $\omega\in\Omega$ and $\tau\in\calt$, $\calh_T(\omega, \cdot, \tau)$ is of class $C^2(\ol\caln)$\footnote{For an open set $G$, $C^k\big(\ol G\big)$ denotes the set of all functions which are of class $C^k$ in $G$ and whose derivatives can be continuously extended on $\ol G$.}. Let $\overline{\Delta}({\theta^*})$ be an $\bbR^{\sfp_1}$-valued random variable, and  $\overline \Gamma(\theta^*)$  a $\calg$-measurable  $ \bbR^{\sfp_1}\otimes\bbR^{\sfp_1}$-valued random variable, where $\calg \subset \calf$ is a sub-$\sigma$-field.  Take $a_T \in GL(\sfp)$ as a deterministic diagonal matrix  defined by  \bea\label{aT}
	(a_T)_{jj}=\left\{\begin{array}{ll}   \frak a_{j, T} &\big(j \in \{1, ..., \sfp\}\setminus\calj_0\big) \\
	\frak b_{T}^{-\frac{\rho_j}{2}}&(j \in \calj_0)\end{array}\right., \label{a_T}
	\eea 
where $\rho_k>1$ $(k \in \calj_0)$ are  deterministic constants. From   [{\bf H5}] described below, $\frak b_{T}^{-\frac{\rho_k}{2}}$ gives  the  rate of convergence of $\xi_{k, T}^{-\frac{1}{q_k}}$ for each $k \in \calj_0$. 
	Besides, let $c_i : \Omega \to [0, \infty)$ $(i \in \calj_1)$ be nonnegative random variables, and $d_k : \Omega \to (0, \infty)$ $(k \in \calj_0)$  positive random variables.
	Then we also consider the following condition. 

	\bd
	
	\im[${\bf [H5]}$] For some $\tau_0 \in \calt$,
	\beas
	\sup_{\tau \in \calt}\big|a_T\partial_\theta\calh_T(\theta^*, \tau)-a_T\partial_\theta\calh_T(\theta^*, \tau_0)\big|\overset{P}{\to} 0,
	\eeas
	\begin{eqnarray*}
		&&\big(a_T\partial_\theta\calh_T(\theta^*, \tau_0), ~  (\frak a_{i, T}\xi_{i, T})_{i \in \calj_1}, ~{(\frak b_T^{-\frac{q_k\rho_k}{2} }\xi_{k, T})_{k \in \calj_0} \big)} \overset{d_s(\calg)}{\rightarrow}\bigg(\big(\overline{\Delta}({\theta^*}), 0\big), (c_i)_{i \in \calj_1}, (d_k)_{k \in \calj_0} \bigg),
		\eeas  
		in $\bbR^\sfp \times \bbR^{|\calj_1|} \times \bbR^{|\calj_0|}$.
		Also, for any $R>0$,
		\beas
		&&\sup_{\substack{(\theta, \tau)\in \ol \caln\times\calt\\ |a_T^{-1}(\theta-\theta^*)|\leq R}}\bigg|a_T\partial^2_\theta\calh_T(\theta, \tau)a_T+ \begin{pmatrix}
			\overline \Gamma(\theta^*) & O\\
			O & O
		\end{pmatrix} \bigg|\overset{P}{\rightarrow}0.
	\end{eqnarray*}
	\ed
	  Define $U_T$ and $U$ as (\ref{U_T}) and (\ref{U}), respectively.
	We also define $\hat{u}_T$ by $\hat{u}_T= (a_T)^{-1}(\hat{\theta}_T - \theta^*)$.
	In the following, for any $u=(u_1, ..., u_{\sfp})\in \bbR^\sfp$, we denote $(u_1, ..., u_{\sfp_1})$ by $\overline u$. Define $\bbZ$ and $\bbV_T$ for any $u \in \bbR^\sfp$ by
	\bea
	\bbZ(u)&=&\exp\bigg\{\overline\Delta(\theta^*)[\overline u]-\frac{1}{2}\overline \Gamma(\theta^*)\big[{\overline u}^{\otimes2}\big]-\sum_{i \in \calj_1}c_i\frac{d}{dx}p_i (\theta^*_i)u_i-\sum_{k \in \calj_0}d_k|u_k|^{q_k}\bigg\}\label{z}
	\eea
	and
	\bea
	\bbV_T(u)&=&\exp\bigg\{\overline \Delta_T(\theta^*, \hat\tau_T)[\overline u]-\frac{1}{2}\overline \Gamma(\theta^*)\big[{\overline u}^{\otimes2}\big]-\sum_{i \in \calj_1}\frak a_{i, T}\xi_{i, T}\frac{d}{dx}p_i (\theta^*_i)u_i\nonumber\\
	&&-\sum_{k \in \calj_0}{\frak b_{T}^{-\frac{q_k\rho_k}{2}}}\xi_{k, T} |u_k|^{q_k}\bigg\},\label{calv_T}
	\eea
	respectively,  where
	\beas
	\overline\Delta_T(\theta^*, \cdot) \yeq \bigg(\frak a_{1, T}\frac{\partial}{\partial\theta_1}\calh_T(\theta^*, \cdot), ..., \frak a_{\sfp_1, T}\frac{\partial}{\partial \theta_{\sfp_1}}\calh_T(\theta^*, \cdot)\bigg).
	\eeas
	 The following theorem constitutes the main result. Its proof is given in Section \ref{proofs}.
	 
	 \begin{theorem}\label{thmY2}
	 	Assume {\rm [{\bf H1}]-[{\bf H5}]} and  $[{\bf A3}]$.  Also, assume $[{\bf A4}]$ for $\bbZ$ defined as $(\ref{z})$. 
	 	Then 
	 	\beas
	 	\hat{u}_T= (a_T)^{-1}(\hat{\theta}_T - \theta^*)\overset{d_s(\calg)}{\to} \hat u.
	 	\eeas
	 	Moreover, assume $[{\bf A5}]$ for $\bbV_T$ defined as $(\ref{calv_T})$.
	 	Then 
	 	\beas
	 	\hat u_T-\hat v_T =o_P(1).
	 	\eeas
 \end{theorem}

	\begin{remark}\label{remid}{\rm
			\bd

			\im[(i)]    Yoshida and  Yoshida \cite{yoshida2022quasi} assumed that    $q_j>0$, $\rho_k \geq 1$ and  {$d_k\geq 0$ $(j \in \calj, k \in \calj_0)$}.  
				However, in this article, we assume $0<q_j\leq 1$, $\rho_k>1$ and  $d_k > 0$ $(j \in \calj, k \in \calj_0)$. 
				It is because our estimator need to shrink its  $\calj_0$-components to $0$ by stronger penalization, 
				which leads to the explicit form of $\bbZ$  as (\ref{z}) where  the nuisance parameter $\tau$ does not appear.

			\im[(ii)] Condition [{\bf H1}] reflects that for each $\tau \in \calt$, by maximizing $\calh_T(\cdot, \tau)$, we can only identify $\theta \in \Theta$ satisfying  $h(\theta, \tau)= 0$. This implies that identifiability may break. 
			
			\im[(iii)] The first sentence of [{\bf H2}]  means that on $ \{\theta \in \Theta ; \underline\theta=0\}  \times \calt$, $h(\theta, \tau)= 0$ implies $\theta=\theta^*$.  
			In particular, under [{\bf H1}] and the first condition of [{\bf H2}], if we restrict the parameter space $\Theta$ to $ \{\theta \in \Theta ; \underline\theta=0\}$, then we can  identify  $\theta^*$ using $\calh_T$. This comes from the fact that  the model becomes identifiable  for $\theta$ if $\underline\theta=0$. 
		Near $\theta^*$, [{\bf H2}]-{\bf(a)}  imposes on $h$ some kind of Lipschitz continuity  with respect to $\underline\theta$, and {\bf (b)} requires a specific separation of the identifiable  model $\{\theta \in \Theta; \underline\theta=0\}\times\calt$ for $\ol\theta$.
			
			\im[(iv)] Roughly speaking, [{\bf H3}] reflects that $\theta^*$ is  the unique minimizer of the penalty term on $\Theta^*$. To make it easy to understand, assume that  there exists some positive random variable $\lambda_T$ depending on $T \in \bbT$ with $\lambda_T \overset{P}{\to} \infty $ such that for any $j \in \calj$,
			\bea\label{commonrate}
			\xi_{j, T} =\kappa_j \lambda_T,
			\eea where $\kappa_j$ $(j \in \calj)$ are some nonnegative random variables. Then the assumption [{\bf H3}] holds if and only if the following condition $[{\bf H3}]^{\prime}$ holds:
			\bd
			
			\im[{$\bf [H3]^{\prime}$}] With probability $1$, $\theta^*$ is the unique minimizer of $\sum_{j \in \calj}\kappa_jp_j(\theta_j)$ on $\Theta^*$.
			
			\ed
			Therefore, for [{\bf H3}], we only need to adjust the values of $\kappa_j$ $(j \in \calj)$ such that $\theta^*$ uniquely minimizes the penalty term on $\Theta^*$.
			
			{ \im[(v)] 
				 Assume that $q_k<1$ $(k \in \calj_0)$ and $\frak a_{j, T}=\frak b_{T}^{-\frac{1}{2}}$ $(j \in \calj)$. 
				An example of suitable $\{\xi_{j, T}\}_{j \in \calj}$ is given by (\ref{commonrate}) if $\kappa_j\geq 0$ $(j \in \calj)$ are  deterministic numbers with $\kappa_k >0$  $(k \in \calj_0)$ and
				\beas
				\lambda_T =\frak b_T^{\frac{r}{2}}
				\eeas
				for some $r>0$ with $q_k<r\leq 1$ $(k \in \calj_0)$.
				 Indeed, in this case, [{\bf H4}] obviously holds, and
				 by setting  $\rho_k=\frac{r}{q_k}>1$, we have
			\beas
			\frak a_{i, T}\xi_{i, T} \to \kappa_i1_{\{r=1\}} \quad(i \in \calj_1),\qquad  \frak b_T^{-\frac{q_k\rho_k}{2} }\xi_{k, T} = \kappa_k>0\quad(k \in \calj_0).
			\eeas
			Therefore, part of [{\bf H5}] imposed on $\xi_{j, T}$  also holds for $c_i= \kappa_i1_{\{r=1\}} $ and $d_k =\kappa_k$. Thus, just by tuning $\kappa_j$ for $[{\bf H3}]^\prime$, we obtain   $\{\xi_{j, T}\}_{j \in \calj}$ satisfying all the assumptions. 
			If $q_k=1$ $(k \in \calj_0)$, then  [{\bf H4}] implies the need for the adaptive Lasso.
				
		}\begin{en-text} If $h(\theta, \tau)$ is continuously differentiable with respect to $\underline\theta$, then 
				Condition  [{\bf H4}] holds under the following assumption.
				\bd
				\im[{$\bf [H4]^{\prime}$}] 
				For each $i \in \calj_1$, 
				$\xi_{i, T} \max_{j \in \calj}\frak a_{j, T}^2=o_P(1)$. 
				Also, 
					for each $k \in \calj_0$ and $i \in \calj_1$, $\frak a_{i, T}^{-1} \xi_{k, T}^{-\frac{1}{q_k}} ={\colb O_P(1)}$ and 
					\bea\label{xicompare}
					\xi_{i, T} \xi_{k, T}^{-1} \yeq \begin{cases}
						O_P(1) & (q_k<1) \\ o_P(1) &  (q_k=1)
					\end{cases}.
					\eea
					\ed
					This condition is obtained just by choosing $1$ as $\gamma_{k, i}$ in [{\bf H4}].  In particular, for each $k \in \calj_0$, if  $q_k<1$, i.e., the penalty $p_k$ is non-convex, then (\ref{xicompare}) holds  even when $\xi_{k, T}\approx\xi_{i, T}$ for any $i \in \calj_1$.  On the other hand, if  $q_k=1$, then  (\ref{xicompare}) holds only for $\xi_{k, T}$ that is  asymptotically much greater than $\xi_{i, T}$ for any $i \in \calj_1$, which leads to the need for the adaptive Lasso.
				\end{en-text}
			
			\ed
				
	} \end{remark}

\begin{remark}
	{\rm \def\tupg{\textup{g}}
		\def\tupgs{\big(\textup{g}_i(v)\big)}
		  A way to derive the explicit form of $h(\theta, \tau)$ satisfying [{\bf H1}] and [{\bf H2}] is given at the end of Section \ref{countingprocessgab}.
		If we does not need the explicit form  but only the existence of $h$, then  it  can be verified  by some algebraic results such as  Hironaka's theorem, resolution of singularities. 
		 \begin{en-text} 
		 	Assume that $\frak a_T$ equals some scalar matrix $\frak b_T^{-\frac{1}{2}}I_\sfr$ where $\frak b_T $ is a positive sequence with $\frak b_T \to \infty$. 
		 Let  $V_i\subset \bbR^\sfp$ $(i=1, ..., L)$ be open sets, and let $\tupg_i : V_i \to \bbR^\sfp$ $(i=1, ..., L)$  real analytic maps  satisfying
		    \beas
		    \Xi \subset \bigcup_{i=1}^L\tupg_i(V_i).
		    \eeas 
		\end{en-text} 
		\begin{en-text}
			For this family $\big\{(V_i, \tupg_i)\big\}_{i=1}^L$, consider the following condition instead of [{\bf X1}].
		 \bd 
		\im[${\bf [Y1]}$]  $\calh_T$ satisfies the inequality: for every $i=1, ..., L$ and any $v \in V_i$ and any $T \in \bbT$,
		\beas
		\calh_T\tupgs-\calh_T\big(\tupg(0)\big)\leq  M_{i, T}(v)\frak b_T^{\frac{1}{2}}v^k-\frak b_Tv^{2k},
		\eeas
		where  $M_{i, T}$ is a $C(V_i)$-valued random variable depending on $T \in \bbT$ satisfying that 
	 \beas
		\sup_{v \in V_i}|M_T(v)|=O_P(1).
		\eeas
		\ed
	\end{en-text}
		 This will be shown in the subsequent paper.
}
	\end{remark}

	\begin{en-text}
	\subsubsection*{Conditions  [{\bf A3}]-[{\bf A5}] in \ref{}}
	Define $U_T$ by 
\begin{eqnarray}\label{U_T}
U_T = \{u\in \mathbb{R}^\sfp ; \theta^* + a_T u \in \Theta\} 
\end{eqnarray}
For $A\subset \mathbb{R}^\sfp$ and  $\delta >0$,  $A^{\delta}$ denotes as
\begin{eqnarray*}  
	A^{\delta} = \big\{ x\in\mathbb{R}^\sfp ; \inf_{a \in A}|x-a| < \delta \big\}.
\end{eqnarray*}
Then we  define $U \subset \mathbb{R}^\sfp$ by \begin{eqnarray}\label{U}
U = \bigcap_{\delta >0} \bigcup_{N=1}^{\infty} \bigcap_{T\geq N} 	{U_T}^{\delta} .\end{eqnarray}
In the following, for any $u=(u_1, ..., u_{\sfp})\in \bbR^\sfp$, we denote $(u_1, ..., u_{\sfp_1})$ by $\overline u$. Define  for any $u \in \bbR^\sfp$,
	\begin{eqnarray}
\bbZ(u)&=&\exp\bigg\{\overline\Delta(\theta^*)[\overline u]-\frac{1}{2}\overline \Gamma(\theta^*)\big[{\overline u}^{\otimes2}\big]-\sum_{i \in \calj_1}c_i\frac{d}{dx}p_i (\theta^*_i)u_i-\sum_{k \in \calj_0}d_k|u_k|^{q_k}\bigg\}\label{z},\\
\bbV_T(u)&=&\exp\bigg\{\overline \Delta_T(\theta^*, \hat\tau_T)[\overline u]-\frac{1}{2}\overline \Gamma(\theta^*)\big[{\overline u}^{\otimes2}\big]-\sum_{i \in \calj_1}\frak a_{i, T}\xi_{i, T}\frac{d}{dx}p_i (\theta^*_i)u_i\nonumber\\
&&-\sum_{k \in \calj_0}\frak a_{k, T}^{{q_k\rho_k}}\xi_{k, T} |u_k|^{q_k}\bigg\},\label{calv_T}
\end{eqnarray}
where $\overline\Delta_T(\theta^*, \cdot)$ represents $\big(\frak a_{1, T}\frac{d}{d\theta_1}\calh_T(\theta^*, \cdot), ..., \frak a_{1, T}\frac{d}{d\theta_{\sfp_1}}\calh_T(\theta^*, \cdot)\big)$.
Then we re-state  [{\bf A3}]-[{\bf A5}] in J. Yoshida and N. Yoshida \ref{}. 

\bd
\im[{\bf[A3]}] $\ds U \supset \bigcap_{N=1}^{\infty} \overline{\bigcup_{T\geq N} U_T} $.

\im[{\bf[A4]}] 
There exists a $U$-valued random variable $\hat{u}$ such that with probability $1$,  
\begin{eqnarray*}{\mathbb{Z}(\hat{u})} = \sup_{U }{\mathbb{Z}(u)}
\end{eqnarray*}
and such that with probability 1, for all $u \in U$ with $u \neq \hat{u}$,
\begin{eqnarray*}  {\mathbb{Z}(u)} < {\mathbb{Z}(\hat{u})}.
\end{eqnarray*}

\im[{\bf[A5]}] There exist some $ T_0 \in \bbT$ and a sequence of  $U$-valued random variables $\{\hat{v}_T\}_{T\geq T_0, T\in \bbT}$ such that with probability $1$,  
\begin{eqnarray*}{\mathbb{V}_T(\hat{v}_T)} = \sup_{U}{\mathbb{V}_T(u)}
\end{eqnarray*}
and such that $\{\hat v_T\}_{T\geq T_0, T\in\bbT}$ is tight.

\ed
\end{en-text}


	\subsection{Selection consistency}\label{selection}
	In the following, assume that $\calj_0$ is not empty. 
	We denote by $\hat u_{j, T}$ the $j$-th component of $\hat u_T$  $(j=1, ..., \sfp)$.  When we can show that  $\hat u_{k, T} \overset{P}{\to} 0$ $(k\in \calj_0)$ using Theorem \ref{thmY2}, we can also show selection consistency under the following condition [{\bf S}]. Note that similarly as before, for $u=(u_1, ..., u_\sfp) \in \bbR^\sfp$ and $v=(v_1, ..., v_\sfp) \in \bbR^\sfp$, we denote $(u_1, ..., u_{\sfp_1})$,   $(u_j)_{j \in \calj_0}$,  $(v_1, ..., v_{\sfp_1})$ and $(v_j)_{j \in \calj_0}$ by $\overline u$, $\underline u$,  $\overline v$ and $\underline v$, respectively.  Note that $a_T$ is a diagonal matrix defined as (\ref{a_T}).
	\bd
	\im[{\bf [S]}] For any $R>0$ and any positive sequence $\delta_T$ with $\delta_T\rightarrow 0$ as $T\rightarrow\infty$, 
	\beas
	\sup_{u\in S(R, \delta_T)}\inf_{v \in I(R)}\frac{|\overline u-\overline v|}{\sum_{k \in \calj_0}|u_k|^{q_k}}\rightarrow 0 \qquad(T\rightarrow \infty),
	\eeas
	where \beas S(R, \delta_T)&=&\{u \in U_T ; |\overline u|\leq R, ~0<|\underline u|\leq\delta_T\}~~~and\\
	I(R)&=&\{v\in U_T; |\overline v|\leq R, ~|\underline v|=0\}.
	\eeas
	\ed

	\begin{theorem}\label{Selection2}
		Assume that {\rm [{\bf H1}]-[{\bf H5}]} and {\rm [{\bf S}]} hold and that 
		\begin{eqnarray*}\hat u_{k, T} \overset{P}{\rightarrow} 0\qquad(k\in \calj_0).
		\end{eqnarray*}
		Then 
		\begin{eqnarray*}\lim_{T \to \infty}P\big[(\hat \theta_{k, T})_{k \in \calj_0}=0 \big] = 1.
		\end{eqnarray*}
		
	\end{theorem}
The proof is similar to the proof of Theorem 4.3 in \cite{yoshida2022quasi}, and written in Section \ref{proofs}.
\begin{rem}
	{\rm  In this paper,  we assume that $d_k >0$ and $\rho_k >1$ for all $k \in \calj_0$ from the beginning. Under this assumption, Condition [{\bf S2}]  in  \cite{yoshida2022quasi} always holds, and    [{\bf S1}]  in  Yoshida and Yoshida \cite{yoshida2022quasi}  is equivalent to  [{\bf S}].  }
	\end{rem}
	Similarly as Yoshida and Yoshida \cite{yoshida2022quasi}, 
		we give  a simple example where the selection consistency condition [{\bf S}] holds. A more complex case is treated in Section 4.4 in \cite{yoshida2022quasi}.
	\begin{example}\label{exeasysparse}
		{\rm Take $\Theta$ as 
			$
			\Theta = \prod_{i=1}^\sfp I_i,
			$ 
			where $I_i $ $(i=1, ..., \sfp)$ are subsets of $\bbR$.   Then
			for any $R>0$ and any positive sequence $\delta_T$ with $\delta_T\rightarrow 0$ as $T\rightarrow\infty$, 
			\beas
		\sup_{u\in S(R, \delta_T)}\inf_{v \in I(R)}\frac{|\overline u-\overline v|}{\sum_{k \in \calj_{0 }}|u_k|^{q_k}}
			&\leq & \sup_{u\in S(R, \delta_T)}\frac{|\overline u-\overline u|}{\sum_{k \in \calj_{0}}|u_k|^{q_k}}
			\yeq
			0\qquad{\big(\because \ol u \in S(R, \delta_T) \implies \ol u \in I(R) \big)}.
			\eeas
			Thus, [{\bf S}] holds. More generally, if  $\Theta$ can be decomposed as 
			\beas
			\Theta \yeq \bigg\{\theta \in \bbR^\sfp ; \big( (\theta_i)_{i \in \{1, ..., \sfp\} \setminus \calj_0}, (\theta_k)_{k \in \calj_0} \big)  \in A \times B \bigg\}
			\eeas
			for some $A \subset \bbR^{\sfp -|\calj_0|}$ and some $B \subset \bbR^{|\calj_0|}$,
			then the same argument goes.
		
	}
		
	\end{example}

\section{Superposition of  parametric proportional hazard models}\label{countingprocessgab}	
	Let $\mathcal{B} = (\Omega, \mathcal{F}, {\bf F}, P)$, ${\bf F}= (\mathcal{F}_t)_{t \in \bbR_+}$\footnote{\colr $\bbR_+=[0, \infty)$.} be a stochastic basis, and assume that {\bf F} satisfies usual conditions. We consider a one-dimensional point process $N_t$ 
	on $\bbR_+$. For the sake of simplicity,   the initial value $N_0$ assumed to be zero.
	Let the ${\bf F}$-intensity of $N_t$ be 
	\begin{eqnarray}\label{superposition}
	&&\lambda_t(g, \alpha, \beta)\yeq  g+\sum_{j=1}^{\sfa}\alpha_j e^{\beta_j X_{t}^{j}}\\
	&&\big(g \in [0, M_g], \, \alpha=(\alpha_1, ...,  \alpha_{\sfa}) \in [0, M_\alpha]^\sfa,  \, \beta=(\beta_1, ..., \beta_{\sfa}) \in {\colorr [-L_\beta, M_\beta]^{\sfa}},  \, t\geq0\big),\nonumber
	\end{eqnarray} 
	where $\{X_t^j\}_{t\in \bbR_+}$ $(j=1, ..., \sfa)$  are  
	$\bbR$-valued 
	${\bf F}$-predictable processes, and {\colorr  $M_g, M_\alpha>0$ and  $L_\beta, M_\beta \geq 0 $  are  constants with $ M_\beta+L_\beta >0$}.  
	Although $\beta_j$ and $X_t^j$ $(j=1, ..., \sfa)$ are assumed to be one-dimensional,  the discussion below can be generalized for vector-valued $\beta_j$ and $X_t^j$ as   (\ref{superpositiongeneral}).  The superposed model (\ref{superposition})  is also  thought of   as a two-layered neural network employing  exponential as the activation functions in  the hidden layer.

	Assume that $\{X^j_t\}_{t \in \bbR_+}$  $(j=1, ..., \sfa)$ are locally bounded. 
	We also assume the ergodicity of $X_t=(X_t^1, ..., X_t^\sfa)$  as follows: for any  bounded measurable function  $\psi$ on $\bbR^\sfa$,
	\begin{eqnarray}\label{ergo1}
	\frac{1}{T}\int_0^T    \psi(X_t)  dt \overset{P}{\rightarrow}   \int_{\bbR^\sfa} \psi(x) \nu(dx)\qquad(T \rightarrow \infty).
	\end{eqnarray}   where $\nu$ is a probability measure on $\bbR^\sfa$. 
	\begin{en-text}
		Note that from (\ref{ergo1}), for any compact subset $K \subset \bbR^{\sf{k}}$ and  any continuous function $f: \bbR^\sfa\times K \to \bbR$ satisfying that  $\big\{\sup_{\xi \in K}|f(X_t, \xi)| \big\}_{t \geq 0}$ is uniformly integrable with respect to $P$, 
		\begin{eqnarray}\label{ergouniform}
		\sup_{\xi \in K}\bigg|\frac{1}{T}\int_0^T   f(X_t, \xi)  dt- \int_{\bbR^\sfa} f(x, \xi) \nu(dx)\bigg|{\rightarrow} ~  0\quad in~L^1(dP)\qquad(T \rightarrow \infty).
		\end{eqnarray}  
	\end{en-text}
	
	Let the true intensity  be 
	\beas 
	\lambda^*_t \yeq g^*+\sum_{i \in \cala}\alpha_i^* e^{\beta_i^* X_t^i}\qquad(t \geq 0),
	\eeas
	where $\cala$ is a subset of $\{1, ..., \sfa\}$  and $g^* \in (0, M_g) , \alpha_i^*\in (0, M_\alpha),  \beta_i^* \in (-L_\beta, M_\beta)\setminus \{0\}$  $(i\in \cala)$. 
	{\colorr Note that if $\cala \neq \{1, ..., \sfa\}$, then  this model is non-identifiable. Indeed, for any subset $\calb \subset \cala^c=\{1, ..., \sfa\} \setminus \cala$ and any $\big(g, (\alpha_k)_{k \in \calb}\big) \in [0, M_g] \times [0, M_\alpha]^{|\calb|}$ with $g + \sum_{k \in \calb}\alpha_k=g^*$,
		\beas
		\lambda_t^*\yeq g+ \sum_{k\in \calb}\alpha_k e^{0\cdot X_t^k}+ \sum_{i\in \cala}\alpha_i^* e^{\beta_i^* X_t^i} \qquad(t \geq 0).
		\eeas}
	Therefore, there are many $(g, \alpha, \beta)$ satisfying $\lambda_t(g, \alpha, \beta)=\lambda_t^*$ $(t\geq 0)$.
	
	For given $T>0$, we consider  penalized estimation and maximize the estimation function
	\begin{eqnarray*}
		&&\mathbb{\Psi}_T(g, \alpha, \beta) \yeq \int_0^{T} \log\lambda_t(g, \alpha, \beta)dN_t-  \int_0^{T} \lambda_t(g, \alpha, \beta) dt -\kappa_{g}T^{\frac{r}{2}}|g|^q-\kappa_{\alpha}T^{\frac{r}{2}}\sum_{j=1}^{\sfa}|\alpha_j|^q\\
		&& \big(g \in [0, M_g], \alpha \in [0, M_\alpha]^\sfa,  \beta\in [-L_\beta, M_\beta]^{\sfa}\big),
	\end{eqnarray*} where $\kappa_{g}, \kappa_{\alpha}, r, q\geq 0$ are  tuning parameters with $0<q<r\leq1$.
	Let $(\hat g_T, \hat \alpha_T, \hat \beta_T)=(\hat g_{T}, \hat \alpha_{1, T}, ..., \hat \alpha_{\sfa, T}, \hat \beta_{1, T}, ..., \hat\beta_{\sfa, T})$ be a random variable which maximizes $\mathbb\Psi_T$ on $[0, M_g]\times [0, M_\alpha]^\sfa\times [-L_\beta, M_\beta]^\sfa$.
	
		 We assume the following conditions.
		 \bd
		 \im[{$\bf [P1]$}] 
		 {\colorr For any $p> 1$,
		 	\beas \sup_{\beta_j \in [-L_\beta, M_\beta], t \geq 0}E\bigg[\bigg\{\big(1+|X_t^j|^3\big) e^{\beta_j X^j_t}\bigg\}^p\bigg] < \infty \qquad(j=1, ..., \sfa).\eeas } 
		 
		  \im[{$\bf [P2]$}]  
		  For any $\beta=(\beta_1, ..., \beta_\sfa) \in [-L_\beta, M_\beta]^\sfa$,
		 \beas
		 \int_{\bbR^\sfa}{w(\beta, x)^{\otimes2}}\nu(dx)
		 \text{~~ is non-degenerate},
		 \eeas
		 where for any $x=(x_1, ..., x_\sfa) \in \bbR^\sfa$ and any $\beta \in [-L_\beta, M_\beta]^\sfa$,
		 \bea\label{w(beta, x)}
		 w(\beta, x)\yeq \bigg(1, ~ \bigg(\frac{e^{\beta_{i}x_{i}}-e^{\beta^*_{i}x_{i}} }{\beta_{i}-\beta_{i}^*}\bigg)_{ i \in \cala}\, , ~~\big(e^{\beta^*_{i}x_{i}}\big)_{i \in \cala}\, ,~~\bigg(\frac{e^{\beta_{k}x_{k}}-1}{\beta_{k}}\bigg)_{k \in \cala^c}\bigg).
		 \begin{en-text}
		 	w(\beta, x)&=&\big(1, \frac{e^{\beta_{k_1}x_{k_1}}-e^{\beta^*_{k_1}x_{k_1}} }{\beta_{k_1}-\beta_{k_1}^*}, ..., \frac{e^{\beta_{k_\sfb}x_{k_\sfb}}-e^{\beta^*_{k_\sfb}x_{k_\sfb}} }{\beta_{k_\sfb}-\beta_{k_\sfb}^*}, e^{\beta^*_{k_1}x_{k_1}}, ..., e^{\beta^*_{k_\sfb}x_{k_\sfb}},\\
		 &&\frac{e^{\beta_{l_1}x_{l_1}}-1}{\beta_{l_1}}, ..., \frac{e^{\beta_{l_{\sfa-\sfb}}x_{l_{\sfa-\sfb}}}-1}{\beta_{l_{\sfa-\sfb}}}\big)\\ 
		 &&~~~~~~~~~~~~~~~~~~~~~~~~~~~~~ ~~~~~~~~~~~\qquad\big(x=(x_1, ..., x_\sfa) \in [0, \infty)^\sfa\big).
		\end{en-text}
		 \eea

		 \im[{$\bf [P3]$}]  $\kappa_g<\kappa_\alpha$.
		 \ed
		 We define a  random variable $\overline\Delta^\dagger$
		 taking value in $\bbR^{1+2|\cala|}$ as $\displaystyle \overline\Delta^{\dagger} \sim N_{1+2|\cala|}\big(\ol \mu, \overline \Gamma\big)$\footnote{\colorr $N_{m}(\mu, \Sigma)$ denotes the $m$-dimensional Gaussian distribution with mean $\mu$ and covariance matrix $\Sigma$.} where
		 \beas
		 \ol\mu&=&-q1_{\{r=1\}}\bigg(\kappa_g(g^*)^{q-1},~ \big(0\big)_{i \in \cala}\,, ~ \big(\kappa_\alpha(\alpha^*_{i})^{q-1} \big)_{i \in \cala}\bigg), \label{olmu}\\
		 \overline \Gamma&=&\int_{\bbR^\sfa}\frac{v(x)^{\otimes2}}{g^*+\sum_{i\in \cala}\alpha_i^* e^{\beta_i^* x_i}} \nu(dx) \label{olgam}
		 \eeas
		 and
		 \beas
		 v(x)&=&\bigg(1,~  \big(\alpha_{i}^*x_{i}e^{\beta^*_{i}x_{i}}\big)_{i \in \cala},~  \big(e^{\beta^*_{i}x_{i}}\big)_{i \in \cala}\bigg)  \qquad\big(x=(x_1, ..., x_\sfa) \in {\colr \bbR^\sfa}\big).
		 \eeas
		 Note that if [{\bf P2}] holds, then $\overline{\Gamma}$ is also non-degenerate.
		 \begin{theorem}\label{point1} Assume  ${\bf [P1]}$-${\bf [P3]}$. Then
		 	\bea
		&& T^{\frac{1}{2}} \bigg(\hat g_T-g^*,~\big(\hat\beta_{i, T}-\beta_i^*\big)_{i \in \cala}\,, ~ \big(\hat \alpha_{i, T}-\alpha^*_i\big)_{i \in \cala}\bigg) \overset{d}{\to}		 	\overline \Gamma^{-1}\overline\Delta^{\dagger},\label{weakcon1inex1}\\
		&&T^{\frac{r}{2q}}\big(\hat \alpha_{k, T}\big)_{k \in \cala^c}\overset{P}{\to}0.\label{weakcon2ex1}
		 	\eea
		 	Moreover, 
		 	\beas
		 	\lim_{T\to \infty}P\big[\hat \alpha_{k, T}=0 ~~(k \in \cala^c),~~\hat \alpha_{i, T}\neq 0 ~~(i \in \cala)\big] =1.
		 	\eeas
		 	\end{theorem}
	 	 The proof is given in Section \ref{proof of points}, where Theorem \ref{thmY2} is applied. 
	 	 The key function $h$ in [{\bf H1}] and [{\bf H2}] can be obtained as follows. 
	 	 Define $\eta$ and $\eta^*$ as for any $g \in [0, M_g], \alpha \in [0, M_\alpha]^\sfa,  \beta\in [-L_\beta, M_\beta]^{\sfa} $ and $x \in \bbR^\sfa$,
	 	 \beas
	 	 \eta(g, \alpha, \beta, x) &=&g+\sum_{j =1}^\sfa \alpha_je^{\beta_j x_j},\qquad \eta^*(x) \yeq  g^*+\sum_{i \in \cala}\alpha_i^*e^{\beta_i^* x_i}.
	 	 \eeas
	 	 Note that $ \eta(g, \alpha, \beta, X_t)=\lambda_t(g, \alpha, \beta)$ and  $ \eta^*(X_t)=\lambda_t^*$.  Then
	 	\beas
	 	\eta(g, \alpha, \beta, x)-\eta^*(x)&=&g-g^*+\sum_{i \in \cala}\big\{\alpha_ie^{\beta_i x_i} -
	 	\alpha_i^*e^{\beta_i^* x_i}\big\} +\sum_{k \in \cala^c}\alpha_k e^{\beta_k x_k}\\
	 	\\&=&g-g^*+\sum_{i \in \cala}\sum_{m=0}^\infty \big(\alpha_i \beta_i^m -
	 	\alpha_i^*(\beta_i^*)^m \big)\frac{x_i^m}{m!} +\sum_{k \in \cala^c}\sum_{m=0}^\infty \alpha_k \frac{(\beta_k x_k)^m}{m!}.
	 	\eeas
 	 By equating the coefficients,  the equation	$\eta(g, \alpha, \beta, \cdot)-\eta^*(\cdot)=0$ holds if and only if 
 		\beas
 	\begin{cases} f_1 := g-g^*+\sum_{i \in \cala}(\alpha_i -\alpha_i^*)+ \sum_{k \in \cala^c}\alpha_k \yeq 0 
\y
 f_{2i} := \alpha_i \beta_i^m -
 \alpha_i^*(\beta_i^*)^m \yeq 0 \qquad(m  \geq 1, i\in\cala) \y
f_{3k} :=\alpha_k \beta_k^m \yeq 0 \qquad(m  \geq 1, k \in \cala^c)
 \end{cases}.
  \eeas
 This is equivalent to 
 \beas
 \begin{cases} g-g^*+\sum_{k \in \cala^c}\alpha_k \yeq 0 
 	\y
 	 \beta_i -
 	 \beta_i^* \yeq 0 \qquad( i\in\cala) 
 	\y
 	\alpha_i -
 	\alpha_i^*\yeq 0 \qquad(i\in\cala) \y
 	\alpha_k \beta_k \yeq 0 \qquad( k \in \cala^c)
 \end{cases}.
 \eeas
 That is, the ideal $\big(f_1, (f_{2i})_{i \in \cala}, (f_{3k})_{k \in \cala^c} \big)$ of $\bbR[g, \alpha,  \beta]$  equals  the ideal  
 \beas
 I = \bigg(g-g^*+\sum_{k \in \cala^c}\alpha_k,  ~
 (\beta_i -
 \beta_i^*)_{i\in\cala},~
 (\alpha_i -
 \alpha_i^*)_{ i\in\cala} , ~
 (\alpha_k \beta_k)_{ k \in \cala^c}\bigg).
 \eeas
 Therefore, we can decompose $\eta(g, \alpha, \beta, \cdot)-\eta^*(\cdot)$ as for any $x \in \bbR^\sfa$, 
 \beas
 \eta(g, \alpha, \beta, x)-\eta^*(x)\yeq \tilde w(\alpha, \beta, x) \bigg[\bigg(g-g^*+\sum_{k \in \cala^c}\alpha_k,  ~
 (\beta_i -
 \beta_i^*)_{i\in\cala},~
 (\alpha_i -
 \alpha_i^*)_{ i\in\cala} , ~
 (\alpha_k \beta_k)_{ k \in \cala^c}\bigg)\bigg],
 \eeas
 where
 \beas
 \tilde w(\alpha, \beta, x)\yeq \bigg(1, ~ \bigg(\alpha_i\frac{e^{\beta_{i}x_{i}}-e^{\beta^*_{i}x_{i}} }{\beta_{i}-\beta_{i}^*}\bigg)_{ i \in \cala}\, , ~~\big(e^{\beta^*_{i}x_{i}}\big)_{i \in \cala}\, ,~~\bigg(\frac{e^{\beta_{k}x_{k}}-1}{\beta_{k}}\bigg)_{k \in \cala^c}\bigg).
 \eeas
 Since $\tilde w$ is degenerate if $\alpha_i=0$ $(i \in \cala)$,  we take $\alpha_i$ from $\tilde w$ and obtain
 \beas
  \eta(g, \alpha, \beta, x)-\eta^*(x)\yeq w(\beta, x) \bigg[\bigg(g-g^*+\sum_{k \in \cala^c}\alpha_k,  ~
  \alpha_i(\beta_i -
  \beta_i^*)_{i\in\cala},~
  (\alpha_i -
  \alpha_i^*)_{ i\in\cala} , ~
  (\alpha_k \beta_k)_{ k \in \cala^c}\bigg)\bigg],
 \eeas 
 where $w$ is defined in (\ref{w(beta, x)}).
 Thus,  defining  $h(g, \alpha, \beta)$ as \beas
 h(g, \alpha, \beta)\yeq \big(g-g^*+\sum_{k \in \cala^c}\alpha_k,  ~
 \alpha_i(\beta_i -
 \beta_i^*)_{i\in\cala},~
 (\alpha_i -
 \alpha_i^*)_{ i\in\cala} , ~
 (\alpha_k \beta_k)_{ k \in \cala^c}\big),\eeas we have
 \bea\label{hexample}
 \lambda_t\big(g, \alpha, \beta)-\lambda_t^*=w(\beta, X_t)[h(g, \alpha, \beta)] \qquad\big(t\geq 0, g \in [0, M_g], \alpha \in [0, M_\alpha]^\sfa,  \beta\in [-L_\beta, M_\beta]^{\sfa}\big),
 \eea
 Note that the remaining function $w$ is non-degenerate in the sense of [{\bf P2}].

	\begin{en-text} \beas
	\lambda_t\big(g, \alpha, \beta)-\lambda_t^*&=&g-g^*+\sum_{i \in \cala}\big\{\alpha_ie^{\beta_i x_i} -
	\alpha_i^*e^{\beta_i^* x_i}\big\} +\sum_{k \in \cala^c}\alpha_k e^{\beta_k x_k}\\
	\\&=&g-g^*+\sum_{i \in \cala}\sum_{m=0}^\infty \big(\alpha_i \beta_i^m -
	\alpha_i^*(\beta_i^*)^m \big)\frac{x_i^m}{m!} +\sum_{k \in \cala^c}\sum_{m=0}^\infty \alpha_k \frac{(\beta_k x_k)^m}{m!}.
	\eeas
	For any $i \in \cala$,
	\beas
	\sum_{m=0}^\infty \big(\alpha_i \beta_i^m -
	\alpha_i^*(\beta_i^*)^m \big)\frac{x_i^m}{m!} &=& \alpha_i \sum_{m=0}^\infty \big( \beta_i^m -
	(\beta_i^*)^m \big)\frac{x_i^m}{m!} +(\alpha_i-\alpha_i^*)\sum_{m=0}^\infty  \beta_i^m \frac{x_i^m}{m!} \\
	&=& \alpha_i(\beta_i- \beta_i^*) \sum_{m=0}^\infty \frac{ \beta_i^m -
		(\beta_i^*)^m}{\beta_i-\beta_i^*} \frac{x_i^m}{m!} +(\alpha_i-\alpha_i^*)\sum_{m=0}^\infty  \beta_i^m \frac{x_i^m}{m!} 
	\eeas
	For any $k \in \cala^c$,
	\beas
	\sum_{m=0}^\infty \alpha_k \frac{(\beta_k x_k)^m}{m!}&=&\alpha_k +\alpha_k\beta_k\sum_{m=1}^\infty \alpha_k \frac{\beta_k^{m-1} x_k^m}{m!}
	\eeas
	Thus, 
	\beas
	\lambda_t\big(g, \alpha, \beta)-\lambda_t^*&=&g-g^*+\sum_{i \in \cala} \alpha_i(\beta_i- \beta_i^*) \sum_{m=0}^\infty \frac{ \beta_i^m -
		(\beta_i^*)^m}{\beta_i-\beta_i^*} \frac{x_i^m}{m!} +\sum_{k \in \cala^c}\sum_{m=0}^\infty \alpha_k \frac{(\beta_k x_k)^m}{m!}+\sum_{k \in \cala^c}\alpha_k +\alpha_k\beta_k\sum_{m=1}^\infty \alpha_k \frac{\beta_k^{m-1} x_k^m}{m!}
	\eeas
\end{en-text}

\section{Counting process having intensity with multicollinear covariates}\label{pointprocesswithmulticolinear}	\newcommand\ftrue{\alpha^*} \newcommand\true{\alpha^{**}}
		In this section, we also treat time-series data derived from some counting process although 
	the discussion here  can obviously applied to other examples that is essentially based on   the linear regression.
	Similarly as Section \ref{countingprocessgab}, let $\mathcal{B} = (\Omega, \mathcal{F}, {\bf F}, P)$, ${\bf F}= (\mathcal{F}_t)_{t \in \bbR_+}$ be a stochastic basis, and assume that {\bf F} satisfies usual conditions. We consider a 
	counting process $N_t$ 
	on $\bbR_+$. For the sake of simplicity,   the initial value $N_0$ assumed to be zero. 
	Let the ${\bf F}$-intensity of $N_t$ be 
	\begin{eqnarray}\label{ex2lambda}
	\lambda_t(\alpha)&=& \sum_{j=1}^{\sfa}\alpha_j X_{t}^{j}\qquad\big(\alpha=(\alpha_1, ...,  \alpha_{\sfa}) \in [0, M_\alpha]^\sfa \big),
	\end{eqnarray} 
	where $\{X_t^j\}_{t\in \bbR_+}$ $(j=1, ..., \sfa)$  are  ${\bf F}$-predictable 
	non-negative progressively measurable processes, and  $M_\alpha>0$  is a constant. 
	As parametric proportional hazard models, we may consider the composite with exponential function on the right-hand side of (\ref{ex2lambda}). In this case, the model does not need the boundary constraint $\alpha_j \geq 0$ and becomes slightly simpler.
	
	Assume that  $\{X^j_t\}_{t \in \bbR_+}$  $(j=1, ..., \sfa)$ are locally bounded.
	We also assume the ergodicity of $X_t=(X_t^1, ..., X_t^\sfa)$ as (\ref{ergo1}) for any bounded {measurable} function $\psi$ on $[0, \infty)^\sfa$, where $\nu$ is a probability measure on $[0, \infty)^\sfa$.  

	\begin{en-text} 案2 We  suppose that  for  an unknown non-empty subset  $D  \subsetneq \{1, ..., \sfa\}$,
		the set $ \{X^j\}_{j\in D}$ forms a basis in  the linear space generated by $\{ X^1, ..., X^{\sfa}\}$. Then  there exist  unknown real numbers  $\{b_{ij}\}_{i \in D^c,  j \in D}$ such that 
		\bea
		X_t^i=\sum_{j \in D} b_{i j}X_{t}^j \qquad(i \in D^c). \label{linearr}
		\eea
		There exists some $R>0$ satisfying
		\bea\label{Gnonde2}
		\int_{[0, \infty)^{\sfa}, |x|\leq R}\big((x_{j})_{j \in D}\big)^{\otimes2} \nu(dx)~is~non\text{-}degenerate.
		\eea
	\end{en-text}
	
	Suppose that $X$ is a  separable and stationary process whose invariant distribution is $\nu$   with   $E\big[|X_t|^2\big]=\int_{[0, \infty)^\sfa}|x|^2\nu(dx)< \infty$.   We consider  the singular case where $X$ is  multicollinear as
	\bea\label{ex2perfectmultico}
	0 ~<~ \sfr:={\rm rank}\bigg[	\int_{[0, \infty)^{\sfa}}x^{\otimes2}\nu(dx) \bigg]~<~ \sfa.
	\eea
	Then there exists an unknown subset  $D  \subsetneq \{1, ..., \sfa\}$ with $|D|=\sfr$ such that
	\bea\label{Gnonde2}
	\int_{[0, \infty)^{\sfa}}\big((x_{j})_{j \in D}\big)^{\otimes2}  \nu(dx)~is~non\text{-}degenerate.
	\eea
	Since
	\beas
	E\bigg[\frac{1}{T}\int_0^T X_t^{\otimes2}  dt\bigg] \yeq \int_{[0, \infty)^{\sfa}}x^{\otimes2} \nu(dx) \qquad(T>0),
	\eeas
	there exist  unknown real numbers  $\{b_{ij}\}_{i \in D^c,  j \in D}$ such that with probability one, 
	\bea
	X_t^i=\sum_{j \in D} b_{i j}X_{t}^j \qquad(i \in D^c, a.e. \,t). \label{linearr}
	\eea
	From the separability of $X$, (\ref{linearr}) holds for every $t\geq 0$.
	Thus, we have with probability one, 
	\bea\label{AandX}
	A \big((X^j)_{j \in D}\big)^\prime=X^\prime,
	\eea
	where $A=(A_{ij})_{\leq i\leq \sfa,  j \in D}$ is an $\sfa \times \sfr$ matrix  defined as
	\beas
	A_{ij} \yeq \begin{cases} \delta_{ij}& (i \in D, j \in D)\\ b_{ij}& (i \in D^c, j \in D) \end{cases}.
	\eeas
	Here $\delta_{ij}$ denotes the Kronecker delta.
	Then 
	\bea\label{lambdaforh}
	\lambda_t(\alpha)&=& \alpha A\big((X^j)_{j \in D}\big)^\prime \qquad\big( \alpha \in [0, M_\alpha]^\sfa \big).
	\eea
	
	{
		Let the true intensity $\lambda^*$ be 
		\beas
		\lambda^*_t \yeq \sum_{j=1}^\sfa \ftrue_j X_t^j \qquad(t\geq 0),
		\eeas
		where $\ftrue=(\ftrue_1, ..., \ftrue_\sfa) \in [0, \infty)^\sfa\setminus\{0\}^\sfa$. 
		\begin{en-text}
			We define non-random functions $\eta(x, \alpha)$ and $\eta^*(x)$ as
			\beas
			\eta(x, \alpha)\yeq \sum_{j=1}^\sfa \alpha_j x_j,  \quad \eta^*(x)=\sum_{j=1}^\sfa \ftrue^*_j x_j
			\qquad\big(x=(x_1, ..., x_\sfa) \in[0, \infty)^\sfa, \alpha \in [0, M_\alpha]^\sfa\big). \eeas
		\end{en-text} 
		Then for any $\alpha \in [0, M_\sfa]^\sfa$, the equation $\lambda_t(\alpha)=\lambda^*_t $ $\,(t \geq 0)$ holds   almost surely if and only if
		\beas
		\alpha A=\ftrue A. \label{truea}
		\eeas
		Thus, $\alpha$ is  non-identifiable sine $\sfa>\sfr$.  In particular, the set of all the true values is
		\beas
		\{\alpha \in [0, M_\alpha]^\sfa ; \alpha A= \ftrue A \} \yeq \{\ftrue + {\rm Ker}(A)\}\cap [0, M_\alpha]^\sfa,
		\eeas
		where ${\rm Ker}(A)$ denotes the kernel of $A : \bbR^\sfa \to \bbR^\sfr$,  $\alpha \mapsto \alpha A$. 
	}
	
	\begin{en-text}
		Let the true intensity $\lambda^*$ be 
		\beas
		\lambda^*_t=\sum_{j\in D}a_j^* X_t^j \qquad(t\geq 0),
		\eeas
		where $(a_j^*)_{j \in D}$ is an element of $[0, \infty)^{|D|}\setminus\{0\}^{|D|}$.
		Then for any $\alpha \in [0, M_\sfa]^\sfa$, the equation $\lambda_t(\alpha)=\lambda^*_t $ $\,(t \geq 0)$ holds   if and only if
		\bea
		\alpha A=(a_j^*)_{j \in D}^{\prime}. \label{truea}
		\eea
		Thus, this model is  non-identifiable sine $\sfa>\sfr$. {\colr In particular, the set of all the true values is
			\beas
			\{\alpha \in [0, M_\alpha]^\sfa ; \alpha A=(a_j^*)_{j \in D}^{\prime} \} \yeq \{y^* + {\rm Ker}(A)\}\cap [0, M_\alpha]^\sfa,
			\eeas
			where ${\rm Ker}(A)$ denotes the kernel of $A^\prime : \bbR^\sfa \to \bbR^\sfr$ {\colb $\alpha \to \alpha A$} and $y^*=(y_1^*, ..., y_\sfa^*) \in [0, \infty)^\sfa$ satisfies 
			\beas
			y^*A=(a^*_j)_{j \in D}.
			\eeas
			An example of   $y^*$ is 
			\bea\label{y*}
			y^*_j=\begin{cases} a_j^* &(j \in D) \\ 0 &  (j \in D^c)
			\end{cases}.
			\eea }
	\end{en-text}
	
	For given $T>0$, let us consider  penalized estimation and maximize the estimation function
	\begin{eqnarray*}
		\mathbb{\Psi}_T(\alpha) &=& \int_0^{T} \log\lambda_t(\alpha)dN_t-  \int_0^{T} \lambda_t(\alpha) dt -T^{\frac{r}{2}}\sum_{j=1}^{\sfa}\kappa_j|\alpha_j|^q\\
		&& \qquad\big(\alpha=(\alpha_1, ...,  \alpha_{\sfa}) \in [0, M_\alpha]^\sfa\big),
	\end{eqnarray*} where $\kappa_{j}, q, r >0$ $(j=1, ..., \sfa)$ are tuning parameters  with $0<q<r \leq 1$.
	Let $\hat \alpha_T=(\hat \alpha_{1, T}, ..., \hat \alpha_{\sfa, T})$ be a random variable which asymptotically maximizes $\mathbb\Psi_T$ on $ [0, M_\alpha]^\sfa$.
	
	We prepare to derive the asymptotic behavior of $\hat\alpha_T$. 
	Denote   by $Pe$ the time-invariant part of the penalty function    i.e.
	\beas
	Pe(\alpha)=\sum_{j=1}^\sfa\kappa_j|\alpha_j|^q \qquad(\alpha \in \bbR^\sfa).
	\eeas
	We consider the following conditions.
	\bd
	\im[${\bf [L1]}$] There exists some {$\true=(\true_1, ..., \true_\sfa) \in \{\ftrue + {\rm Ker}(A)\}\cap[0, M_\sfa)^\sfa$} which uniquely minimizes  $Pe$ on the true value sets $\{\ftrue + {\rm Ker}(A)\}\cap [0, M_\alpha]^\sfa$.

	\im[${\bf [L2]}$] $\ds \int_{[0, \infty)^\sfa}\frac{|x|^p}{\big\{\sum_{j=1}^\sfa\ftrue_jx_j\big\}^{p-1}}{ 1_{\big\{\sum_{j=1}^\sfa \ftrue_jx_j>0\big\}}}\nu(dx) <\infty$ $\qquad(p=2, 3, 4)$.
	
	\ed
	In the sense of [{\bf L1}], $\true$ is the most economical value.   Similarly as before, define $\calj_0, \calj_1\subset\{1, ..., \sfa\}$ as two partitions of $ \calj :=\{1, ..., \sfa\}$ satisfying 
	\beas
	\true_i \neq 0,~~ \true_k=0 \qquad(i \in \calj_1, k \in \calj_0).
	\eeas
	We also define an $\bbR^{|\calj_1|}$-valued random variable $\overline\Delta $ as
	$\displaystyle \overline\Delta \sim N_{|\calj_1|}(0, \overline \Gamma)$ , where
	\beas
	\overline \Gamma&=&\int_{[0, \infty)^\sfa}\frac{\big((x_i)_{i \in \calj_1}\big)^{\otimes2}}{\sum_{j=1}^\sfa\ftrue_j x_j}1_{\big\{\sum_{j=1}^\sfa \ftrue_jx_j>0\big\}} \nu(dx).
	\eeas
	From Lemma \ref{lemmafinal} described later,  {\colr $|\calj_1| \leq \sfr$}, and $\overline{\Gamma}$ is  non-degenerate under [{\bf L1}] and [{\bf L2}]. We  define an $\bbR^{|\calj_1|}$-valued random variable $\overline\Delta^\dagger$ as
	\beas
	\overline\Delta^{\dagger}\yeq \overline\Delta-q1_{\{r=1\}}(\kappa_i|\true_{i}|^{q-1})_{i \in \calj_1}.
	\eeas
	\begin{theorem}\label{point2} Assume $[{\bf L1}]$-$[{\bf L3}]$. Then
		\beas
		\big(T^{\frac{1}{2}}(\hat \alpha_{i, T}-
		\true_{i})_{i \in \calj_1}, T^{\frac{r}{2q}}(\hat\alpha_{k, T})_{k \in \calj_0}\big)\overset{d}{\to} \big(\overline\Gamma^{-1}\overline\Delta^{\dagger}, 0\big)
		\eeas
		Moreover, 
		\beas
		\lim_{T\to \infty}P\big[\hat\alpha_{k, T}=0 ~~(k \in \calj_0), ~~\hat\alpha_{i, T}\neq 0 ~~(i \in \calj_1)\big]=1.
		\eeas
	\end{theorem}
	The proof is given in Section \ref{proof of points}.

	{Next, we consider a sufficient condition for [{\bf L1}]. Define $\cale$ as a family of all $E \subset \{1, ..., \sfa\}$ satisfying that $|E|=\sfa-\sfr$ and that
		\beas
		\left<\{e_j ; j \in E\}\right>\oplus {\rm Ker} A=\bbR^\sfa,
		\eeas
		where we denote by $e_j$ $(j=1, ..., \sfa)$ the element of $\bbR^\sfa$ with its $j$-th component $1$ and the other components $0$. For any $E \in \cale$, we denote by  $pr_E : \bbR^\sfa \to \bbR^\sfa$ the projection onto $\left<\{e_j ; j \in E\}\right>$ along ${\rm Ker}(A)$, that is,  for any $y \in \bbR^\sfa$,
		\beas
		pr_E(y)=\sum_{j \in E}z_je_j,
		\eeas
		where $(z_j)_{j \in E} $ is the unique element of $\bbR^{\sfa-\sfr}$ satisfying that $y-\sum_{j \in E}z_je_j \in {\rm Ker}A$.   Note that $pr_E(\ftrue)$ denotes the intersecting point of the two hyperplanes $\left<\{e_j ; j \in E\}\right>$ and $\ftrue + {\rm Ker}(A)$.
		Then we consider the following condition.
		\bd
		\im[${\bf [L1]^{\#}}$] There exists  some $E_0 \in \cale$  such that 
		$pr_{E_0}(\ftrue)\in [0, M_\alpha)^\sfa$ is the unique minimizer of $Pe$ on  the finite set 
		\beas
		\{pr_E(\ftrue) ; E \in \cale\} \cap [0, \infty)^\sfa.
		\eeas
		
		\ed
		\begin{proposition}\label{linearprop}
			If $[{\bf L1}]^{\#}$ holds, then $[{\bf L1}]$ holds and $\true$ in $[{\bf L1}]$ is determined as
			\beas
			\true \yeq pr_{E_0}(\ftrue).
			\eeas
			
		\end{proposition}
		The proof is also given in Section \ref{proof of points}.
		\begin{remark}{\rm 
				The set $\{pr_E(\ftrue) ; E \in \cale\} \cap [0,  \infty)^\sfa$ is not empty. Indeed, denote the set $\{1\leq j \leq \sfa ; \ftrue_j>0\}$ by $F_1$. If $\left<\{e_j ; j \in F_1\}\right> \cap {\rm Ker}(A) \neq \phi$, then we can choose some $\alpha=(\alpha_1, ..., \alpha_\sfa)  \in \{\ftrue+{\rm Ker}(A)\}\cap [0, \infty)^\sfa$ such that
				\beas
				F_2 :=\{1\leq j \leq \sfa ; \alpha_j>0\} \subsetneq F_1.
				\eeas
				If $\left<\{e_j ; j \in F_2\}\right> \cap {\rm Ker}(A) \neq \phi$, then do the same.  Repeating these operations, we obtain some $\alpha^\dag \in \{\ftrue+{\rm Ker}(A)\}\cap[0, \infty)^\sfa$
				satisfying $\left<\{e_j ; j \in K\}\right> \cap {\rm Ker}(A) = \phi$, where $K=\{1\leq j \leq \sfa ; \alpha_j^\dag>0\}$. Then there exists some $E_K\in \cale$ such that $K\subset E_K$, and hence we have  $\alpha^\dag = pr_{E_K}(\alpha^\dag)=pr_{E_K}(\ftrue)$. Thus,  $\{pr_E(\ftrue) ; E \in \cale\} \cap [0,  \infty)^\sfa$ is not empty.
				In particular, by  choosing sufficiently large $M_\alpha$, Condition ${\bf [L1]^{\#}}$ holds.}
	\end{remark}}
\section{Proof of Theorems \ref{thmY2} and \ref{Selection2}}\label{proofs}
\begin{proof}[Proof of Theorem \ref{thmY2}]
	We first show [{\bf A1}] using [{\bf H1}]-[{\bf H4}] and the following properties assumed by  [{\bf H5}].
	\bea
	&&	\frak b_T^{-\frac{1}{2}}\xi_{i, T}~ \leq~ C_0^{\frac{1}{2}}\frak a_{i, T}\xi_{i, T}\yeq O_P(1) \quad(i \in \calj_1), \label{H51}\\
	&& \frak b_{T}^{-\frac{q_k\rho_k}{2}}\xi_{k, T}\to^d d_k>0 \quad(k \in \calj_0),\label{H52}
	\eea where $C_0\geq 1$ is constant satisfying (\ref{bT}).
	From [{\bf H1}] (and the definition of $E_T$ in [{\bf H3}]), for any $(\theta, \tau) \in \Xi$,
	\bea
	&&\bbH_T(\theta, \tau)-\bbH_T(\theta^*, \tau) \nonumber\\
	&=&\calh_T(\theta, \tau)-\calh_T(\theta^*, \tau)-\bigg(\sum_{j \in \calj}\xi_{j, T}p_j(\theta_j)-\sum_{j \in \calj}\xi_{j, T}p_j(\theta_j^*)\bigg)\nonumber\\
	&\leq & K_T(\theta, \tau)[\frak a_T^{-1} h(\theta, \tau)]-\frac{1}{2}\big\{G(\theta, \tau)+r_T(\theta, \tau)\big\}[\big(\frak a_T^{-1}h(\theta, \tau)\big)^{\otimes 2}]\nonumber\\
	&&-
	\big(E_T(\theta)-E_T(\theta^*)\big)\nonumber\\
	&\leq & O_P(1)\big|\frak a_T^{-1}h(\theta, \tau)\big|-\frac{1}{2}\big\{\lambda+o_P(1)\big\}\big|\frak a_T^{-1}h(\theta, \tau)\big|^{ 2}-\big(E_T(\theta)-E_T(\theta^*)\big) \label{A1X0}\\
	&=& O_P(1)-\big(E_T(\theta)-E_T(\theta^*)\big), \label{A1X01}
	\eea
	where $\lambda$ is a  random variable defined by 
	$\lambda= \inf_{(\theta, \tau) \in \Xi}\lambda_{\min}\big[G(\theta, \tau)\big]$.\footnote{\colr  $\lambda_{\rm min}[A]$ denotes the minimum eigenvalue of a  matrix $A$.}
	  Note that $\lambda$ is almost surely positive from [{\bf H1}].    From the equation $p_j(0)=0$ $(j \in \calj)$, 
	\beas
	\big(\max_{j=1, ..., \sfr}\frak a_{j, T}\big)^2E_T(\theta^*) ~\lesssim~\frak b_T^{-1}E_T(\theta^*) &=& \frak b_T^{-1}\sum_{i \in \calj_1}\xi_{i, T}p_i(\theta^*_i)\\
	&=&o_P(1) \qquad\big(\because (\ref{H51}) \big).\eeas   Therefore, (\ref{A1X0}) implies
	\beas
	&& \big(\max_{j=1, ..., \sfr}\frak a_{j, T}\big)^2\big\{\bbH_T(\theta, \tau)-\bbH_T(\theta^*, \tau) \big\}\\
	&\leq&  \big(\max_{j=1, ..., \sfr}\frak a_{j, T}\big)^2\bigg\{ O_P(1)\big|\frak a_T^{-1}h(\theta, \tau)\big|-\frac{1}{2}\big\{\lambda+o_P(1)\big\}\big|\frak a_T^{-1}h(\theta, \tau)\big|^{ 2}\bigg\}-\big(\max_{j=1, ..., \sfr}\frak a_{j, T}\big)^2E_T(\theta)+o_P(1)\\
	&\leq&  \big(\max_{j=1, ..., \sfr}\frak a_{j, T}\big)^2\bigg\{ O_P(1)\big|\frak a_T^{-1}h(\theta, \tau)\big|-\frac{1}{2}\big\{\lambda+o_P(1)\big\}\big|\frak a_T^{-1}h(\theta, \tau)\big|^{ 2}\bigg\}+o_P(1){\qquad\big(\because E_T(\theta) \geq 0\big)}.
	\eeas
	Therefore, for any  open set $O \subset{\bbR^{\sfp+\sfq}}$ that contains $\Xi^*$,
	\bea\label{O}
	\varlimsup_{T \to \infty}P\bigg[\sup_{(\theta, \tau) \in \Xi \setminus O}\big\{\bbH_T(\theta, \tau)-\bbH_T(\theta^*, \tau)\big\}\geq 0\bigg]=0.
	\eea
	Thus, we only need to consider  $(\theta, \tau) \in \Xi$ near $\Xi^*$.  
	

	For a positive number $r>0$, define $B=B(r)$ as $B=\{\theta \in \bbR^\sfp ; |\theta-\theta^*|<r\} \cap \Theta$. 
	Take $r$ satisfying that  the neighborhood $\calw$  in [{\bf H2}],
	\beas
	B\subset \calw  \cap \Theta.
	\eeas
	From [{\bf H2}],  $\Xi^*\cap\big\{(\theta, \tau) \in \Xi ; \underline \theta =0 \big\}\subset \{\theta^*\}\times\calt$. Therefore, for any  $(\theta^\dagger, \tau^\dagger)\in \Xi^*\setminus \big(B\times\calt\big)$, $|(\theta^\dagger_k)_{k \in \calj_0}|$ is positive, where $\theta^\dagger_j$ represents the $j$-th component of $\theta^\dagger$.
	Take any $\delta>0$ with $\delta<\frac{\delta_0}{4|\Theta|}$, where $\delta_0$ is the positive number defined in {\bf (iii)} in Section \ref{settings}, and $|\Theta|=\sup_{\theta \in \Theta}|\theta|$. 
	Since $\Xi^*\setminus \big(B\times\calt\big)$ is compact, there exist a finite subset $\{(\theta^\dagger_{(1)},  \tau^\dagger_{(1)}), ..., (\theta^\dagger_{(L)}, \tau^\dagger_{(L)})\} \subset \Xi^*\setminus \big(B\times\calt\big)$ and  a finite open cover $\{V_l\}_{l=1}^L$ of $\Xi^*\setminus \big(B\times\calt\big)$  such that $V_l$ is an open ball with center $(\theta^\dagger_{(l)},  \tau^\dagger_{(l)})$ and with radius $\eta_{l}$  where 
	\beas
	\eta_{l}=\delta\min_{\substack{j=1, ..., \sfp,  \theta^\dagger_{(l), j}\neq 0}}|\theta^\dagger_{(l), j}|, 
	\eeas and 
	$\theta^\dagger_{(l), j}$ represents the $j$-th component of $\theta^\dagger_{(l)}$. 
	Then,  from Lemma \ref{p_j} below, for any $l=1, ..., L$, any $(\theta, \tau) \in V_l$ with $|\theta|\leq |\Theta|$ and any $j \in \calj$,
	\beas
	{p_j(\theta_j)}&\geq&		(1-C_1\delta){p_j(\theta^\dagger_{(l), j})},
	\eeas
	where $C_1$ is some positive constant depending only on $p_j$ $(j \in \calj)$  and $|\Theta|$.
	Therefore, for any $l=1, ..., L$ and any $(\theta, \tau) \in V_l   \cap \Xi$,
	\beas
	E_T(\theta)-E_T(\theta^*)
	&\geq&(1-C_1\delta)E_T(\theta^\dagger_{(l)})-E_T(\theta^*)\\
	&=&E_T(\theta^\dagger_{(l)})\bigg(1-C_1\delta-\frac{E_T(\theta^*)}{E_T(\theta^\dagger_{(l)})}\bigg).
	\eeas
	Also, for any $l=1, ..., L$, $E_T(\theta^\dagger_{(l)}) \to \infty$ in probability since $|(\theta^{\dagger}_{(l), k})_{k \in \calj_0}|>0$ and since $\xi_{k, T} \to \infty$ in probability for any $k \in \calj_0$  from (\ref{H52}). 
	Thus, from (\ref{A1X01}) and [{\bf H3}],
	\bea
	&&\varlimsup_{T \to \infty}P\bigg[\sup_{(\theta, \tau) \in \cup_{l=1}^LV_l \cap \Xi
	}\big\{\bbH_T(\theta, \tau)-\bbH_T(\theta^*, \tau)\big\}\geq 0\bigg]\nonumber\\
	&\leq &\varlimsup_{T \to \infty}P\bigg[\sup_{(\theta, \tau) \in \cup_{l=1}^LV_l  \cap \Xi}\big\{O_P(1)-\big(E_T(\theta)-E(\theta^*)\big)\big\}\geq 0\bigg]\nonumber\\
	&\leq &\varlimsup_{T \to \infty}P\bigg[\sup_{l=1, ..., L}\frac{E_T(\theta^*)}{E_T(\theta^\dagger_{(l)})}\geq 1-2C_1\delta\bigg]\nonumber\\
	&\leq&\varlimsup_{T \to \infty}P\bigg[{E_T(\theta^*)}\geq (1-2C_1\delta)\inf_{\theta^\dagger \in \Theta^*\setminus B}E_T(\theta^\dagger)\bigg]\nonumber\\
	&\to& 0 \label{V_l},
	\eea as $\delta \to 0$.  Since $ (B \times \calt)  \cup \big(\cup_{l=1}^LV_l  \cap \Xi\big)= O \cap \Xi$ for some open set $O\subset{\bbR^{\sfp+\sfq}}$ and since it contains  $\Xi^*$,  the two evaluations (\ref{O}) and (\ref{V_l}) imply that 
	\beas
	&&\varlimsup_{T \to \infty}P\bigg[\sup_{(\theta, \tau) \in \Xi \setminus (B \times \calt)}\big\{\bbH_T(\theta, \tau)-\bbH_T(\theta^*, \tau)\big\}\geq 0\bigg] \\
	&\leq& \varlimsup_{T \to \infty}P\bigg[\sup_{(\theta, \tau) \in \cup_{l=1}^LV_l  \cap \Xi}\big\{\bbH_T(\theta, \tau)-\bbH_T(\theta^*, \tau)\big\}\geq 0\bigg]
	+\varlimsup_{T \to \infty}P\bigg[\sup_{(\theta, \tau) \in \Xi \setminus O}\big\{\bbH_T(\theta, \tau)-\bbH_T(\theta^*, \tau)\big\}\geq 0\bigg]
	\\&\overset{\delta \to 0}{\to}& 0.
	\eeas
	Thus, 
	\bea
	\varlimsup_{r \to 0}\varlimsup_{T \to \infty}P\bigg[\sup_{(\theta, \tau) \in \Xi \setminus (B \times \calt)}\big\{\bbH_T(\theta, \tau)-\bbH_T(\theta^*, \tau)\big\}\geq 0\bigg] &=& 0.
	\label{neigh}
	\eea

	\begin{en-text}
	From [{\bf H4}], take some constant $C_2$ satisfying
	\beas \sup_{\substack{\theta \in \calw^\prime \cap \Theta,\\ \tau \in \calt}}\frac{|h_i(\theta, \tau)-h_i(\ol \theta, 0, \tau)|}{\sum_{k\in\calj_0}|\theta_k|^{\gamma_{k, i}}} \leq C_2 \qquad( i \in \calj_1).\eeas
	\end{en-text}
We have 
for any $(\theta, \tau) \in B\times \calt$, 
\bea
\epsilon_0 |\ol \theta-\ol \theta^*| &\leq& |h(\ol\theta, 0, \tau)| \qquad\big(\because [{\bf H2}]~{\bf (b)}\big)
\nonumber
\\&\leq& |h(\theta, \tau)| +|h(\theta, \tau) -h(\ol\theta, 0, \tau)|\nonumber 
\\&\leq&|h(\theta, \tau)| +C_2\sum_{k \in \calj_0}|\theta_k|\qquad\big(\because [{\bf H2}]~{\bf (a)}\big)\label{hC} 
\eea for some constant $C_2>0$.
Therefore,  for any $(\theta, \tau) \in B\times \calt$,
\bea
\sum_{i \in \calj_1}\xi_{i, T}|\theta_i-\theta^*_i|
&\lesssim&  \max_{i \in \calj_1}\xi_{i, T}|\ol \theta-\ol\theta^*|
\\	 &\lesssim&  \max_{i \in \calj_1}\xi_{i, T}\bigg\{|h(\theta, \tau)| +\sum_{k \in \calj_0}|\theta_k|\bigg\}\nonumber
\\&\lesssim & \max_{i \in \calj_1}\xi_{i, T} \frak b_T^{-\frac{1}{2}}\big|\frak a_{T}h(\theta, \tau)\big| +\max_{i \in \calj_1}\xi_{i, T}\sum_{k \in \calj_0} |\theta_k|\nonumber
\\&= & O_P(1)\big|\frak a_{T}h(\theta, \tau)\big| +\max_{i \in \calj_1}\xi_{i, T}\sum_{k \in \calj_0} |\theta_k|\qquad\big(\because (\ref{H51})\big).
\label{hG}
\eea 
We may assume that  $r<|\theta^*_i|/2$ $(i \in \calj_1)$ and that $r< \delta_0$,  where $\delta_0$ is the positive number in the condition {\bf (iii)} in Section \ref{settings}. 
Then for some constant $C_3, C_4>0$ and  for any $(\theta, \tau) \in B\times \calt$,
\bea
-\big(E_T(\theta)-E_T(\theta^*)\big)&=&-\sum_{i \in \calj_1}\xi_{i, T}\big( p_i(\theta_i)-p_i(\theta^*_i)\big)-\sum_{k \in \calj_0}\xi_{k, T} p_k(\theta_k) \nonumber\\
&\leq& C_3\sum_{ i \in \calj_1} \xi_{i, T}|\theta_i-\theta^*_i|-\sum_{k \in \calj_0}\xi_{k, T} |\theta_k|^{q_k}\nonumber\\
&\leq& C_4 \bigg\{O_P(1)\big|\frak a_{T}h(\theta, \tau)\big| +\max_{i \in \calj_1}\xi_{i, T}\sum_{k \in \calj_0}|\theta_k|\bigg\}-\sum_{k \in \calj_0}\xi_{k, T} |\theta_k|^{q_k}\qquad\big(\because (\ref{hG})\big)  \nonumber\\
&=& O_P(1)\big|\frak a_{T}h(\theta, \tau)\big| +\sum_{k \in \calj_0}\bigg( C_4\max_{i \in \calj_1}\xi_{i, T}\xi_{k, T}^{-1}|\theta_k|^{1-q_k}-1\bigg)\xi_{k, T} |\theta_k|^{q_k}\nonumber\\
&\leq &O_P(1)\big|\frak a_{T}h(\theta, \tau)\big| +\sum_{k \in \calj_0}\bigg(C_4 \max_{i \in \calj_1}\xi_{i, T}\xi_{k, T}^{-1}|r|^{1-q_k}-1\bigg)\xi_{k, T} |\theta_k|^{q_k} \nonumber\\
&= &O_P(1)\big|\frak a_{T}h(\theta, \tau)\big|\nonumber\\&& +\sum_{k \in \calj_0}\bigg( C_4\max_{i \in \calj_1}\xi_{i, T}\xi_{k, T}^{-1}|r|^{1-q_k}-1\bigg)\frak b_{ T}^{-\frac{\rho_kq_k}{2} }\xi_{k, T} |\frak b_{ T}^{\frac{\rho_k}{2} }\theta_k|^{q_k}.
\label{E_T-E_T}
\eea 
Therefore, the two evaluations (\ref{A1X0}) and (\ref{E_T-E_T}) imply that  for any $(\theta, \tau) \in B\times \calt$,
\bea
\bbH_T(\theta, \tau)-\bbH_T(\theta^*, \tau)&\leq &O_P(1)|\frak a_T^{-1}h(\theta, \tau)|-(\lambda+o_P(1)) |\frak a_T^{-1}h(\theta, \tau)|^2\nonumber\\
&&-\sum_{k \in \calj_0}\bigg( 1-C_4\max_{i \in \calj_1}\xi_{i, T}\xi_{k, T}^{-1}|r|^{1-q_k}\bigg)\frak b_{ T}^{-\frac{\rho_kq_k}{2} }\xi_{k, T} |\frak b_{ T}^{\frac{\rho_k}{2} }\theta_k|^{q_k}\label{H_T1}\\
&\leq &O_P(1)-\sum_{k \in \calj_0}\bigg( 1-C_4\max_{i \in \calj_1}\xi_{i, T}\xi_{k, T}^{-1}|r|^{1-q_k}\bigg)\frak b_{ T}^{-\frac{\rho_kq_k}{2} }\xi_{k, T} |\frak b_{ T}^{\frac{\rho_k}{2} }\theta_k|^{q_k}\label{H_T11}.
\eea
Since $\max_{i \in \calj_1}\xi_{i, T}\xi_{k, T}^{-1}=\begin{cases}
O_P(1) & (q_k<1) \\ o_P(1) &  (q_k=1)
\end{cases}$ $\,(k \in\calj_0)$ from  [{\bf H4}] and from (\ref{H52}),  we obtain
\bea
\varlimsup_{\eta \to +0}\varlimsup_{r \to 0}\varlimsup_{T \to \infty}P\bigg[\min_{k \in \calj_0}\bigg\{\big( 1-C_4\max_{i \in \calj_1}\xi_{i, T}\xi_{k, T}^{-1}|r|^{1-q_k}\big)\frak b_{ T}^{-\frac{\rho_kq_k}{2} }\xi_{k, T}\bigg\}\leq \eta\bigg]\yeq 0.\label{20230107eta}
\eea 
From (\ref{H_T11}) and (\ref{20230107eta}), we have for any $\eta>0$,
\bea\label{A11}
&&\varlimsup_{K \to \infty}\varlimsup_{r \to 0}\varlimsup_{T \to \infty}P\bigg[\sup_{(\theta, \tau) \in B\times \calt, |(\frak b_{ T}^{\frac{\rho_k}{2}}\theta_k)_{k \in \calj_0}|\geq K}\big\{\bbH_T(\theta, \tau)-\bbH_T(\theta^*, \tau)\big\}\geq 0\bigg] \nonumber\\
&\leq &\varlimsup_{K \to \infty}\varlimsup_{r \to 0}\varlimsup_{T \to \infty}P\bigg[\sup_{(\theta, \tau) \in B\times \calt, |(\frak b_{ T}^{\frac{\rho_k}{2}}\theta_k)_{k \in \calj_0}|\geq K}\bigg\{O_P(1)\nonumber\\
&&\qquad -\sum_{k \in \calj_0}\bigg( 1-C_4\max_{i \in \calj_1}\xi_{i, T}\xi_{k, T}^{-1}|r|^{1-q_k}\bigg)\frak b_{ T}^{-\frac{\rho_kq_k}{2} }\xi_{k, T} |\frak b_{ T}^{\frac{\rho_k}{2} }\theta_k|^{q_k}\bigg\}\geq 0\bigg]\nonumber\\
&\leq&\varlimsup_{r \to 0}\varlimsup_{T \to \infty}P\bigg[\min_{k \in \calj_0}\bigg\{\big( 1-C_4\max_{i \in \calj_1}\xi_{i, T}\xi_{k, T}^{-1}|r|^{1-q_k}\big)\frak b_{ T}^{-\frac{\rho_kq_k}{2} }\xi_{k, T}\bigg\}\leq \eta\bigg]\nonumber\\
&\overset{\eta \to +0}{\to}& 0.\label{geqK}
\eea

Take an arbitrary $K>1$. From (\ref{hC}), for any $(\theta, \tau) \in \Xi$ with $|(\frak b_{ T}^{\frac{\rho_k}{2}}\theta_k)_{k \in \calj_0}|\leq K$,
\bea
|\frak a_T^{-1}h(\theta, \tau)| &\gtrsim & \frak b_T^{\frac{1}{2}}|h(\theta, \tau)|~\geq~ \frak b_T^{\frac{1}{2}}\bigg\{\epsilon_0 |\ol\theta-\ol\theta^*|-C_2\sum_{k \in \calj_0}|\theta_k|\bigg\}\nonumber
\\
&=& \frak b_T^{\frac{1}{2}}\bigg\{\epsilon_0 |\ol\theta-\ol\theta^*|-C_2\sum_{k \in \calj_0} \frak b_T^{-\frac{\rho_k}{2}}|\frak b_T^{\frac{\rho_k}{2}}\theta_k|\bigg\}\nonumber\\
&=& \epsilon_0\frak b_T^{\frac{1}{2}} |\ol\theta-\ol\theta^*|-C_2\sum_{k \in \calj_0} O_P(1)|\frak b_T^{\frac{\rho_k}{2}}\theta_k| \nonumber\\
&\geq & \frac{\epsilon_0}{C_5} \bigg|\bigg(\frak a_{i, T}^{-1}(\theta_i-\theta^*_i)\bigg)_{i \in \{1, ..., \sfp\}\setminus\calj_0}\bigg|- O_P(1)K \label{final},
\eea 
where $C_5>0$ is some constant.
\begin{en-text}Since $1-C_4\max_{i \in \calj_1}\xi_{i, T}\xi_{k, T}^{-1}|r|^{1-q_k}= O_P(1)$ from [{\bf H4}] and since $\frak b_{ T}^{-\frac{\rho_kq_k}{2} }\xi_{k, T}=O_P(1)$ from (\ref{H52}), 
	The inequality (\ref{H_T1}) implies that  for any $(\theta, \tau) \in B \times \calt$, with $|(\frak b_{ T}^{\frac{\rho_k}{2}}\theta_k)_{k \in \calj_0}|\leq K$,
	\bea
	\bbH_T(\theta, \tau)-\bbH_T(\theta^*, \tau) 
	&\leq&  O_P(1)|\frak a_T^{-1}h(\theta, \tau)|-(\lambda+o_P(1)) |\frak a_T^{-1}h(\theta, \tau)|^2 +O_P(1)K^{q_k} \label{H_T2}
	\eea
\end{en-text}
{For any subset $A_{T, r} \subset \Xi$ depending on $T$ and $r>0$, the inequality (\ref{H_T1}) implies that 
	\bea
	&&\varlimsup_{r \to 0}\varlimsup_{T \to \infty}P\bigg[\sup_{(\theta, \tau) \in (B\times \calt) \cap A_{T, r}}\big\{\bbH_T(\theta, \tau)-\bbH_T(\theta^*, \tau)\big\}\geq 0\bigg]\nonumber\\
	&\leq & \varlimsup_{r \to 0}\varlimsup_{T \to \infty}P\bigg[\sup_{(\theta, \tau) \in (B\times \calt) \cap A_{T, r}}\bigg\{O_P(1)|\frak a_T^{-1}h(\theta, \tau)|-(\lambda+o_P(1)) |\frak a_T^{-1}h(\theta, \tau)|^2\nonumber\\
	&&-\sum_{k \in \calj_0}\bigg( 1-C_4\max_{i \in \calj_1}\xi_{i, T}\xi_{k, T}^{-1}|r|^{1-q_k}\bigg)\frak b_{ T}^{-\frac{\rho_kq_k}{2} }\xi_{k, T} |\frak b_{ T}^{\frac{\rho_k}{2} }\theta_k|^{q_k}\bigg\}\geq 0\bigg]\nonumber\\
	&\leq & \varlimsup_{r \to 0}\varlimsup_{T \to \infty}P\bigg[\sup_{(\theta, \tau) \in (B\times \calt) \cap A_{T, r}}\bigg\{O_P(1)|\frak a_T^{-1}h(\theta, \tau)|-(\lambda+o_P(1)) |\frak a_T^{-1}h(\theta, \tau)|^2\bigg\} \geq 0\bigg]\label{H_T3},
	\eea 
	where the last inequality follows since (\ref{20230107eta}) implies 
	\beas
	\varlimsup_{r \to 0}\varlimsup_{T \to \infty}P\bigg[\min_{k \in \calj_0}\bigg\{\big( 1-C_4\max_{i \in \calj_1}\xi_{i, T}\xi_{k, T}^{-1}|r|^{1-q_k}\big)\frak b_{ T}^{-\frac{\rho_kq_k}{2} }\xi_{k, T}\bigg\}\leq 0\bigg]\yeq 0.
	\eeas}
Thus, from  (\ref{final}) and (\ref{H_T3}),  for any $K>1$,
\bea \label{leqK}
\varlimsup_{R \to \infty}\varlimsup_{r \to 0}\varlimsup_{T \to \infty}P\bigg[\sup_{\substack{(\theta, \tau) \in B\times \calt, |(\frak b_{ T}^{\frac{\rho_k}{2}}\theta_k)_{k \in \calj_0}|\leq K, \\ \big|\big(\frak a_{i, T}^{-1} (\theta_i-\theta^*_i)\big)_{i \in \{1, ..., \sfp\}\setminus\calj_0}\big|\geq R}}\big\{\bbH_T(\theta, \tau)-\bbH_T(\theta^*, \tau)\big\}\geq 0\bigg]\yeq 0.
\eea
From (\ref{neigh}), (\ref{geqK}) and (\ref{leqK}), the condition [{\bf A1}] holds.

Finally, we show [{\bf A2}] for the continuous random fields $\bbV_T$ and $\bbZ_T$ defined      by (\ref{calv_T}) and (\ref{z}), respectively, under [{\bf H5}].  
Take any positive number $R>0$, and fix it. We consider sufficiently large $T$ such that 
for any $u \in U_T(R) $, $\theta^*+a_Tu \in \{x=(x_1, ..., x_\sfp) \in \ol\caln ; { |x_i-\theta^*_i|<2^{-1}|\theta^*_i|}~(i\in\calj_1),~ |x_k| < \delta_0 ~(k \in \calj_0) \}$.
For each $T\in\bbT$, define a $C(U_T\times\calt)$-valued random variable $\calz_T$ as
\beas
\calz_T(u, \tau)=\exp\big\{\calh_T(\theta^*+a_Tu, \tau)-\calh_T(\theta^*, \tau)\big\} \qquad\big((u, \tau)\in U_T\times \calt\big).
\eeas
Then  the continuous random field $\bbZ_T$ defined by (\ref{Z_T}) satisfies that 
\beas
\bbZ_T(u, \tau)=\calz_T(u, \tau)\exp\bigg\{-\sum_{i \in \calj_1}\xi_{i, T}
\big(p_i(\theta_i^*+\frak a_{i, T}u_i)-p_i (\theta^*_i)\big)-\sum_{k \in \calj_0}{ \frak b_{ T}^{-\frac{q_k\rho_k}{2}}}\xi_{k, T} |u_k|^{q_k}\bigg\}
\eeas
for any $(u, \tau) \in U_T(R) \times \calt$.
Define $\Delta_T(\theta^*, \tau)$ and $R_T(u, \tau)$ as 
\begin{eqnarray*}
	\Delta_T(\theta^*, \tau)&=&a_T^{\prime}\partial_\theta\calh_T(\theta^*, \tau)\\
	R_T(u, \tau)&=&\int_0^1(1-k)\bigg\{a_T^{\prime}\partial^2_\theta\calh_T(\theta^*+ka_Tu, \tau)a_T+ \begin{pmatrix}
		\overline \Gamma(\theta^*) & O\\
		O & O
	\end{pmatrix} \bigg\}dk \\ 		
\end{eqnarray*} for any $(u, \tau) \in U_T\times\calt$, respectively.
Then,  from Taylor's series, for any $(u, \tau) \in U_T\times\calt$,
\begin{eqnarray*}
	\calz_T(u, \tau)=\exp\bigg\{\Delta_T(\theta^*, \tau)[u]-\frac{1}{2}\overline \Gamma(\theta^*)[{\overline u}^{\otimes 2}]+R_T(u, \tau)[u^{\otimes2}] \bigg\}.
\end{eqnarray*}
Then from [{\bf H5}], 
\beas
\sup_{(u, \tau)\in U_T\times \calt, |u|\leq R}\big|R_T(u, \tau)\big|=o_P(1).
\eeas
and
\beas
\sup_{\tau \in \calt}\big|\Delta_T(\theta^*. \tau)-\Delta_T(\theta^*, \tau_0)\big|=o_P(1).
\eeas
{\colorr	Thus, for any $(u, \tau) \in U_T \times \calt$ with $|u| \leq R$,
	\beas
	\calz_T(u, \tau)&=&\exp\bigg\{\Delta_T(\theta^*, \tau_0)[u]-\frac{1}{2}\overline \Gamma(\theta^*)[{\overline u}^{\otimes2}]+o_P(1)|u|\bigg\}\\
	&=& \exp\bigg\{\ol \Delta_T(\theta^*, \tau_0)[\ol u]-\frac{1}{2}\overline \Gamma(\theta^*)[{\overline u}^{\otimes2}]+o_P(1)|u|\bigg\}.
	\eeas
	(Obviously, $o_P(1)|u|$ may be denoted by $o_P(1)$. However,  we write in this form for use in the proof of Theorem \ref{Selection2}.)
	Then  for any $(u, \tau) \in U_T \times \calt$ with $|u| \leq R$,
	\bea
	\bbZ_T(u, \tau)&=&\exp\bigg\{\ol\Delta_T(\theta^*, \tau_0)[\ol u]-\frac{1}{2}\overline \Gamma(\theta^*)[{\overline u}^{\otimes2}]\nonumber\\
	&&-\sum_{i \in \calj_1}\frak a_{i, T}\xi_{i, T}\frac{d}{dx}p_i(\theta_i^*)u_i-\sum_{k \in \calj_0} \frak b_{ T}^{-\frac{q_k\rho_k}{2}}\xi_{k, T} |u_k|^{q_k}+o_P(1)|u|\bigg\}.\label{S4Taylor}
	\eea }
Therefore,   the continuous random field $\bbV_T$ defined by (\ref{calv_T}) satisfies that
\beas
\sup_{ u \in U_T(R), \tau \in \calt}\big|\bbZ_T(u, \tau)-\bbV_T(u)\big| \overset{P}{\to}0.
\eeas
Moreover, for the continuous random field $\bbZ$ defined by (\ref{z}) and for any $R>0$,
\beas
\bbV_T \overset{d_s(\calg)}{\to } \bbZ \qquad in~C\big(\overline B_R\big).
\eeas
Thus, [{\bf A2}] holds.

\end{proof}

\begin{lemma}\label{p_j}
	Let $M$ be a positive number. Take an arbitrary positive number $\delta$  with $\delta<\frac{\delta_0}{4M} \wedge 1$, where $\delta_0$ is defined in  Condition (iii) in Section \ref{settings}.  Also, take an arbitrary $x_0 \in \bbR\setminus\{0\}$ with $|x_0| \leq M$.
	Then  for any $j \in \calj$ and any $x \in \bbR$ with $|x|\leq M$ and with $|x-x_0|<  \delta |x_0|$,
	\bea\label{lemma6.1formula}
	{p_j(x)}&\geq&		(1-C_1\delta){p_j(x_0)},
	\eea
	where $C_1$ is some positive constant depending only on $p_j$ $(j \in \calj)$  and $M$. $($Note that $(\ref{lemma6.1formula})$ holds for any $x \in \bbR$  if $x_0=0$ since $p_j(0)=0$.$)$
\end{lemma}
\begin{proof}
	Fix $j \in \calj$ and  $x \in \bbR$ with $|x|\leq M$ and with $|x-x_0|<  \delta |x_0|$. 
	First, assume that $x_0 \leq \frac{\delta_0}{2}$.  Since $\delta |x_0|<\frac{\delta_0}{4}$, we have $|x|\leq \delta_0$. Therefore, from the condition {\bf (iii)},
	\beas
	\frac{p_j(x)}{p_j(x_0)}&=&\frac{|x|^{q_j}}{|x_0|^{q_j}}\\
	&\geq&\frac{(|x_0| - \delta|x_0|)^{q_j}}{|x_0|^{q_j}}\\
	&=&(1-\delta)^{q_j}\\
	&\geq& 1-K_1 \delta, 
	\eeas
	where $K_1>0$ depends only on $q_j$ (or $p_j$). 
	
	Second, assume that $x_0 > \frac{\delta_0}{2}$. Since $\delta |x_0|<\frac{\delta_0}{4}$, we have $|x|\geq \frac{\delta_0}{4}$.  Therefore,
	\beas
	\frac{p_j(x)}{p_j(x_0)}&=&1-\frac{p_j(x_0)-p_j(x)}{p_j(x_0)}\\
	&\geq&1-\frac{\sup_{\frac{\delta_0}{4} \leq |y|\leq M}|\frac{d}{dy}p_j(y)|}{\inf_{\frac{\delta_0}{2}\leq |y|\leq M}p_j(y)}\\
	&=&1-K_2 \delta,
	\eeas
	where $K_2>0$ depends only on $p_j$  and $M$.   Then we take $C_1$ as $C_1=\max_{j \in \calj} (K_1 \vee K_2)$.

\end{proof}

\begin{proof}[Proof of Theorem \ref{Selection2}]
	Take an arbitrary positive sequence $\epsilon_T$ with $\epsilon_T\rightarrow 0$ as $T\rightarrow\infty$.
	It suffices to show that \begin{eqnarray*}P\big[\big|(\hat u_{k, T})_{k \in \calj_0}\big|\geq \epsilon_T \big] \rightarrow 0.
	\end{eqnarray*}
	Since $\big|(\hat u_{j, T})_{j \in \calj_0}\big|=o_P(1)$, we can take a positive sequence $\delta_T$ with $\delta_T\rightarrow 0$ as $T\rightarrow\infty$ such that
	\beas
	\varlimsup_{T\rightarrow\infty}P\big[\big|(\hat u_{j, T})_{j \in \calj_0}\big|\geq \delta_T \big]=0.
	\eeas
	Then from (\ref{S4Taylor}), 
	\beas
	&&\varlimsup_{T\rightarrow\infty}P\big[\big|(\hat u_{j, T})_{j \in \calj_0}\big|\geq \epsilon_T \big]{\yeq \varlimsup_{T\rightarrow\infty}P\big[\epsilon_T\leq \big|(\hat u_{j, T})_{j \in \calj_0}\big|<\delta_T\big]}
\\&=&{\varlimsup_{R\rightarrow\infty}\varlimsup_{T\rightarrow\infty}P\big[\epsilon_T\leq \big|(\hat u_{j, T})_{j \in \calj_0}\big|<\delta_T, |\hat u_T|\leq R \big]}
\\&\leq &{\varlimsup_{R\rightarrow\infty}\varlimsup_{T\rightarrow\infty}P\big[\hat u_T \in S(R, \delta_T), ~\big|(\hat u_{j, T})_{j \in \calj_0}\big|\geq \epsilon_T   \big]}
\\&\leq&	\varlimsup_{R\rightarrow\infty}\varlimsup_{T\rightarrow\infty}
	P\bigg[\sup_{\substack{u\in S(R, \delta_T), \\
			|\underline u|\geq\epsilon_T}}\bbZ_T(u, { \hat \tau_T})-\sup_{\substack{v\in I(R) \\ 
		}}\bbZ_T(v, { \hat \tau_T})\geq0\bigg] \\
	&=&\varlimsup_{R\rightarrow\infty}\varlimsup_{T\rightarrow\infty}
	P\bigg[\sup_{\substack{u\in S(R, \delta_T), \\
			 |\underline u|\geq\epsilon_T}}\inf_{\substack{v\in I(R) \\ 
		 }}\big\{\bbZ_T(u, { \hat \tau_T})-\bbZ_T(v, {\hat \tau_T} )\big\}\geq0\bigg] \\
	&=&\varlimsup_{R\rightarrow\infty}\varlimsup_{T\rightarrow\infty}
	P\bigg[\sup_{\substack{u\in S(R, \delta_T), \\
			|\underline u|\geq\epsilon_T}}\inf_{\substack{v\in I(R) \\ 
		}}\big\{O_P(1)|\overline u-\overline v|+o_P(1)|\underline u|-\sum_{k\in \calj_0}\xi_{k, T}\frak b_{ T}^{-\frac{q_k\rho_k}{2}}|u_k|^{q_k}\big\}\geq0\bigg], 
	\eeas
	where $O_P(1)$ and $o_P(1)$ depend on $R>0$. Since $\xi_{k, T}\frak b_{ T}^{-\frac{q_k\rho_k}{2}} \overset{d}{\to} d_k>0$ $(k \in \calj_0)$ from [{\bf H5}], we have 
	\beas
	\varlimsup_{\eta \to +0}\varlimsup_{T \to \infty}P\big[\xi_{k, T}\frak b_{ T}^{-\frac{q_k\rho_k}{2}} \leq \eta\big]=0.
	\eeas
	Therefore,
	\beas
	&&\varlimsup_{T\rightarrow\infty}P\big[\big|(\hat u_{j, T})_{j \in \calj_0}\big|\geq \epsilon_T \big]\\&\leq&
	\varlimsup_{\eta \to 0, \eta>0}\varlimsup_{R\rightarrow\infty}\varlimsup_{T\rightarrow\infty}
	P\bigg[\sup_{\substack{u\in S(R, \delta_T), \\
			|\underline u|\geq\epsilon_T}}\inf_{\substack{v\in I(R) \\ 
		}}\big\{O_P(1)|\overline u-\overline v|+o_P(1)|\underline u|\\
	&&-\eta\sum_{k\in \calj_0}|u_j|^{q_k}\big\}\geq0\bigg]\\
	&\leq&
	\varlimsup_{\eta \to 0, \eta>0}\varlimsup_{R\rightarrow\infty}\varlimsup_{T\rightarrow\infty}
	P\bigg[\sup_{\substack{u\in S(R, \delta_T), \\
			|\underline u|\geq\epsilon_T}}\inf_{\substack{v\in I(R) \\ 
		}}\big\{O_P(1)\frac{|\overline u-\overline v|}{\sum_{k\in \calj_0}|u_j|^{q_k}}+o_P(1)-\eta\big\}\geq0\bigg].
	\eeas
	Since  [{\bf S}] implies that  for any $R>0$,
	\begin{eqnarray*}
		\sup_{\substack{u\in S(R, \delta_T), \\ |\underline u|\geq\epsilon_T}}\inf_{v\in I(R)}\frac{|\overline u-\overline v|}{\sum_{k\in \calj_0}|u_j|^{q_k}}\rightarrow 0 \qquad(T\rightarrow\infty),
	\end{eqnarray*}
	we have $\ds \varlimsup_{T\to\infty}P\big[\big|(\hat u_{j, T})_{j \in \calj_0}\big|\geq \epsilon_T \big] = 0$. Thus, Theorem  \ref{Selection2} holds.

\end{proof}

\section{Proof of Theorems \ref{point1} and  \ref{point2},  and Proposition \ref{linearprop}}\label{proof of points} 
We first prepare some lemmas.
\begin{lemma}\label{O_P(1)ofNtilde}
	Let $\{N_t\}_{t\geq 0}$ be a counting 
	process 
	whose intensity is denoted by $\lambda_t$. 
	Let $\{X_t\}_{t \geq 0}$ be an $\bbR^\sfd$-valued predictable 
	process assumed to be  locally bounded.
	Let $G \subset \bbR^{\sf{g}}$ be a bounded open domain   admitting the Sobolev embedding. 
	Let $f : \bbR^\sfd \times \ol G\to \bbR^{\sf{f}}$ be a measurable map satisfying the following conditions.
	\bd
	\im[(i)] For each $x \in \bbR^\sfd$, $f(x, \cdot)$ is of class $C^1(\ol G)$.
	\im[(ii)] $\sup_{\gamma \in G}\big|\partial_\gamma^{ i} f(\cdot, \gamma)\big|$ ${(i=0, 1)}$ are bounded on every bounded set of $\bbR^{\sfd}$. 
	\ed 
	If for any $p, q \geq 1$ with $p \geq 2q $,
	\beas
	\sup_{\gamma \in G, t \geq 0}E\bigg[\big|\partial_\gamma^i f(X_t, \gamma)\big|^{p}  \lambda_t^{q}\bigg] < \infty \qquad(i=0, 1),
	\eeas  
	then 
	\beas
	\sup_{{\colr T >1}}E\bigg[\sup_{\gamma\in G}\bigg|\frac{1}{\sqrt T}\int_0^Tf(X_t, \gamma) d\tilde N_t\bigg|\bigg]< \infty,
	\eeas
	where $\tilde N_t = N_t-\int_0^t\lambda_s ds$.
\end{lemma}
\begin{proof}
	Let $T>1$.	Since the process $\sup_{\gamma \in G}\big|\partial_\gamma f(X, \gamma)\big|$ is locally  bounded,
	we have 
	\beas
	\partial_\gamma\int_0^Tf(X_t, \gamma) d\tilde N_t\yeq \int_0^T\partial_\gamma f(X_t, \gamma) d\tilde N_t\qquad(\gamma \in G).
	\eeas
	Take some constant $p$ with $p> \sf{g}$.
	From  Sobolev's inequality,
	\bea
	E\bigg[\sup_{\gamma \in G}\bigg|\frac{1}{\sqrt T}\int_0^Tf(X_t, \gamma) d\tilde N_t\bigg|\bigg] &\lesssim& E\bigg[\int_{G}\sum_{i=0, 1}\bigg|\frac{1}{\sqrt T}\int_0^T\partial_\gamma^i f(X_t, \gamma)  d\tilde N_t\bigg|^{p}d\gamma\bigg]\nonumber\\
	&\lesssim& \sup_{\gamma \in G, i=0,1}E\bigg[\bigg|\frac{1}{\sqrt T}\int_0^T\partial_\gamma^i f(X_t, \gamma)  d\tilde N_t\bigg|^{p}\bigg]\label{202212180545}.
	\eea
	Take some integer $k$ with $2^k \geq p$. By the Burkholder-Davis-Gundy inequality, for each $i=0, 1$,
	\beas
	E\bigg[\bigg|\frac{1}{\sqrt T}\int_0^T\partial_\gamma^i f(X_t, \gamma)  d\tilde N_t\bigg|^{2^k}\bigg]  &\lesssim& E\bigg[\bigg|\frac{1}{ T}\int_0^T\big|\partial_\gamma^i f(X_t, \gamma)\big|^2  d N_t\bigg|^{2^{k-1}}\bigg]\\
	&\lesssim& 
	E\bigg[\bigg|\frac{1}{ T}\int_0^T\big|\partial_\gamma^i f(X_t, \gamma)\big|^2  \lambda_tdt\bigg|^{2^{k-1}}\bigg]
	+E\bigg[\bigg|\frac{1}{ T}\int_0^T\big|\partial_\gamma^i f(X_t, \gamma)\big|^2  d \tilde N_t\bigg|^{2^{k-1}}\bigg]\\
	&\leq& 
	\sup_{t \geq 0}E\bigg[\bigg|\big|\partial_\gamma^i f(X_t, \gamma)\big|^2  \lambda_t\bigg|^{2^{k-1}}\bigg]
	+E\bigg[\bigg|\frac{1}{ \sqrt T}\int_0^T\big|\partial_\gamma^i f(X_t, \gamma)\big|^2  d \tilde N_t\bigg|^{2^{k-1}}\bigg].
	\eeas
	Repeating this evaluation, 
	\beas
	E\bigg[\bigg|\frac{1}{\sqrt T}\int_0^T\partial_\gamma^i f(X_t, \gamma)  d\tilde N_t\bigg|^{2^k}\bigg]  &\lesssim&
	\sum_{j=1}^{k-1}\sup_{t \geq 0}E\bigg[\bigg|\big|\partial_\gamma^i f(X_t, \gamma)\big|^{2^j}  \lambda_t\bigg|^{2^{k-j}}\bigg]
	+E\bigg[\bigg|\frac{1}{ \sqrt T}\int_0^T\big|\partial_\gamma^i f(X_t, \gamma)\big|^{2^{k-1}}  d \tilde N_t\bigg|^{2}\bigg]\\
	&\lesssim&
	\sum_{j=1}^{k-1}\sup_{t \geq 0}E\bigg[\bigg|\big|\partial_\gamma^i f(X_t, \gamma)\big|^{2^j}  \lambda_t\bigg|^{2^{k-j}}\bigg]
	+E\bigg[\frac{1}{ T}\int_0^T\big|\partial_\gamma^i f(X_t, \gamma)\big|^{2^{k}}  \lambda_t dt\bigg]\\
	&\leq&
	\sum_{j=1}^{k}\sup_{t \geq 0}E\bigg[\bigg|\big|\partial_\gamma^i f(X_t, \gamma)\big|^{2^j}  \lambda_t\bigg|^{2^{k-j}}\bigg].
	\eeas
	Thus, from (\ref{202212180545}),
	\beas
	\sup_{t\geq 0}E\bigg[\sup_{\gamma \in G}\bigg|\frac{1}{\sqrt T}\int_0^Tf(X_t, \gamma) d\tilde N_t\bigg|\bigg] 
	&\lesssim& \sup_{\gamma \in G, i=0,1}\sum_{j=1}^{k}\sup_{t \geq 0}E\bigg[\bigg|\big|\partial_\gamma^i f(X_t, \gamma)\big|^{2^j}  \lambda_t\bigg|^{2^{k-j}}\bigg] < \infty.
	\eeas
\end{proof}

\begin{lemma}\label{ergouniform}
	Assume the same conditions as {\rm Lemma} $\ref{O_P(1)ofNtilde}$. Also assume the ergodicity of $X$ as $(\ref{ergo1})$.  If
	\bea\label{ui-ergouniform}
	\bigg\{\sup_{\gamma \in G}|f(X_t, \gamma)|\bigg\}_{t \geq 0}  ~is~uniformly~integrable, 
	\eea
	then 
	\bea\label{ergouniformform}
	\sup_{\gamma \in G}\bigg|\frac{1}{T}\int_0^T   f(X_t, \gamma)  dt- \int_{\bbR^\sfd} f(x, \gamma) \nu(dx)\bigg|{\rightarrow} ~  0\quad in~L^1(dP)\qquad(T \rightarrow \infty).
	\eea
	In particular, if for any $p\geq 1$, 
	\bea\label{assum-ergouniform}
	\sup_{\gamma \in G, t\geq0}E\bigg[\big|\partial_\gamma^i f(X_t, \gamma)\big|^p\bigg]< \infty\qquad(i=0,1),
	\eea
	then $(\ref{ergouniformform})$ holds.
\end{lemma}
\begin{proof}
	Let $M>0$, and define a bounded function $f_M$ as $f_M(x, \gamma) = \big(f(x, \gamma)\wedge M\big)\vee (-M)$. Then
	\beas
	&&E\bigg[\sup_{\gamma \in G}\bigg|\frac{1}{T}\int_0^T   f(X_t, \gamma)  dt- \int_{\bbR^\sfd} f(x, \gamma) \nu(dx)\bigg|\bigg]\\
	&\leq&
	E\bigg[\sup_{\gamma \in G}\bigg|\frac{1}{T}\int_0^T   f(X_t, \gamma)  dt- \frac{1}{T}\int_0^T   f_M(X_t, \gamma)  dt\bigg|\bigg]
	+\sup_{\gamma \in G}\bigg|\int_{\bbR^\sfd} f(x, \gamma) \nu(dx)- \int_{\bbR^\sfd} f_M(x, \gamma) \nu(dx)\bigg|\\
	&&+E\bigg[\sup_{\gamma \in G}\bigg|\frac{1}{T}\int_0^T   f_M(X_t, \gamma)  dt- \int_{\bbR^\sfd} f_M(x, \gamma) \nu(dx)\bigg|\bigg].
	\eeas
	The first and second terms on the rightmost side are as small as we want by taking sufficiently large $M>0$
	since 
	\beas
	\sup_{\gamma \in G}\bigg| \int_{\bbR^\sfd} f(x, \gamma)1_{\big\{|f(x, \gamma) |\geq M\big\}} \nu(dx)\bigg| 
	&\leq & \int_{\bbR^\sfd}
	\sup_{\gamma\in G}|f(x, \gamma)
	|\, 1_{\big\{\sup_{\gamma \in G}|f(x, \gamma)|
		\geq M\big\}} \nu(dx)
	\\	&=&{\lim_{L \to \infty }\lim_{T \to \infty }E\bigg[ \frac{1}{T}\int_{0}^T
		L \wedge\sup_{\gamma\in G}|f(X_t, \gamma)
		|\, 1_{\big\{\sup_{\gamma \in G}|f(X_t, \gamma)|
			\geq M\big\}} dt\bigg] }
	\\	&\leq &\varlimsup_{T \to \infty }E\bigg[\frac{1}{T}\int_{0}^T
	\sup_{\gamma\in G}|f(X_t, \gamma)
	|\, 1_{\big\{\sup_{\gamma \in G}|f(X_t, \gamma)|
		\geq M\big\}} dt\bigg]
	\\&\leq& \sup_{t \geq 0}E\bigg[
	\sup_{\gamma\in G}|f(X_t, \gamma)
	|\, 1_{\big\{\sup_{\gamma \in G}|f(X_t, \gamma)|
		\geq M\big\}} \bigg].
	\eeas
	and since (\ref{ui-ergouniform}) holds.
	Also, let $\delta>0$ and take a finite set $G_\delta \subset G$  such that $\max_{\gamma_1, \gamma_2 \in G_\delta}|\gamma_1-\gamma_2|<\delta$. Then
	\beas
	&&E\bigg[\sup_{\gamma \in G}\bigg|\frac{1}{T}\int_0^T   f_M(X_t, \gamma)  dt- \int_{\bbR^\sfd} f_M(x, \gamma) \nu(dx)\bigg|\bigg] \\
	&\leq& E\bigg[\max_{\gamma \in G_\delta}\bigg|\frac{1}{T}\int_0^T   f_M(X_t, \gamma)  dt- \int_{\bbR^\sfd} f_M(x, \gamma) \nu(dx)\bigg|\bigg] \\
	&&+E\bigg[\frac{1}{T}\int_0^T  \sup_{\substack{\gamma_1, \gamma_2 \in G\\ |\gamma_1-\gamma_2|< \delta}}\big| f_M(X_t, \gamma_1)-  f_M(X_t, \gamma_2)\big|dt\bigg] +\int_{x\in \bbR^\sfd}  \sup_{\substack{\gamma_1, \gamma_2 \in G\\ |\gamma_1-\gamma_2|< \delta}}\big| f_M(x, \gamma_1)-  f_M(x, \gamma_2)\big|\nu(dx)
	\\&\overset{T \to \infty}{\to}& 2\int_{x\in \bbR^\sfd}  \sup_{\substack{\gamma_1, \gamma_2 \in G\\ |\gamma_1-\gamma_2|< \delta}}\big| f_M(x, \gamma_1)-  f_M(x, \gamma_2)\big|\nu(dx) ~\overset{\delta \to 0}{\to}~0.
	\eeas
	Thus,  (\ref{ergouniformform}) holds.
	
	Also, from  Sobolev's inequality, (\ref{ui-ergouniform}) holds if (\ref{assum-ergouniform}) holds.
\end{proof}

\begin{proof}[Proof of Theorem \ref{point1}]
	We  define a new parameter space $\Xi =\Theta \times \calt$ by
	\beas
	\Theta&=& [0, M_g]\times[-L_\beta, M_\beta]^{|\cala|}\times[0, M_\alpha]^{|\cala|}\times[0, M_\alpha]^{\sfa-|\cala|},\\
	\calt&=&[-L_\beta, M_\beta]^{\sfa-|\cala|},
	\eeas
	and also define new parameters $\theta=(\overline{\theta}, \underline\theta) \in \Theta$ and  $\tau \in \calt$ by
	\beas
	\overline \theta&=&\big(g, \, (\beta_i)_{i \in \cala} \,, \,\big(\alpha_{i})_{i \in\cala}  \big), \label{new1}\\
	\underline\theta&=&(\alpha_{k})_{k \in \cala^c}, \label{new2}\\
	\tau&=&(\beta_k)_{k\in\cala^c} \label{new3}.
	\eeas
	That is, we consider a parameter transformation as $(\theta, \tau)= \varphi(g, \alpha, \beta)$, where $\varphi : [0, M_g] \times [0, M_\alpha]^\sfa\times [-L_\beta, M_\beta]^\sfa \to \Xi$ is defined as
	\bea
	\varphi(g, \alpha, \beta)&=&\big(g, \, (\beta_i)_{i \in \cala} \,, \,\big(\alpha_{i})_{i \in\cala} \,, \, (\alpha_{k})_{k \in \cala^c} \,, \, (\beta_k)_{k\in\cala^c}   \big). \label{varphi}
	\eea
	Define estimators $\hat \theta_T$ and  $\hat \tau_T$ taking values in $\Theta$ and $\calt$, respectively, as
	\beas
	(\hat \theta_T, \hat\tau_T)\yeq \varphi(\hat g_T, \hat \alpha_T, \hat \beta_T).
	\eeas 
	We define $\sfp$ and $\sfp_1$ by $\sfp=1+|\cala|+\sfa$ and $\sfp_1=1+2|\cala|$, respectively. We also define $\calj_1$, $\calj_0$  and $\calj$  by $\calj_1=\{1\} \cup \big\{ 2+|\cala|, ..., \sfp_1\big\}$, $\calj_0=\{\sfp_1+1, ..., \sfp\}$ and $\calj=\calj_1 \cup \calj_0$, respectively.
	We also define one of the true values $\theta^*=(\overline\theta^*, \underline\theta^*) \in \bbR^{\sfp_1}\times \bbR^{\sfp-\sfp_1}$ as 
	\beas
	\overline\theta^*\yeq \big(g^*, \, (\beta_i^*)_{i \in \cala} \,, \,\big(\alpha_{i}^*)_{i \in\cala}  \big)\,, \qquad \underline\theta^*\yeq (0)_{k \in \cala^c}.
	\eeas
	For $\sfr:=\sfp$, let $\frak a_T={\rm diag}(\frak a_{1, T}, ..., \frak a_{\sfr, T})= T^{-\frac{1}{2}}I_\sfr$,\footnote{$I_\sfm$ denotes the $\sfm$-dimensional matrix} $\frak b_T := T$ and  $\rho_{k}=\frac{r}{q}>1$ $(k \in \calj_0)$. We  take $a_T \in GL(\sfp)$ as a deterministic diagonal matrix  defined by \beas
	(a_T)_{jj}=\begin{cases}T^{-\frac{1}{2}}& (j\in \{1, ..., \sfp_1\})\\
		T^{-\frac{r}{2q} }& (j\in \{\sfp_1+1, ..., \sfp\})\end{cases}.
	\eeas
	Define $U_T$ and $U$ by (\ref{U_T}) and (\ref{U}), respectively. Then from Example 2.4 in Yoshida and Yoshida \cite{yoshida2022quasi}, Condition [{\bf A3}] holds, and
	\beas
	U\yeq \bbR\times\bbR^{|\cala|}\times\bbR^{|\cala|}\times[0, \infty)^{\sfa-|\cala|} \subset \bbR^{\sfp}.
	\eeas
	Define $c_i$ $(i \in \calj_1)$ and $d_k$ $(k \in \calj_0)$ as
	\beas
	c_i=\begin{cases}
		\kappa_g1_{\{r=1\}}	& (i=1)\\
		\kappa_\alpha1_{\{r=1\}}	& otherwise
	\end{cases},\qquad d_k=\kappa_\alpha.
	\eeas
	Define a random field $\bbZ$ as 	
	\begin{eqnarray*}
		\bbZ(u)=\overline\Delta[\overline{u}]-\frac{1}{2}\overline \Gamma[{\overline u}^{\otimes2}] -q\sum_{i \in \calj_1}c_i|\theta^*_{i}|^{q-1}u_i-\sum_{k \in \calj_0}d_k|u_k|^q,
	\end{eqnarray*}
	for any $u=(u_1, ..., u_\sfp)\in\bbR^\sfp$, where $\overline u=(u_1, ..., u_{\sfp_1})$ and $\ol\Delta \sim N_{\sfp_1}\big(0, \ol\Gamma\big)$.
	We also define a $U$-valued random variable $\hat{u}$ by  
	\beas
	\hat u\yeq \big(\overline \Gamma^{-1}\overline\Delta^{\dagger}, ~0\big),
	\eeas
	where $0 \in \bbR^{\sfa-\sfb}$.
	Note that $\hat u$  becomes a unique maximizer of $\bbZ$ on $U$, and [{\bf A4}] holds.
	Define a random field $\calh_T : \Omega \times \Xi \to \bbR{\cup \{- \infty\}}$ as 
	\begin{eqnarray*}
		\calh_T( \theta, \tau )&=&\int_0^{T} \log\lambda_t\big(\phi(\theta, \tau)\big)dN_t-  \int_0^{T} \lambda_t\big(\phi(\theta, \tau)\big) dt,
	\end{eqnarray*} where   $\phi$ denotes  $\varphi^{-1}$ for $\varphi$ defined in (\ref{varphi}).
	Then $\hat\theta_T$ is a maximizer of
	\beas
	\calh_T(\theta, \hat\tau_T)-\sum_{j \in \calj}\xi_{j, T}p_{j}(\theta_j)
	\eeas on $\Theta$, where  for any $j \in \calj$,
	\bea
	\xi_{j, T}&=& \begin{cases} \kappa_g T^{\frac{r}{2}} & (j=1) \\
		\kappa_\alpha T^{\frac{r}{2}} & otherwise \end{cases}, \label{xi}\\
	p_{j}(x)&=&|x|^q \qquad(x \in \bbR).\nonumber
	\eea
	Note that  $q_j $ $(j \in \calj)$ defined in (iii) of  Section  \ref{settings} are determined as $q_j =q<1$.
	In the following,  we show [{\bf H1}]-[{\bf H5}] under [{\bf P1}]-[{\bf P3}], and  use 
	Theorem \ref{thmY2}.
	
	We first show [{\bf H1}]. From Taylor's series, for any $(\theta, \tau) \in \Xi$,
	\beas
	&&\calh_T(\theta, \tau)-\calh_T(\theta^*, \tau)
	\\&\leq&  \int_0^T\log\frac{\lambda_t\big(\phi(\theta, \tau)\big)}{\lambda_t^*}dN_t- 
	\int_0^{T} \big\{ \lambda_t\big(\phi(\theta, \tau)\big)-
	\lambda_t^*  \big\}dt,\\
	&=& \int_0^T\frac{\lambda_t\big(\phi(\theta, \tau)\big)-\lambda_t^*}{\lambda_t^*}dN_t
	-\int_0^1 (1-s)	\int_0^T\frac{\big\{\lambda_t\big(\phi(\theta, \tau)\big)-\lambda_t^*\big\}^2}{\big\{s\lambda_t\big(\phi(\theta, \tau)\big)+(1-s)\lambda_t^*\big\}^2}dN_t ds
	\\
	&&- 
	\int_0^{T} \big\{ \lambda_t\big(\phi(\theta, \tau)\big)-
	\lambda_t^*\big\}dt,\\
	&= & \int_0^T\frac{\lambda_t\big(\phi(\theta, \tau)\big)-\lambda_t^*}{\lambda_t^*}d\tilde N_t
	-\int_0^1 (1-s)	\int_0^T\frac{\big\{\lambda_t\big(\phi(\theta, \tau)\big)-\lambda_t^*\big\}^2}{\big\{s\lambda_t\big(\phi(\theta, \tau)\big)+(1-s)\lambda_t^*\big\}^2}dN_t ds,
	\\			
	\eeas
	where $\tilde N$ is a martingale defined by  $\tilde N_t=N_t -\int_{0}^t \lambda_s^*ds$. Take an arbitrary $R>0$. The integrand of the second term in the rightmost side is evaluated as
	for any  $(\theta, \tau) \in \Xi$, $0 \leq s \leq 1$ and  $0\leq t \leq T$, \beas
	\frac{\big\{\lambda_t\big(\phi(\theta, \tau)\big)-\lambda_t^*\big\}^2}{\big\{s\lambda_t\big(\phi(\theta, \tau)\big)+(1-s)\lambda_t^*\big\}^2} &\geq& \frac{\big\{\lambda_t\big(\phi(\theta, \tau)\big)-\lambda_t^*\big\}^2}{M_Rg^*}\varphi_R(X_t)
	\\&\geq& 
	\frac{\big\{\lambda_t\big(\phi(\theta, \tau)\big)-\lambda_t^*\big\}^2}{M_R\lambda_t^*}\varphi_R(X_t) \qquad\big(\because \lambda^*_t \geq g^*\big),
	\eeas
	where $\varphi_R : \bbR^{\sfa} \to [0, 1]$ is a continuous function vanishing outside of $[-R-1, R+1]^{\sfa}$ and satisfying $\varphi_R\equiv1$ on $[-R, R]^{\sfa}$,  and $M_R>0$ is a constant depending on $R$. Therefore,
	\beas
	\calh_T(\theta, \tau)-\calh_T(\theta^*, \tau)\leq \int_0^T\frac{\lambda_t\big(\phi(\theta, \tau)\big)-\lambda_t^*}{\lambda_t^*}d\tilde N_t -\frac{1}{2}	\int_0^T\frac{\big\{\lambda_t\big(\phi(\theta, \tau)\big)-\lambda_t^*\big\}^2}{M_R\lambda_t^*}\varphi_R(X_t)dN_t.
	\eeas
	
	In the following, we denote $\phi(\theta, \tau)$ by  $(g, \alpha, \beta)$. As (\ref{hexample}), we have
	\beas
	\lambda_t\big(\phi(\theta, \tau)\big)-\lambda_t^*=w(\beta, X_t)[h(\theta, \tau)] \qquad(t\geq 0, (\theta, \tau) \in \Xi),
	\eeas
	where $h : \bbR^\sfp \to \bbR^\sfp$ is  a continuous function defined as for any $(\theta, \tau) \in \bbR^\sfp$,
	\bea
	h(\theta, \tau)&=&\bigg( g+\sum_{k \in \cala^c}\alpha_k-g^*, \,\, 
	\big({\alpha_{i}}(\beta_{i}-\beta_{i}^*)\big)_{i\in \cala}\,,\,\, 
	(\alpha_{i}-\alpha_{i}^*)_{i \in \cala}, \,\, (\alpha_{k}\beta_{k})_{k \in \cala^c}\bigg).\label{h}
	\eea 
	Then 
	\beas
	&&\calh_T(\theta, \tau)-\calh_T(\theta^*, \tau)\\
	&\leq & \frac{1}{\sqrt T}\int_0^T\frac{w( \beta, X_t)}{\lambda_t^*}d\tilde N_t\big[\sqrt Th(\theta, \tau)\big]-\frac{1}{2M_RT}	\int_0^T\frac{ w(\beta, X_t)^{\otimes2}}{\lambda_t^*}\varphi_R(X_t)dN_t \big[\big(\sqrt Th(\theta, \tau)\big)^{\otimes2}\big]
	\\			&=&
	K_T(\theta, \tau)\big[\sqrt Th(\theta, \tau)\big]-\frac{1}{2}\big\{G(\theta, \tau)+r_T(\theta, \tau)\big\}\big[\big(\sqrt Th(\theta, \tau)\big)^{\otimes2}\big]
	\qquad(t\geq 0, (\theta, \tau) \in \Xi),
	\eeas
	where for any $(\theta, \tau) \in \Xi$,
	\beas
	K_T(\theta, \tau)&=&\frac{1}{\sqrt T}\int_0^T\frac{w( \beta, X_t)}{\lambda_t^*}d\tilde N_t,\\
	G(\theta, \tau)&=&\frac{1}{M_R}\int_{ \bbR^\sfa}w( \beta, x)^{\otimes2}\varphi_R(x)\nu(dx),\\
	r_T(\theta, \tau)&=&\frac{1}{M_R}\bigg\{\frac{1}{T}\int_0^T\frac{w( \beta, X_t)^{\otimes2}}{\lambda_t^*}\varphi_R(X_t)d\tilde N_t\\
	&&+ \frac{1}{T}\int_0^T w( \beta, X_t)^{\otimes2}\varphi_R(X_t)dt-\int_{\bbR^\sfa}w( \beta, x)^{\otimes2}\varphi_R(x)\nu(dx)\bigg\}.
	\eeas
	Since [{\bf P1}] implies that for any $p, q \geq 1$ with $p \geq 2q $, \beas
	\sup_{\beta \in [-L_\beta, M_\beta]^\sfa, t \geq 0}E\bigg[\bigg|\frac{\partial_\beta^i w( \beta, X_t)}{\lambda_t^*}\bigg|^{p}  (\lambda_t^*)^{q}\bigg] < \infty \qquad(i=0, 1)\qquad\big(\because \lambda_t^* \geq g^*\big),
	\eeas  
	we can use  Lemma \ref{O_P(1)ofNtilde}, and  obtain
	\begin{en-text}
		By Sobolev's inequality and  moment inequalities for martingales, for $p>\sfa$, we have 
		\beas
		&&E\bigg[\sup_{\beta\in[-L_\beta, M_\beta]^\sfa }\bigg|\frac{1}{\sqrt T}\int_0^T\frac{w(\beta, X_t)}{\lambda_t^*}d\tilde N_t\bigg|\bigg]\\
		&\leq& E\bigg[C_1\int_{[-L_\beta, M_\beta]^\sfa }\bigg|\frac{1}{\sqrt T}\int_0^T\frac{w(\beta, X_t)}{\lambda_t^*}d\tilde N_t\bigg|^p+\bigg|\frac{1}{\sqrt T}\int_0^T\frac{\partial_\beta w(\beta, X_t)}{\lambda_t^*}d\tilde N_t\bigg|^pd\beta\bigg]\\
		&\leq& C_2\sup_{t\geq 0, \beta \in [-L_\beta, M_\beta]^\sfa}E\bigg[\bigg|\frac{|w(\beta, X_t)|^2+\big|\partial_\beta w(\beta, X_t)\big|^2}{\lambda_t\big(\phi(\theta^*, \tau)\big)}\bigg|^p\bigg]
		\\&\leq& C_3\sup_{t\geq 0, \beta \in [-L_\beta, M_\beta]^\sfa}E\bigg[|w(\beta, X_t)|^{2p}+\big|\partial_\beta w(\beta, X_t)\big|^{2p}\bigg]
		\\&\leq& C_4 \sup_{t\geq 0, \beta \in [-L_\beta, M_\beta]^\sfa}\bigg(1+\sum_{l=0, 1, 2, j=1, ..., \sfa}E\bigg[\bigg|\frac{d^l}{d\beta_j^l} e^{\beta_j X^j_t}\bigg|^{2p}\bigg] \bigg)
		\\
		&<&\infty  \qquad\big(\because [{\bf P1}]\big),
		\eeas where $C_1, C_2, C_3, C_4$ are positive constants.  
	\end{en-text}
	\bea\label{Ktight}
	\sup_{\theta, \tau) \in \Xi}\big|K_T(\theta, \tau)\big|=O_P(1).
	\eea
	Similarly, $\ds \sup_{\beta\in[-L_\beta, M_\beta]^\sfa }\bigg|\frac{1}{ \sqrt T}\int_0^T\frac{w(\beta, X_t)^{\otimes2 }}{\lambda_t^*}\varphi_R(X_t)d\tilde N_t\bigg|=O_P(1)$.
	Since for any $p>1$,
	\beas
	\sup_{\beta\in[-L_\beta, M_\beta]^\sfa, t\geq0}E\bigg[\bigg|\partial_\gamma^i \big\{w(\beta, X_t)^{\otimes2 }\big\}\varphi_R(X_t)\bigg|^p\bigg]< \infty \qquad(i=0, 1),
	\eeas  
	we can use  Lemma \ref{ergouniform}, and  obtain	\beas
	\frac{1}{T}\int_0^Tw( \beta, X_t)^{\otimes2}\varphi_R(X_t)dt-\int_{\bbR^\sfa}w( \beta, x)^{\otimes2}\varphi_R(x)\nu(dx) \overset{P}{\to} 0,
	\eeas
	uniformly in  $\beta \in  [- L_\beta, M_\beta]^\sfa$ as $T \to \infty$.  Therefore, 
	\beas
	\sup_{(\theta, \tau) \in \Xi}\big|r_T(\theta, \tau)\big|=o_P(1).
	\eeas 
	Moreover, [{\bf P2}] implies the non-degeneracy of $G$ for sufficiently large $R$. Thus, [{\bf H1}] holds for $\frak a_T={T^{-\frac{1}{2}}I_\sfr}$.
	
	
	[{\bf H2}]  also holds. {Indeed, continuing to denote $\phi(\theta, \tau)$ by $(g, \alpha, \beta)$,  we have 
		\beas
		\Theta^*&=&\{\theta \in \Theta ; \, \exists \tau \in \calt,  h(\theta, \tau)=0\}\\
		&=&\bigg\{\theta \in \Theta ; \, g+\sum_{k\in \cala^c}\alpha_{k}=g^*,~ \big(\beta_i -\beta_{i}^*\big)_{i \in \cala}=0, ~
		\big(\alpha_{i}-\alpha_{i}^*\big)_{i \in \cala}=0\bigg\},
		\eeas
		where $h$ is given by (\ref{h}). Therefore, $\Theta^* \cap \{\underline \theta=0\}=\Theta^* \cap \{\alpha_{k}=0 ~(k \in \calj_0)\}=\{\theta^*\}$. Condition {\bf (a)} of [{\bf H2}] obviously holds since $h$ can be smoothly extended on $\bbR^\sfp$. Condition {\bf (b)}  of [{\bf H2}] holds since $\alpha_i>2^{-1}\alpha_i^*>0$ $(i \in \cala)$ if $\theta$ is close to $\theta^*$.
	}
	
	We show [{\bf H3}]. From Remark \ref{remid}, we only need to show $[{\bf H3}]^{\prime}$. \begin{en-text}Since $h$ is given by (\ref{h}), we have
		\beas
		\Theta^*&:=&\{\theta \in \Theta ; \, \exists \tau \in \calt,  h(\theta, \tau)=0\}\\
		&=&\bigg\{\theta \in \Theta ; \, g+\sum_{k\in \cala^c}\alpha_{k}=g^*,~ \big(\beta_i -\beta_{i}^*\big)_{i \in \cala}=0, ~
		\big(\alpha_{k}-\alpha_{k}^*\big)_{k \in \cala^c}=0\bigg\}.
		\eeas
	\end{en-text}
	Since $\kappa_g<\kappa_\alpha$ from [{\bf P3}], for any $\theta \in \Theta^*$,
	\beas
	\kappa_g|g|^q+\kappa_\alpha\sum_{j =1}^{\sfa}|\alpha_j|^q	&=& \kappa_g\bigg|g^*-\sum_{k\in \cala^c}\alpha_{k}\bigg|^q+\kappa_\alpha\sum_{k \in \cala^c}|\alpha_{k}|^q+\kappa_\alpha\sum_{i\in\cala}|\alpha^*_{i}|^q\\
	&\geq&\kappa_g|g^*|^q-\kappa_g\sum_{k \in \cala^c}|\alpha_{k}|^q+\kappa_\alpha\sum_{k \in \cala^c}|\alpha_{k}|^q+\kappa_\alpha\sum_{i\in \cala}|\alpha^*_{i}|^q\\
	&\geq&\kappa_g|g^*|^q+\kappa_\alpha\sum_{i \in \cala}|\alpha^*_{i}|^q,
	\eeas
	where the equations hold if and only if $\alpha_{k}=0$ $(k \in \cala^c)$.
	Therefore, under $\theta \in \Theta^*$, the penalty term is minimized if and only if  $\alpha_{k}=0$ $(k \in \cala^c)$, that is, if  and only if 
	$\theta=\theta^*$. Thus, $[{\bf H3}]^{\prime}$ holds.

	
	Condition [{\bf H4}] also holds for $\xi_{j, T}$ defined in (\ref{xi}) since $\kappa_\alpha>0$ {and $q_j=q<1$ $(j \in \calj)$}.
	Finally, we show [{\bf H5}]. Let $\calg=\{\phi, \Omega\}$, and take  $\caln$ as
	\beas
	\caln \yeq {\rm Int}(\Theta)\cap \big\{(\theta_1, ..., \theta_\sfp) \in \bbR^\sfp ; \theta_1 > 2^{-1}g^*\big\}.
	\eeas 
		Note that {$\caln$ satisfies (\ref{caln1}) and (\ref{caln2})} and that for any $(\theta, \tau) \in \ol \caln \times \calt$ and any $t\geq 0$, \bea\label{g*/2}
		\lambda_t\big(\phi(\theta, \tau)\big) \geq  2^{-1}g^*.
		\eea 
	For any $(\theta, \tau)\in \ol\caln \times \calt$,
	\beas
	\partial_\theta\calh_T(\theta^*, \tau)&=&\bigg(\int_{0}^T\frac{v(X_t)}{g^*+\sum_{i \in \cala}\alpha_i^*e^{-\beta_i^*X^i_t}}d\tilde N_t,  \int_{0}^T\frac{\big(e^{\tau_k X_t^{k}}\big)_{k \in \cala^c}}{g^*+\sum_{i \in \cala}\alpha_i^*e^{-\beta_i^*X^i_t}}d\tilde N_t,\bigg)^\prime.
	\eeas
	Similarly as before, from   [{\bf P1}] and Lemma \ref{O_P(1)ofNtilde},
	\beas
	\sup_{\tau \in\calt}\bigg|\frac{1}{\sqrt T}\int_{0}^T\frac{\big(e^{\tau_k X_t^{k}}\big)_{k \in \cala^c}}{g^*+\sum_{i \in \cala}\alpha_i^*e^{-\beta_i^*X^i_t}}d\tilde N_t\bigg|=O_P(1).
	\eeas
	Moreover, from the martingale central limit theorem, we obtain
	\beas
	\frac{1}{\sqrt{T}}\int_{0}^T\frac{v(X_t)}{g^*+\sum_{i \in \cala}\alpha_i^*e^{-\beta_i^*X^i_t}}d\tilde N_t \overset{d}{\to} \overline\Delta,
	\eeas
	{checking Lindeberg's condition as  for $S_t=S_t(T) :=\frac{1}{\sqrt{T}}\int_{0}^t\frac{v(X_s)}{g^*+\sum_{i \in \cala}\alpha_i^*e^{-\beta_i^*X^i_s}}d\tilde N_s$ and  for any $a>0$,
		\beas
		E\bigg[\sum_{t \leq T} \big(\Delta S_t\big)^21_{\{|\Delta S_t|>a\}}\bigg] &\leq&a^{-1}E\bigg[\sum_{t \leq T} \big|\Delta S_t\big|^3\bigg]
		\\&\leq&a^{-1}E\bigg[\int_0^T \bigg|\frac{1}{\sqrt T}\frac{v(X_t)}{\lambda_t^*}\bigg|^3 dN_t\bigg]
		\\&\leq &a^{-1}T^{-\frac{1}{2}}\sup_{t\geq 0}E\bigg[\bigg|\frac{v(X_t)}{\lambda_t^*}\bigg|^3 \lambda^*_t\bigg]  ~\to~ 0 \qquad(T\to \infty). 
		\eeas
	}
	Thus, the first half of the argument of [{\bf X5}] holds.
	Also, for any $(\theta, \tau)\in \ol \caln \times \calt$,
	\beas
	\partial_\theta^2\calh_T(\theta, \tau)&=&-\int_{0}^T\frac{\partial_\theta\lambda_t\big(\phi(\theta, \tau)\big)^{\otimes2}}{\lambda_t\big(\phi(\theta, \tau)\big)^2}dN_t+\int_{0}^T\frac{\partial_\theta^2\lambda_t\big(\phi(\theta, \tau)\big)}{\lambda_t\big(\phi(\theta, \tau)\big)}dN_t\\
	&&-\int_{0}^T\partial_\theta^2\lambda_t\big(\phi(\theta, \tau)\big)dt\\
	&=&-\int_{0}^T\frac{\partial_\theta\lambda_t\big(\phi(\theta, \tau)\big)^{\otimes2}}{\lambda_t\big(\phi(\theta, \tau)\big)^2}\lambda_t\big(\phi(\theta^*, \tau)\big)dt-\int_{0}^T\frac{\partial_\theta\lambda_t\big(\phi(\theta, \tau)\big)^{\otimes2}}{\lambda_t\big(\phi(\theta, \tau)\big)^2}d\tilde N_t\\
	&&+\int_{0}^T\frac{\partial_\theta^2\lambda_t\big(\phi(\theta, \tau)\big)}{\lambda_t\big(\phi(\theta, \tau)\big)}d\tilde N_t+\int_{0}^T\frac{\partial_\theta^2\lambda_t\big(\phi(\theta, \tau)\big)}{\lambda_t\big(\phi(\theta, \tau)\big)}\lambda_t\big(\phi(\theta^*, \tau)\big)dt\\
	&&-\int_{0}^T\partial_\theta^2\lambda_t\big(\phi(\theta, \tau)\big)dt.
	\eeas
	Similarly as before,  {\colorr from  [{\bf P1}] and (\ref{g*/2})}, we can use Lemma \ref{O_P(1)ofNtilde}, and obtain 
	\beas
	\sup_{(\theta, \tau) \in \ol \caln \times \calt}\bigg|-\frac{1}{\sqrt T}\int_{0}^T\frac{\partial_\theta\lambda_t\big(\phi(\theta, \tau)\big)^{\otimes2}}{\lambda_t\big(\phi(\theta, \tau)\big)^2}d\tilde N_t+\frac{1}{\sqrt T}\int_{0}^T\frac{\partial_\theta^2\lambda_t\big(\phi(\theta, \tau)\big)}{\lambda_t\big(\phi(\theta, \tau)\big)}d\tilde N_t\bigg| =O_P(1).
	\eeas
	Therefore, for any $(\theta, \tau)\in \ol \caln \times \calt$,
	\beas
	\frac{1}{T}\partial_\theta^2\calh_T(\theta, \tau)
	&=&-\frac{1}{T}\int_{0}^T\frac{\partial_\theta\lambda_t\big(\phi(\theta, \tau)\big)^{\otimes2}}{\lambda_t\big(\phi(\theta, \tau)\big)^2}\lambda_t\big(\phi(\theta^*, \tau)\big)dt+\frac{1}{T}\int_{0}^T\frac{\partial_\theta^2\lambda_t\big(\phi(\theta, \tau)\big)}{\lambda_t\big(\phi(\theta, \tau)\big)}\lambda_t\big(\phi(\theta^*, \tau)\big)dt\\
	&&-\frac{1}{T}\int_{0}^T\partial_\theta^2\lambda_t\big(\phi(\theta, \tau)\big)dt+o_P(1).
	\eeas
	From  [{\bf P1}] and (\ref{g*/2}),    we can apply Lemma \ref{ergouniform} to 
	\beas
	f(\gamma, X_t) \yeq \frac{\partial_\theta\lambda_t\big(\phi(\theta, \tau)\big)^{\otimes2}}{\lambda_t\big(\phi(\theta, \tau)\big)^2}\lambda_t\big(\phi(\theta^*, \tau)\big)+\frac{\partial_\theta^2\lambda_t\big(\phi(\theta, \tau)\big)}{\lambda_t\big(\phi(\theta, \tau)\big)}\lambda_t\big(\phi(\theta^*, \tau)\big)
	-\partial_\theta^2\lambda_t\big(\phi(\theta, \tau)\big),
	\eeas  where $\gamma=(\theta, \tau) \in  \ol \caln \times \calt$.
	Therefore,  
	for any $R>0$,
	\beas
	&&\sup_{\substack{(\theta, \tau)\in \ol\caln\times\calt\\ |a_T^{-1}(\theta-\theta^*)|\leq R}}\bigg|a_T^\prime\partial^2_\theta\calh_T(\theta, \tau)a_T+ \begin{pmatrix}
		\overline \Gamma(\theta^*) & O\\
		O & O
	\end{pmatrix} \bigg|\overset{P}{\rightarrow}0.
\end{eqnarray*}
Thus,  the second half of the argument of [{\bf H5}] also holds. Then using Theorem \ref{thmY2}, we have
\beas
a_T^{-1}(\hat\theta_T-\theta^*) \overset{d}{\to} \hat u.
\eeas
This implies (\ref{weakcon1inex1}) and (\ref{weakcon2ex1}).

Also, we see easily that [{\bf S}] holds from Example  \ref{exeasysparse}. Therefore, using Theorem \ref{Selection2}, we have 
\beas
\lim_{T\to \infty}P\big[\hat \alpha_{k, T}=0 ~(k \in \cala^c)\big] =1.
\eeas

\end{proof}

\begin{proof}[Proof of Theorem \ref{point2}]
From (\ref{Gnonde2}), we can take some large $L>0$ such that the following matrix is non-degenerate:
\bea\label{Rlarge}
\int_{ [0, \infty)^\sfa}\big\{( x_j)_{j \in D}\big\}^{\otimes2}\varphi_L(x)\,\big\{\ftrue\cdot x 1_{\{\ftrue\cdot x >0\}}+1_{\{\ftrue \cdot x =0\}}\big\}\nu(dx),
\eea
where $\varphi_L : \bbR^{\sfa} \to [0, 1]$ is a continuous function vanishing outside of $[-L-1, L+1]^{\sfa}$ and satisfying $\varphi_L\equiv1$ on $[-L, L]^{\sfa}$.
Choose some constant $M_L>1$ satisfying that
\bea\label{ML}
\lambda_t(\alpha)\varphi_L(X_t) \leq M_L  \qquad\big(t\geq 0, \, \alpha \in [0, M_\alpha]^\sfa\big).
\eea
We define  an estimating function $\widetilde{\mathbb{\Psi}}_T$  as
\beas
\widetilde{\mathbb{\Psi}}_T(\alpha)  \yeq  \mathbb{\Psi}_T(\alpha) + \int_0^T\bigg\{\lambda_t(\alpha)-\frac{\{\lambda_t(\alpha)\}^2}{2M_L}\varphi_L(X_t)\bigg\}1_{\{\lambda^*_t=0\}}dt \qquad\big( \alpha \in [0, M_\alpha]^\sfa\big).
\eeas
Note that $\widetilde{\mathbb{\Psi}}_T \geq \mathbb{\Psi}_T$.
For each $T$, let $\widetilde\alpha_T=(\widetilde\alpha_{1, T}, ..., \widetilde\alpha_{\sfa, T})$ be an arbitrary  $[0, M_\alpha]^\sfa$-valued random variable that asymptotically maximizes $\widetilde{\mathbb{\Psi}}_T(\alpha)$.

In order to show Theorem \ref{point2}, it is sufficient to show that under $[{\bf L1}]$ and $[{\bf L2}]$,
\bea\label{suffices1}
\big(T^{\frac{1}{2}}(\widetilde \alpha_{i, T}-
\true_{i})_{i \in \calj_1}, T^{\frac{r}{2q}}(\widetilde\alpha_{k, T})_{k \in \calj_0}\big)\overset{d}{\to} \big(\overline\Gamma^{-1}\overline\Delta^{\dagger}, 0\big)
\eea
and
\bea\label{suffices2}
\lim_{T\to \infty}P\big[\widetilde\alpha_{k, T}=0 ~~(k \in \calj_0), ~~\widetilde\alpha_{i, T}\neq 0 ~~(i \in \calj_1)\big]=1.
\eea
In fact, assume (\ref{suffices1}) and  (\ref{suffices2})  for any asymptotic maximizer $\widetilde \alpha_T$ of $\widetilde{\mathbb{\Psi}}_T$. Then   since $\lambda^*_t=\sum_{i \in \calj_1}\true_i X_t^i~a.s.$ and therefore $1_{\{\lambda^*_t=0\}}=1_{\{ X_t^i=0 ~(i \in \calj_1)\}}~a.s.$,
we have
\beas
P\big[\widetilde{\mathbb{\Psi}}_T(\widetilde\alpha_T)  =  \mathbb{\Psi}_T(\widetilde\alpha_T) \big] &=& P\bigg[\int_0^T\bigg\{\lambda_t(\widetilde\alpha_T)-\frac{\{\lambda_t(\widetilde\alpha_T)\}^2}{2M_L}\varphi_L(X_t)\bigg\}1_{\big\{X_t^i=0~(i \in \calj_1)\big\}}dt~=0\bigg]\\
&\geq& P\big[\widetilde\alpha_{k, T}=0 ~~(k \in \calj_0)\big] \to 1 \quad(T\to \infty) \qquad\big(\because (\ref{suffices2})\big).
\eeas
Therefore, the following evaluation  asymptotically holds:
\beas
\widetilde{\mathbb{\Psi}}_T(\hat \alpha_T)  ~\leq~ \widetilde{\mathbb{\Psi}}_T(\widetilde\alpha_T) \yeq {\mathbb{\Psi}}_T(\widetilde\alpha_T)  ~\leq~ \mathbb{\Psi}_T(\hat \alpha_T).
\eeas 
Thus, together with the inequality $\widetilde{\mathbb{\Psi}}_T \geq \mathbb{\Psi}_T$, 
$\widetilde{\mathbb{\Psi}}_T(\hat \alpha_T)$ is asymptotically equal to $\widetilde{\mathbb{\Psi}}_T(\widetilde\alpha_T)$. This means that $\hat \alpha_T$ also asymptotically maximizes $\widetilde{\mathbb{\Psi}}_T$. Then (\ref{suffices1}) and  (\ref{suffices2})  holds when substituting  $\hat \alpha_T$ for $\widetilde\alpha_T$. Therefore, Theorem \ref{point2} holds.

In the following, we show (\ref{suffices1}) and (\ref{suffices2}) under [{\bf L1}] and [{\bf L2}]. We  consider two parameter spaces $\Theta$ and $\calt$ by
\beas
\Theta= [0, M_\alpha]^\sfa,~~\calt=\{1\},
\eeas
and we consider new parameters $\theta \in \Theta$ and  $\tau \in \calt$ by
\beas
\theta=\alpha, ~~\tau=1.
\eeas 
{Define estimators $\hat \theta_T$ and  $\hat \tau_T$ taking values in $\Theta$ and $\calt$, respectively, as
$\hat \theta_T=\widetilde \alpha_T, ~\hat\tau_T= 1$. In the following, we omit $\tau$.}
Also, define  $\theta^* \in \Theta$ as $\theta^*=\true$.
We take $\sfp$  and $\sfp_1$ as  $\sfp=\sfa$ and $\sfp_1=|\calj_1|$, respectively.
For notational simplicity,  assume that $\calj_1=\{1, ..., \sfp_1\}$ and $\calj_0 = \{\sfp_1+1, ..., \sfp\}$.
Define $\frak a_T={\rm diag}(\frak a_{1, T}, ..., \frak a_{\sfr, T})$ as $\frak a_T= T^{-\frac{1}{2}}I_{\sfr}$.  ($\sfr $ is  already defined as $\sfr=|D|$ in Section \ref{pointprocesswithmulticolinear}.)
Let $\frak b_{T}:=T$ and  $\rho_{k}:=\frac{r}{q}>1$ $(k\in\calj_0)$, and take $a_T \in GL(\sfp)$ as a deterministic diagonal matrix  defined by \bea\label{aTex2}
(a_T)_{jj}=\left\{\begin{array}{ll}T^{-\frac{1}{2}} &\big(j \in \calj_1=\{1, ..., \sfp\}\setminus \calj_0 \big)\\  T^{-\frac{r}{2q}}&(j \in \calj_0)\end{array}\right..
\eea
Define $U_T$ and $U$ by (\ref{U_T}) and (\ref{U}), respectively. Then from Example 2.4 in Yoshida and Yoshida \cite{yoshida2022quasi},  Condition [{\bf A3}] holds, and
\beas
U\yeq \big\{u=(u_1, ..., u_\sfp) \in \bbR^\sfp  ; u_k \geq 0 ~(k \in \calj_0)\big\}.
\eeas
Define $c_i \in \bbR$ $(i \in \calj_1)$ and $d_k \in \bbR$ $(k \in \calj_0)$ as
\beas
c_i=\kappa_i1_{\{r=1\}},~~ d_k=\kappa_k,
\eeas
respectively. Define a random field $\bbZ$ as 	
\begin{eqnarray*}
\bbZ(u)=\overline\Delta[(u_i)_{i \in \calj_1}]-\frac{1}{2}\overline \Gamma\big[\big((u_i)_{i \in \calj_1}\big)^{\otimes2}\big] -q\sum_{i \in \calj_1}c_i|\true_{i}|^{q-1}u_i-\sum_{k \in \calj_0}d_k|u_k|^q,
\end{eqnarray*}
for any $u=(u_1, ..., u_\sfp)\in\bbR^\sfp$.
We also define a $U$-valued random variable $\hat{u}$ by  
\beas
\hat u\yeq \big(\overline \Gamma^{-1}\overline\Delta^{\dagger}, ~0\big),
\eeas
Note that with probability $1$, $\hat u$  becomes a unique maximizer of $\bbZ$ on $U$, and [{\bf A4}] holds.
Define a random field $\calh_T : \Omega \times \Theta \to \bbR \cup \{-\infty\}$ as 
\begin{eqnarray*}
\calh_T( \alpha)&=&\int_0^{T} \log\lambda_t(\alpha)dN_t-  \int_0^{T} \lambda_t(\alpha) dt+\int_0^T\bigg\{\lambda_t(\alpha)-\frac{\{\lambda_t(\alpha)\}^2}{2M_L}\varphi_L(X_t)\bigg\}1_{\{\lambda^*_t=0\}}dt. \\
\end{eqnarray*}
Then  the estimation function $\widetilde{\mathbb{\Psi}}_T$ can be expressed as
\beas
\widetilde{\mathbb{\Psi}}_T(\alpha)  &=& \calh_T(\alpha)  -T^{\frac{r}{2}} \sum_{j=1}^\sfa \kappa_j\alpha_j^q\\
&=&\calh_T(\alpha)-\sum_{j =1}^\sfp\xi_{j, T}p_{j}(\alpha_j) \qquad(\alpha \in \Theta),
\eeas
where  for any $j  \in \calj := \calj_1\cup\calj_0=\{1, ..., \sfp\}$,
\bea
\xi_{j, T}&=&  \kappa_j T^{\frac{r}{2}}, \label{xiex2}\\
p_{j}(x)&=&|x|^q \qquad(x \in \bbR).\nonumber
\eea
Note that $q_j$ $(j \in \calj)$  defined in (iii) of Section \ref{settings} are determined as $q_j = q < 1$.

From Theorem \ref{thmY2}, if [{\bf H1}]-[{\bf H5}] holds for $\widetilde{\mathbb{\Psi} }_T$ under ${\bf [L1]}$ and  ${\bf [L2]}$, then 
\beas
a_T^{-1} (\widetilde \alpha_T-\true)\overset{d}{\to}		 	\hat u,
\eeas
i.e. (\ref{suffices1}) holds.
Therefore, in the following, we show [{\bf H1}]-[{\bf H5}]  under [{\bf L1}] and [{\bf L2}].

We first show [{\bf H1}]. 
Since  $\int_0^T 1_{\{\lambda^*_t =0\}}dN_t =0 ~a.s.$,  we  have $\int_0^{T} \log\lambda_t(\alpha)dN_t= \int_0^{T} \log\lambda_t(\alpha)1_{\{\lambda^*_t>0\}}dN_t~a.s.$.
Then from Taylor's series,
\beas
&&\calh_T(\alpha)-\calh_T(\true)\\
&=& \int_0^T\log\lambda_t(\alpha)1_{\{\lambda^*_t >0\}}dN_t- 
\int_0^{T}  \lambda_t(\alpha)dt  +\int_0^T\bigg\{\lambda_t(\alpha)-\frac{\{\lambda_t(\alpha)\}^2}{2M_L}\varphi_L(X_t)\bigg\}1_{\{\lambda^*_t=0\}}dt \\
&&-\int_0^T\log\lambda_t^*1_{\{\lambda^*_t >0\}}dN_t+ 
\int_0^{T}  \lambda_t^*dt 
\\&=& \int_0^T\log\frac{\lambda_t(\alpha)}{\lambda_t^*}1_{\{\lambda^*_t >0\}}dN_t- 
\int_0^{T} \big\{ \lambda_t(\alpha)-
\lambda_t^*  \big\}1_{\{\lambda^*_t >0\}}dt-\int_0^T\frac{\{\lambda_t(\alpha)\}^2}{2M_L}\varphi_L(X_t)1_{\{\lambda^*_t=0\}}dt \\
&=& \int_0^T\frac{\lambda_t(\alpha)-\lambda_t^*}{\lambda_t^*}1_{\{\lambda^*_t >0\}}dN_t
-\int_0^1 (1-s)	\int_0^T\frac{\big\{\lambda_t(\alpha)-\lambda_t^*\big\}^2}{\big\{s\lambda_t(\alpha)+(1-s)\lambda_t^*\big\}^2}1_{\{\lambda^*_t >0\}}dN_t ds
\\
&&- 
\int_0^{T} \big\{ \lambda_t(\alpha)-
\lambda_t^*  \big\}1_{\{\lambda^*_t >0\}}dt-\int_0^T\frac{\{\lambda_t(\alpha)\}^2}{2M_L}\varphi_L(X_t)1_{\{\lambda^*_t=0\}}dt \\
&=& \int_0^T\frac{\lambda_t(\alpha)-\lambda_t^*}{\lambda_t^*}1_{\{\lambda^*_t >0\}}d\tilde N_t
-\int_0^1 (1-s)	\int_0^T\frac{\big\{\lambda_t(\alpha)-\lambda_t^*\big\}^2}{\big\{s\lambda_t(\alpha)+(1-s)\lambda_t^*\big\}^2}1_{\{\lambda^*_t >0\}}dN_t ds\\
&&-\int_0^T\frac{\{\lambda_t(\alpha)\}^2}{2M_L}\varphi_L(X_t)1_{\{\lambda^*_t=0\}}dt.
\eeas
where $\tilde N$ is a martingale defined by  $\tilde N_t=N_t -\int_{0}^t \lambda_s^*ds$.
The integrand of the second term in the rightmost side is evaluated as
for any  $\alpha \in \Theta$, $0 \leq s \leq 1$ and  $0\leq t \leq T$, \beas
\frac{\big\{\lambda_t(\alpha)-\lambda_t^*\big\}^2}{\big\{s\lambda_t(\alpha)+(1-s)\lambda_t^*\big\}^2}1_{\{\lambda^*_t >0\}} &\geq& \frac{\big\{\lambda_t(\alpha)-\lambda_t^*\big\}^2}{M_L^2}\varphi_L(X_t)1_{\{\lambda^*_t >0\}}\qquad\big(\because (\ref{ML})\big).
\eeas
Therefore, noting that $M_L^2 > M_L$, we have
\beas
&&\calh_T(\alpha)-\calh_T(\ftrue)\\
&\leq&\int_0^T\frac{\lambda_t(\alpha)-\lambda_t^*}{\lambda_t^*}1_{\{\lambda^*_t >0\}} d\tilde N_t -\frac{1}{2M_L^2}	\int_0^T\big\{\lambda_t(\alpha)-\lambda_t^*\big\}^2\varphi_L(X_t)1_{\{\lambda^*_t >0\}}dN_t
\\
&&-\int_0^T\frac{\{\lambda_t(\alpha)\}^2}{2M_L^2}\varphi_L(X_t)1_{\{\lambda^*_t=0\}}dt
\\
&=& \int_0^T\frac{\lambda_t(\alpha)-\lambda_t^*}{\lambda_t^*}1_{\{\lambda^*_t >0\}} d\tilde N_t -\frac{1}{2M_L^2}
\int_0^T\big\{\lambda_t(\alpha)-\lambda_t^*\big\}^2\varphi_L(X_t)\big\{1_{\{\lambda^*_t >0\}}dN_t+1_{\{\lambda^*_t =0\}}dt\big\}.
\eeas

From (\ref{lambdaforh}), we have  
\beas
\lambda_t(\alpha)-\lambda_t^*=(X_t^j)_{j \in D}[h(\alpha)] \qquad(t\geq 0, \alpha \in [0, M_\alpha]^\sfa), 
\eeas 
where $h=(h_1, ..., h_\sfr) : \Theta  \to \bbR^{\sfr}$ is  a continuous function defined by
\bea
h(\alpha)=(\alpha-\ftrue)A=(\alpha-\true)A \qquad(\alpha \in [0, M_\alpha]^\sfa).\label{h2}
\eea 
Then 
\beas
&&\calh_T(\alpha)-\calh_T(\alpha^*)\\
&\leq & \frac{1}{\sqrt T}\int_0^T\frac{( X_t^j)_{j \in D}}{\lambda_t^*}{ 1_{\{\lambda^*_t >0\}}}d\tilde N_t\big[\sqrt Th(\alpha)\big]-\frac{1}{2M_L^2T}	\int_0^T\big\{( X_t^j)_{j \in D}\big\}^{\otimes2} \varphi_L(X_t)\\&&\cdot \big\{1_{\{\lambda^*_t >0\}}dN_t +1_{\{\lambda^*_t =0\}}dt\big\}\big[\big(\sqrt Th(\alpha)\big)^{\otimes2}\big]
\\			&=&
K_T\big[\sqrt{T} h(\alpha)\big]-\frac{1}{2}\big\{G+r_T\big\}\big[\big(\sqrt Th(\alpha)\big)^{\otimes2}\big]
\qquad(t\geq 0, \alpha \in \Theta),
\eeas
where  
\beas
&& K_T\yeq \frac{1}{\sqrt T}\int_0^T\frac{( X_t^j)_{j \in D}}{\lambda_t^*}{ 1_{\{\lambda^*_t >0\}}}d\tilde N_t,\\
&& G \yeq  \frac{1}{M_L^2}\int_{ [0, \infty)^\sfa}\big\{( x_j)_{j \in D}\big\}^{\otimes2}\varphi_L(x)\,\big\{\ftrue\cdot x \, 1_{\{\ftrue\cdot x >0\}}+1_{\{\ftrue \cdot x =0\}}\big\}\nu(dx),
\\&& r_ T\yeq \frac{1}{M_L^2T}\int_0^T{\big\{( X_t^j)_{j \in D}\big\}^{\otimes2}}\varphi_L(X_t)1_{\{\lambda^*_t >0\}}d\tilde N_t\\
&& \qquad + \frac{1}{M_L^2T}\int_0^T {\big\{( X_t^j)_{j \in D}\big\}^{\otimes2}}\varphi_L(X_t)\,\big\{\lambda^*_t1_{\{\lambda^*_t >0\}}+1_{\{\lambda^*_t =0\}}\big\}dt-G.
\eeas 
Since from [{\bf L2}],
\beas
\sup_{t \geq 0}E\big[|K_T|^2\big]&=&\sup_{t \geq 0}E\bigg[\frac{1}{T}\int_{0}^T\bigg|\frac{( X_t^j)_{j \in D}}{\lambda_t^*}\bigg|^{2}  \lambda_t^*{1_{\{\lambda^*_t >0\}}}dt\bigg]  
\\&\leq&\sup_{t \geq 0}E\bigg[\bigg|\frac{( X_t^j)_{j \in D}}{\lambda_t^*}\bigg|^{2}  \lambda_t^*{1_{\{\lambda^*_t >0\}}}dt\bigg] \\
&=&\int_{[0, \infty)^\sfa}\frac{\big|( x_j)_{j \in D}\big|^2}{{\alpha^*\cdot x}}  { 1_{\{\alpha^*\cdot x >0\}}}\nu(dx)~<~ \infty,
\eeas 
we  obtain $K_T=O_P(1)$.  
Similarly,
\beas
\frac{1}{M_L^2T}\int_0^T{\big\{( X_t^j)_{j \in D}\big\}^{\otimes2}}\varphi_L(X_t)1_{\{\lambda^*_t >0\}}d\tilde N_t\yeq o_P(1).
\eeas
From the ergodicity, we have
\beas
\frac{1}{M_L^2T}\int_0^T {\big\{( X_t^j)_{j \in D}\big\}^{\otimes2}}\varphi_L(X_t)\,\big\{\lambda^*_t1_{\{\lambda^*_t >0\}}+1_{\{\lambda^*_t =0\}}\big\}-G\overset{P}{\to} 0.
\eeas
Therefore, 
\beas
r_T=o_P(1).
\eeas
Moreover,  (\ref{Rlarge})  implies the non-degeneracy of $G$.  Thus, [{\bf H1}] holds.


Second, we show [{\bf H2}].  
{From Lemma \ref{lemmafinal} described below, ${\rm
	Ker}A\cap\left<\{e_j\}_{j \in \calj_1}\right> =\{0\}$.  Therefore, for any $\alpha \in \Theta$ with $\alpha_k=0$ $(k \in \calj_0)$,
\beas
|h(\alpha)|&=&|(\alpha-\true)A|\yeq \bigg|\bigg(\sum_{j \in \calj_1}\alpha_je_j-\true\bigg)A\bigg|\\
&\geq& \epsilon_0|(\alpha_j)_{j \in \calj_1}-(\true_j)_{j \in \calj_1}|,
\eeas
where $\epsilon_0$ is some positive constant. Therefore,  {\bf (a)} and  {\bf(b)} of [{\bf H2}] hold,  and  $\Theta^*\cap\{\underline\theta=0\}=\{\alpha \in \Theta ; h(\alpha)=0\}\cap\{\alpha_k=0~(k \in \calj_0)\}=\{\theta^*\}$. Thus, [{\bf H2}] holds.
}

From [{\bf L1}],  Condition $[{\bf H3}]^\prime$ holds for $\Theta^*=\{\alpha \in \Theta ; h(\alpha)=0\}=\{\ftrue + {\rm Ker}(A)\} \cap [0, M_\alpha]^\sfa$, which implies [{\bf H3}]. 
Condition [{\bf H4}] obviously holds for $\xi_{j, T}$ defined in (\ref{xiex2}) {since $\kappa_j>0$ {and $q_j=q<1$} $(j=1, ..., \sfa)$. }

Finally, we show [{\bf H5}]. Take $\calg=\{\phi, \Omega\}$,  and take  $\caln$ as
\bea\label{defcaln}
\caln \yeq {\rm Int}(\Theta)\cap \big\{\alpha \in \bbR^\sfa ; |\alpha_i-\true_i| < 2^{-1}\true_i ~~(i \in \calj_1)\big\}.
\eea 
		Note that {$\caln$ satisfies (\ref{caln1}) and (\ref{caln2})} and that  for any $\alpha \in \ol\caln$, $\lambda_t^* >0$ implies $\lambda_t(\alpha)>0$ since $\lambda_t(\alpha) \geq 2^{-1}\lambda^*_t$.
		We have
		\beas
		\partial_\theta\calh_T(\true)&=&\int_{0}^T\frac{X_t}{\lambda_t^*}1_{\{\lambda_t^*>0\}}dN_t-\int_0^T X_t dt+\int_0^TX_t1_{\{\lambda^*_t=0\}}dt
		\\&=&\int_{0}^T\frac{X_t}{\lambda_t^*}1_{\{\lambda_t^*>0\}}d\tilde N_t.
		\eeas
		Similarly as before, from   [{\bf L2}],
		\beas
		\frac{1}{\sqrt T}\bigg|\int_{0}^T\frac{(X_t^j)_{j \in \calj_0}}{\lambda_t^*}1_{\{\lambda_t^*>0\}}d\tilde N_t\bigg| =O_P(1).
		\eeas
		Moreover, from [{\bf L2}] and the martingale central limit theorem, we have
		\beas
		S_T : =\frac{1}{\sqrt T}\int_{0}^T\frac{(X_t^j)_{j \in \calj_1}}{\lambda_t^*}1_{\{\lambda_t^*>0\}}d\tilde N_t \overset{d}{\to} \overline\Delta,
		\eeas
		{checking Lindeberg's condition as for any $a>0$,
			\beas
			E\sum_{t \leq T} \big(\Delta S_t\big)^21_{\{|\Delta S_t|>a\}} &\leq&a^{-1}E\sum_{t \leq T} \big|\Delta S_t\big|^3
			\\&\leq&a^{-1}E\bigg[\int_0^T \bigg|\frac{1}{\sqrt T}\frac{(X_t^j)_{j \in \calj_1}}{\lambda_t^*}1_{\{\lambda_t^*>0\}}\bigg|^3 dN_t\bigg]
			\\&=&a^{-1}T^{-\frac{1}{2}}E\bigg[\bigg|\frac{(X_0^j)_{j \in \calj_1}}{\lambda_0^*}\bigg|^3 \lambda^*_01_{\{\lambda_0^*>0\}}\bigg]  ~\to~ 0 \qquad(T\to \infty). 
			\eeas
		}
		Thus, the first half of the argument of [{\bf H5}] holds.
		
		Let us show the second half of the argument of [{\bf H5}]. That is, take any $R>0$ and we show	\bea\label{H5secondex2}
		&&\sup_{\substack{\alpha \in \ol\caln\\ |a_T^{-1}(\alpha-\true)|\leq R}}\bigg|a_T^\prime\partial^2_\theta\calh_T(\alpha)a_T+ \begin{pmatrix}
			\overline \Gamma & O\\
			O & O
		\end{pmatrix} \bigg|\overset{P}{\rightarrow}0.
	\end{eqnarray}
for $\partial_\theta^2\calh_T(\alpha)$ that satisfies the following equation:
{\beas
\partial_\theta^2\calh_T(\alpha)&=&-\int_{0}^T\frac{X_t^{\otimes2}}{\lambda_t(\alpha)^2}1_{\{\lambda^*>0\}}dN_t
-\int_0^T\frac{X_t^{\otimes2}}{M_L}\varphi_L(X_t)1_{\{\lambda^*_t=0\}}dt\nonumber\\
&=&-\int_{0}^T\frac{X_t^{\otimes2}}{\lambda_t(\alpha)^2}1_{\{\lambda^*>0\}}dN_t
-\int_0^T\frac{X_t^{\otimes2}}{M_L}\varphi_L(X_t)1_{\big\{X_t^i=0  ~(i\in\calj_1)\big\}}dt \\
&=&-\int_{0}^T\frac{X_t^{\otimes2}}{\lambda_t(\alpha)^2}1_{\{\lambda^*>0\}}dN_t
-\int_0^T\frac{\big\{(X_t^j1_{\{j \in \calj_0\}})_{ j=1, ..., \sfp}\big\}^{\otimes2}}{M_L}\varphi_L(X_t)1_{\big\{X_t^i=0 ~(i\in\calj_1)\big\}}dt\quad( \alpha \in \ol \caln). 
\eeas
Since $a_T$ is defined as (\ref{aTex2}) and $\frac{r}{q} >1$, we have
\beas
&&a_T^\prime\partial^2_\theta\calh_T(\alpha)a_T \\ &=&{a_T^\prime}\bigg(-\int_{0}^T\frac{X_t^{\otimes2}}{\lambda_t(\alpha)^2}1_{\{\lambda^*>0\}}dN_t\bigg){a_T}-\frac{1}{T^{\frac{r}{q} }}\int_0^T\frac{\big\{(X_t^j1_{\{j \in \calj_0\}})_{ j=1, ..., \sfp}\big\}^{\otimes2}}{M_L}\varphi_L(X_t)1_{\big\{X_t^i=0 ~(i\in\calj_1)\big\}}dt \\
&=&(\sqrt{T}a_T)^\prime\bigg(-\frac{1}{T}\int_{0}^T\frac{X_t^{\otimes2}}{\lambda_t(\alpha)^2}1_{\{\lambda^*>0\}}dN_t\bigg)(\sqrt Ta_T) + o(1).
\eeas
Thus,  for (\ref{H5secondex2}), it suffices to show
{\bea\label{letusshow}
	\sup_{\substack{\alpha \in \ol\caln\\ |a_T^{-1}(\alpha-\true)|\leq R}}\bigg|-\frac{1}{T}\int_{0}^T\frac{X_t^{\otimes2}}{\lambda_t(\alpha)^2}1_{\{\lambda^*>0\}}dN_t+\Gamma\bigg| \overset{P }{\to} 0,
	\eea }
where  $\Gamma$ is a $\sfa \times \sfa$ matrix defined  as
\beas
\Gamma \yeq  \int_{[0, \infty)^\sfa}\frac{x^{\otimes2}}{\ftrue\cdot x}1_{\{\ftrue \cdot x>0\}}\nu(dx).
\eeas
Take any $\epsilon>0$.  Then there exists some $T_0=T_0(R, \epsilon)$ such that  for any $T\geq T_0$ and any $\alpha \in \ol\caln$ with $|a_T^{-1}(\alpha-\true)|\leq R$,
\beas
(1- \epsilon)\lambda_t^*\leq \lambda_t(\alpha) \ \leq \lambda_t^* +\epsilon \sum_{j \in D}X_t^j ~~a.s.\qquad(t\geq 0).
\eeas
Therefore,  for any $T\geq T_0$ and any $\alpha \in \ol\caln$ with $|a_T^{-1}(\alpha-\true)|\leq R$, 
\beas
A_T [u^{\otimes2}]\leq \frac{1}{T}\int_{0}^T\frac{X_t^{\otimes2}}{\lambda_t(\alpha)^2}1_{\{\lambda^*>0\}}dN_t[u^{\otimes2}] \leq B_T[u^{\otimes 2}] \qquad(u \in \bbR^\sfa, |u|\leq 1),
\eeas
where 
\beas
A_T \yeq \frac{1}{T}\int_{0}^T\frac{X_t^{\otimes2}}{\big(\lambda_t^*+\epsilon\sum_{j \in D}X_t^j\big)^2}1_{\{\lambda^*>0\}}dN_t,\quad B_T \yeq \frac{1}{T}(1-\epsilon)^{-2} \int_{0}^T\frac{X_t^{\otimes2}}{(\lambda_t^*)^2}1_{\{\lambda^*>0\}}dN_t.
\eeas
for any $T>1$,  any $u \in \bbR^{\sfa}$ with $|u| \leq 1$ and any $\alpha \in \ol\caln$.
Since  from [{\bf L2}],
\beas
E\bigg[\bigg|\frac{1}{\sqrt T}\int_{0}^T\frac{X_t^{\otimes2}}{\big(\lambda_t^*+\epsilon\sum_{j \in D}X_t^j\big)^2}1_{\{\lambda^*>0\}}d\tilde N_t\bigg|^2\bigg]&\lesssim&E\bigg[\frac{1}{T}\int_{0}^T\frac{|X_t|^{4}}{\big(\lambda_t^*+\epsilon\sum_{j \in D}X_t^j\big)^4}\lambda^*_t 1_{\{\lambda^*>0\}}dt\bigg]\\
&\leq& E\bigg[\frac{1}{T}\int_{0}^T\frac{|X_t|^{4}}{(\lambda_t^*)^4}\lambda^*_t 1_{\{\lambda^*>0\}}dt\bigg]
\\&=& 	{\int_{[0, \infty)^\sfa}\frac{|x|^4}{\big(\ftrue\cdot x\big)^3} 1_{\{\ftrue\cdot x>0\}} \nu(dx)~<~ \infty}
\eeas
and since
\beas
\int_{[0, \infty)^\sfa}\frac{|x^{\otimes2}|}{\big(\ftrue \cdot x+\epsilon \sum_{j \in D}x_j\big)^2}\, \ftrue\cdot x\,\,1_{\{\ftrue\cdot x>0\}}\,\, \nu(dx) &\leq & \int_{[0, \infty)^\sfa}\frac{|x^{\otimes2}|}{\ftrue\cdot x}\,1_{\{\ftrue\cdot x>0\}}\,\, \nu(dx) < \infty,
\eeas
we have 
\beas
A_T &=&\frac{1}{T}\int_{0}^T\frac{X_t^{\otimes2}}{\big(\lambda_t^*+\epsilon\sum_{j \in D}X_t^j\big)^2}1_{\{\lambda^*>0\}}d\tilde N_t+\frac{1}{T}\int_{0}^T\frac{X_t^{\otimes2}}{\big(\lambda_t^*+\epsilon\sum_{j \in D}X_t^j\big)^2}\lambda_t^* 1_{\{\lambda^*_t>0\}}dt\\
&\overset{P}{\to}& \int_{[0, \infty)^\sfa}\frac{x^{\otimes2}}{\big(\ftrue\cdot x+\epsilon \sum_{j \in D}x_j\big)^2}\,\ftrue\cdot x\,1_{\{\ftrue\cdot x>0\}}\,\, \nu(dx) ~=:~ A(\epsilon).
\eeas
Similarly, 
\beas
B_T &=& (1-\epsilon)^{-2}\bigg\{\frac{1}{T}\int_{0}^T\frac{X_t^{\otimes2}}{(\lambda_t^*)^2}1_{\{\lambda^*_t>0\}}d\tilde N_t+\frac{1}{T}\int_{0}^T\frac{X_t^{\otimes2}}{\lambda_t^*}1_{\{\lambda^*_t>0\}}dt \bigg\}\\
&\overset{P}{\to}& (1-\epsilon)^{-2} \int_{[0, \infty)^\sfa}\frac{x^{\otimes2}}{\ftrue\cdot x}1_{\{\ftrue\cdot x>0\}}\,\nu(dx) ~=:~ B(\epsilon).
\eeas
Then for any $T\geq T_0$, 
\beas
&&\sup_{\substack{\alpha \in \ol\caln\\ |a_T^{-1}(\alpha-\true)|\leq R}}\sup_{u \in \bbR^{\sfa}, |u| \leq 1}\bigg|-\frac{1}{T}\int_{0}^T\frac{X_t^{\otimes2}}{\lambda_t(\alpha)^2}1_{\{\lambda^*>0\}}dN_t[u^{\otimes2}]+\Gamma[u^{\otimes2}]\bigg|\\
&\leq& \sup_{\substack{\alpha \in \ol\caln\\ |a_T^{-1}(\alpha-\true)|\leq R}}\sup_{u \in \bbR^{\sfa}, |u| \leq 1}\bigg(-\frac{1}{T}\int_{0}^T\frac{X_t^{\otimes2}}{\lambda_t(\alpha)^2}1_{\{\lambda^*>0\}}dN_t[u^{\otimes2}]+\Gamma[u^{\otimes2}]\bigg)^{+}\\
&&+\sup_{\substack{\alpha \in \ol\caln\\ |a_T^{-1}(\alpha-\true)|\leq R}}\sup_{u \in \bbR^{\sfa}, |u| \leq 1}\bigg(-\frac{1}{T}\int_{0}^T\frac{X_t^{\otimes2}}{\lambda_t(\alpha)^2}1_{\{\lambda^*>0\}}dN_t[u^{\otimes2}]+\Gamma[u^{\otimes2}]\bigg)^{-}
\\&\leq & \sup_{u \in \bbR^{\sfa}, |u| \leq 1}\bigg(-A_T[u^{\otimes2}]+\Gamma[u^{\otimes2}]\bigg)^{+}+\sup_{u \in \bbR^{\sfa}, |u| \leq 1}\bigg(-B_T[u^{\otimes2}]+\Gamma[u^{\otimes2}]\bigg)^{-}
\\ &\lesssim & \big|-A_T+\Gamma \big|+\big|-B_T+\Gamma\big|
~\overset{P}{\to}~  \big|-A(\epsilon)+\Gamma \big|+\big|-B(\epsilon)+\Gamma\big| \qquad(T\to\infty),
\eeas
where  $f^+:=f\vee0$ and $f^-:=f\wedge 0$  for any $\bbR$-valued function $f$.
Since $\lim_{\epsilon \to +0} A(\epsilon)= \lim_{\epsilon \to +0} B(\epsilon) = \Gamma$, we have
\beas
\sup_{\substack{\alpha \in \ol\caln\\ |a_T^{-1}(\alpha-\true)|\leq R}}\sup_{u \in \bbR^{\sfa}, |u| \leq 1} \bigg|-\frac{1}{T}\int_{0}^T\frac{X_t^{\otimes2}}{\lambda_t(\alpha)^2}1_{\{\lambda^*>0\}}dN_t[u^{\otimes2}]+\Gamma[u^{\otimes2}]\bigg| \overset{P }{\to} 0.
\eeas
Therefore,  (\ref{letusshow}) holds.
}
Thus,  the second half of the argument of [{\bf H5}] also holds. Then using Theorem \ref{thmY2},  we obtain (\ref{suffices1}).

Also, [{\bf S}] holds from Example \ref{exeasysparse}. Therefore, using Theorem \ref{Selection2}, we have 
\beas
\lim_{T\to \infty}P\big[(\widetilde \alpha_{k, T})_{k \in \calj_0}=0\big] =1.
\eeas
Since (\ref{suffices1}) holds and $\true_i \neq 0$ $(i \in \calj_1)$, we obtain (\ref{suffices2}). 

\end{proof}

\begin{proof}[Proof of Proposition \ref{linearprop}]
From the following Lemma \ref{lemmaPe}, 
\beas
\bigg\{\alpha \in \{\ftrue + {\rm Ker}(A)\}\cap[0, \infty)^{\sfa} ;  Pe(\alpha)= \underset{\tilde\alpha \in \{\ftrue + {\rm Ker}(A)\}\cap [0, \infty)^\sfa}\inf Pe(\tilde\alpha)\bigg\} &\subset& \{pr_E(\ftrue) ; E \in \cale\} \cap [0. \infty)^\sfa.
\eeas
Under  $[{\bf L1}]^{\#}$, the set on the right-hand side  has the unique minimizer $pr_{E_0}(\ftrue)$ of $Pe$ on the set itself.  Therefore, $pr_{E_0}(\ftrue)$  uniquely minimizes 
$Pe$ on $\{\ftrue + {\rm Ker}(A)\}\cap [0, \infty)^\sfa$. Since $pr_{E_0}(\ftrue) \in [0, M_\alpha)^\sfa$ under $[{\bf L1}]^{\#}$, $pr_{E_0}(\ftrue)$  also uniquely minimizes 
$Pe$ on $\{\ftrue + {\rm Ker}(A)\}\cap [0, M_\alpha]^\sfa$. Therefore, [{\bf L1}] holds and $\true =pr_{E_0}(\ftrue)$.
\end{proof}

\begin{lemma}\label{lemmaPe} For any $M \in  (0, \infty) \cup\{\infty\}$,
\bea\label{202212180716}
\bigg\{\alpha \in \{\ftrue + {\rm Ker}(A)\}\cap[0, M)^{\sfa} ;  Pe(\alpha)= \underset{\tilde\alpha \in \{\ftrue + {\rm Ker}(A)\}\cap [0, M)^\sfa}\inf Pe(\tilde\alpha)\bigg\}&\subset& \{pr_E(\ftrue) ; E \in \cale\}.
\eea
\end{lemma}

\begin{en-text}\begin{lemma}
\beas
\underset{\alpha \in \{y^* + {\rm Ker}(A)\}\cap [0, \infty)^\sfa}{\rm argmin} Pe(\alpha)\subset \{pr_E(y^*) ; E \in \cale\}.
\eeas
\end{lemma}
\end{en-text}
\begin{proof}
Take any $\alpha \in \{\ftrue + {\rm Ker}(A)\}\cap[0, M)^\sfa$ satisfying $\ds  Pe(\alpha)= \inf_{\tilde\alpha \in \{\ftrue + {\rm Ker}(A)\}\cap [0, M)^\sfa}Pe(\tilde\alpha)$.
Define $F$ as the set of all $j \in \{1, ..., \sfa\}$ with $\alpha_j\neq 0$. 
We prove  
\bea\label{lemcontra}
\left<\{e_j\}_{j \in F}\right> \cap {\rm Ker}A=\{0\}
\eea by contradiction. Suppose that $\left<\{e_j\}_{j \in F}\right> \cap {\rm Ker}A\neq \{0\}$. Then there exists $ (c_j)_{j \in F} \in \bbR^{|F|} \setminus\{0\}^{|F|}$ such that 
\beas
\sum_{j \in F}c_j e_j \in {\rm Ker}A.
\eeas
Take  $\underline \lambda<0$ and  $\ol \lambda>0$ such that for any $\lambda \in  (\underline \lambda, \ol \lambda )$,
\bea\label{final1}
\alpha -\lambda\sum_{j \in F}c_j e_j  \in [0, M)^\sfa.
\eea
Note that  for any $\lambda \in  (\underline \lambda, \ol \lambda )$,
\bea\label{fianl2}
\alpha -\lambda\sum_{j \in F}c_j e_j  \in \{\ftrue + {\rm Ker}(A)\}.
\eea
Define functions $f : (\underline \lambda, \ol \lambda) \to \bbR^\sfa$ and $g : (\underline \lambda, \ol \lambda) \to \bbR$ as
\beas
f(\lambda) \yeq   \alpha -\lambda\sum_{j \in F}c_j e_j, \quad g(\lambda)=Pe\big(f(\lambda)\big) \qquad(\underline \lambda< \lambda<\ol \lambda).
\eeas
Then $g(0)$ cannot be the local minimum of $g$ since $g^{\prime\prime}(\lambda)<0$ for any $\lambda \in (\underline \lambda, \ol \lambda)$.
Therefore, there exists some $\lambda_0 \in (\underline \lambda, \ol \lambda)$ such that
\beas
Pe(\alpha) \yeq g(0) ~>~ g(\lambda_0)\yeq Pe\big(f(\lambda_0)\big).
\eeas 
Also,  $f(\lambda_0) \in \{\ftrue + {\rm Ker}(A)\}\cap [0, M)^\sfa$ from (\ref{final1}) and (\ref{fianl2}).  This contradicts  the minimality of $\alpha$. Thus, (\ref{lemcontra}) holds.

From (\ref{lemcontra}), we can take some $E_1 \in \cale$ with $E_1 \supset F$. 
Then 
\beas
0&=&(\alpha -\ftrue)A \\
&=&\big(\alpha - pr_{E_1}(\ftrue)\big)A.
\eeas 
Since $\alpha \in \left<\{e_j\}_{j \in F}\right> \subset\left<\{e_j\}_{j \in E_1}\right> $ and $\left<\{e_j\}_{j \in E_1}\right> \cap {\rm Ker}A=0$, we have
\beas
\alpha \yeq pr_{E_1}(\ftrue).
\eeas
Therefore,  $\alpha \in \{pr_{E}(\ftrue) ; E \in \cale\}$. Thus, (\ref{202212180716}) holds. 

\end{proof}

{
\begin{lemma}\label{lemmafinal}
Assume $[{\bf L1}]$. Then  
\bea\label{12190155}
{\rm
	Ker}A\cap\left<\{e_j\}_{j \in \calj_1}\right> =\{0\}.
\eea
Moreover, if $[{\bf L2}]$ holds, then $\ol \Gamma$ is non-degenerate.
\end{lemma}
\begin{proof}  
From Lemma \ref{lemmaPe},  under [{\bf L1}],  $\true \in \{pr_{E}(\ftrue) ; E \in \cale\}$. 
Therefore, there exists some $E \in \cale$ such that $\calj_1  \subset E$. Thus, 
${\rm
	Ker}A\cap\left<\{e_j\}_{j \in \calj_1}\right> \subset {\rm Ker}A\cap\left<\{e_j\}_{j \in E}\right>=\{0\}$, and (\ref{12190155}) holds.

{We show that  $\ol \Gamma$ is non-degenerate under [{\bf L2}]. Assume  
	$\ds \ol \Gamma [v^{\otimes 2}]=0$ for some $v \in \bbR^{|\calj_1|}$. Then 
	\bea\label{olGammadege}
	\big|(x_i)_{i \in \calj_1}\cdot v\big|\,1_{\big\{\sum_{j=1}^\sfa \ftrue_jx_j >0\big\}}\yeq 0\quad\nu\text{-}a.e.\,x  \in  [0, \infty)^{\sfa}.
	\eea}
Since $\lambda_t^* = \sum_{j=1}^\sfa \ftrue_j  X_t^j =\sum_{j=1}^\sfa \true_j  X_t^j$ $(t\geq 0)$ almost surely, we have
\beas
1_{\big\{\sum_{j=1}^\sfa \ftrue_jx_j =0\big\}} \yeq 1_{\big\{\sum_{j=1}^\sfa \true_jx_j =0\big\}} \yeq 1_{\{x_j=0 ~(j \in \calj_1)\}} \quad\nu\text{-}a.e.\,x  \in  [0, \infty)^{\sfa}. \eeas
{Therefore, from (\ref{olGammadege}),
	\beas
	(x_i)_{i \in \calj_1}\cdot v \yeq 0\quad\nu\text{-}a.e.\,x  \in  [0, \infty)^{\sfa}.
	\eeas
	Thus, 
	\beas
	0&=& \int_{[0, \infty)^\sfa}\big((x_i)_{i \in \calj_1}\big)^{\otimes2}\nu(dx)~[v^{\otimes 2}]
	\\&=&\int_{[0, \infty)^\sfa}\big\{\big(e_i \cdot x \big)_{i \in \calj_1}\big\}^{\otimes2} \nu(dx) ~[v^{\otimes 2}]
	\\&=&\int_{[0, \infty)^\sfa}\big\{\big(e_iA \cdot (x_j)_{j \in D}\big)_{i \in \calj_1}\big\}^{\otimes2} \nu(dx) ~[v^{\otimes 2}]\qquad\big(\because (\ref{AandX})\big)\\
	&=&(e_iA)_{i \in \calj_1} \int_{[0, \infty)^\sfa}\big((x_j)_{j \in D}\big)^{\otimes2}  \nu(dx)~\big((e_iA)_{i \in \calj_1}\big)^{\prime}~[v^{\otimes 2}].
	\eeas}

\noindent Now $\{e_iA\}_{i \in \calj_1}$ is linearly independent from (\ref{12190155}).
\begin{en-text}
	Therefore,  under [{\bf L2}], 
	\beas
	\overline \Gamma &=& \int_{[0, \infty)^\sfa}\frac{\big((x_i)_{i \in \calj_1}\big)^{\otimes2}}{\sum_{j=1}^\sfa \ftrue_j x_j}1_{\big\{\sum_{j=1}^\sfa \ftrue_jx_j >0\big\}}\nu(dx)
	\\&=& \int_{[0, \infty)^\sfa}\frac{\big((x_i)_{i \in \calj_1}\big)^{\otimes2}}{\sum_{j=1}^\sfa \ftrue_j x_j}\bigg(1_{\big\{\sum_{j=1}^\sfa \ftrue_jx_j >0\big\}}+1_{\big\{x_j=0 ~(j \in \calj_1)\big\}}\bigg)\nu(dx)
	\\&=& \int_{[0, \infty)^\sfa}\frac{\big((x_i)_{i \in \calj_1}\big)^{\otimes2}}{\sum_{j=1}^\sfa \ftrue_j x_j} \nu(dx)
	\\&=&\int_{[0, \infty)^\sfa}\frac{\big\{\big(e_iA \,(x_j)_{j \in D}\big)_{i \in \calj_1}\big\}^{\otimes2}}{\sum_{j=1}^\sfa \alpha^*_j x_j} \nu(dx) \qquad\big(\because (\ref{AandX})\big)\\
	&=&(e_iA)_{i \in \calj_1} \int_{[0, \infty)^\sfa}\frac{\big((x_j)_{j \in D}\big)^{\otimes2} }{\sum_{j=1}^\sfa \alpha^*_j x_j} \nu(dx)~\big((e_iA)_{i \in \calj_1}\big)^{\prime}.
	\eeas
	Now $\{e_iA\}_{i \in \calj_1}$ is linearly independent from (\ref{12190155}).
\end{en-text}
Therefore, from (\ref{Gnonde2}), we obtain  $v=0$,  which implies the non-degeneracy of $\ol\Gamma$.
\end{proof}
}
\begin{en-text}
\begin{lemma}\label{lemmaP2}Define $\Theta^\dagger$ as 
\beas
\Theta^\dagger=\{\alpha \in [0, \infty)^\sfa ; h(\alpha)=0\}. 
\eeas
Then any minimizer  of $Pe$ on $\Theta^\dagger$ belongs to the following set:
\beas
\{pr_E(y^*) ; E \in \cale_+\}.
\eeas

\end{lemma}
\begin{proof}
Let $\alpha^\dagger=(\alpha^\dagger_1, ..., \alpha^\dagger_\sfa) \in \Theta^\dagger$ be any minimizer of $Pe$ on $\Theta^\dagger$.  Define $F^\dagger$ be a set of all $j \in \{1, ..., \sfa\}$ with $\alpha^\dagger_j\neq 0$. Note that $F^\dagger$ is not empty since $h(0)=-(a_j^*)_{j \in D}\neq 0$ and therefore $0 \notin \Theta^\dagger$.  
We prove  
\bea\label{lemcontra}
\left<\{e_j\}_{j \in F^\dagger}\right> \cap {\rm Ker}A=0
\eea by contradiction. Suppose that $\left<\{e_j\}_{j \in F^\dagger}\right> \cap {\rm Ker}A\neq 0$. Then there exists a nonempty subset $K^\dagger$ of $F^\dagger$ such that 
\beas
\left<\{e_j\}_{j \in F^\dagger \setminus K^\dagger}\right> \cap {\rm Ker}A=0,\\
\left<\{e_j\}_{j \in K^\dagger}\right> \subset \left<\{e_j\}_{j \in F^\dagger \setminus K^\dagger}\right>\oplus {\rm Ker}A.
\eeas
Take $E^\dagger \in \cale$ with $E^\dagger \supset F^\dagger \setminus K^\dagger$. Then for any $\alpha \in [0, \infty)^\sfa$ with $\alpha_j=0$ $\big(j \in (F^\dagger)^c \big)$,
\beas
h(\alpha)&=&\big(\sum_{j \in F^\dagger}\alpha_je_j - y^*\big)A \\
&=&\bigg(\sum_{j \in F^\dagger\setminus K^\dagger}\big(\alpha_j+C_j(\alpha_i)_{i \in K^\dagger}\big)e_j - y^*\bigg)A \\
&=&\bigg(\sum_{j \in F^\dagger\setminus K^\dagger}\big(\alpha_j+C_j(\alpha_i)_{i \in K^\dagger}\big)e_j - pr_{E^\dagger}(y^*)\bigg)A.
\eeas 
Since $h(\alpha^\dagger)=0$ and $\left<\{e_j\}_{j \in E^\dagger}\right> \cap {\rm Ker}A=0$, 
\beas
\sum_{j \in F^\dagger\setminus K^\dagger}\big(\alpha^\dagger_j+C_j(\alpha_i^\dagger)_{i \in K^\dagger}\big)e_j - pr_{E^\dagger}(y^*)\yeq 0.
\eeas
Especially, for any $j \in E^\dagger\setminus(F^\dagger\setminus K^\dagger)$, the $j$-th component of $pr_{E^\dagger}(y^*)$ is equal to $0$. Thus, for any $\alpha \in [0, \infty)^\sfa$ with $\alpha_j=0$ $\big(j \in (F^\dagger)^c \big)$,
\beas
\alpha \in \Theta^\dagger~~&\Leftrightarrow&~~ h(\alpha)=0 \nonumber\\
~~&\Leftrightarrow&~~ \big(\alpha_j+C_j(\alpha_i)_{i \in K^\dagger}\big) - pr_{E^\dagger}(y^*)[e_j]=0 ~~~(j \in F^\dagger\setminus K^\dagger) \label{Theta2}.
\eeas
Therefore,  for any $\alpha \in \Theta^\dagger$ with $\alpha_j=0$ $\big(j \in (F^\dagger)^c \big)$,
\bea
Pe(\alpha^\dagger) &\leq& Pe(\alpha)\nonumber\\
&=&\sum_{j \in F^\dagger\setminus K^\dagger}\kappa_j\big| pr_{E^\dagger}(y^*)[e_j]-C_j(\alpha_i)_{i \in K^\dagger}\big|^q+\sum_{j \in K^\dagger}\kappa_j|\alpha_j|^q.\label{Pe2}
\eea
However, considering the minimizers of  the formula (\ref{Pe2}), we see that there exists some $j \in F^\dagger$ such that $\alpha^\dagger_j=0$. This contradicts the definition of $F^\dagger$. Thus, (\ref{lemcontra}) holds.

From (\ref{lemcontra}), we can take some $E_1 \in \cale$ with $E_1  \supset F^\dagger$. 
Then 
\beas
h(\alpha^\dagger)&=&\big(\sum_{j \in F^\dagger}\alpha_j^\dagger e_j - y^*\big)A \\
&=&\bigg(\sum_{j \in F^\dagger}\alpha_j^\dagger e_j  - y^*\bigg)A \\
&=&\bigg(\sum_{j \in F^\dagger}\alpha_j^\dagger e_j  - pr_{E_1}(y^*)\bigg)A.
\eeas 
Since $h(\alpha^\dagger)=0$ and $\left<\{e_j\}_{j \in E_1}\right> \cap {\rm Ker}A=0$, we have
\beas
\alpha^\dagger=\sum_{j \in F^\dagger}\alpha_j^\dagger e_j  =pr_{E_1}(y^*).
\eeas
Therefore, $E_1 \in \cale_+$. Thus, any minimizer of $Pe$ on $\Theta^\dagger$ belongs to $\{pr_{E}(y^*) ; E \in \cale_+\}$. 

\end{proof}
\end{en-text}

\bibliographystyle{spmpsci} 
\bibliography{bibtex-NonregularCondition}

\end{document}

%% file: nakamacro031024.tex
\def\ol{\overline}

\def\bd{\begin{description}}
\def\ed{\end{description}}

\def\D2{\bbD_{2,\infty-}}

\def\D{{\bf D}}

\def\cala{{\cal A}}
\def\calb{{\cal B}}

\def\cale{{\cal E}}
\def\calf{{\cal F}}
\def\calg{{\cal G}}
\def\calh{{\cal H}}
\def\cali{{\cal I}}
\def\calj{{\cal J}}

\def\caln{{\cal N}}

\def\calt{{\cal T}}

\def\calw{{\cal W}}
\def\calx{{\cal X}}

\def\calz{{\cal Z}}
\def\ds{\displaystyle}
\def\yeq{\>=\>}

\def\sfk{{\sf k}}
\def\sfm{{\sf m}}

\def\sfd{{\sf d}}
\def\sfp{{\sf p}}
\def\sfq{{\sf q}}
\def\sfr{{\sf r}}

\def\y{\vspace*{3mm}\\}
\def\halflineskip{\vspace*{3mm}}

\def\be{\begin{equation}}
\def\ee{\end{equation}}
\def\bea{\begin{eqnarray}}
\def\eea{\end{eqnarray}}
\def\beas{\begin{eqnarray*}}
\def\eeas{\end{eqnarray*}}
\def\bi{\begin{itemize}}
\def\ei{\end{itemize}}
\def\im{\item}
\def\bd{\begin{description}}
\def\ed{\end{description}}
\newcommand{\bbD}{{\mathbb D}}

\newcommand{\bbH}{{\mathbb H}}

\newcommand{\bbR}{{\mathbb R}}

\newcommand{\bbT}{{\mathbb T}}

\newcommand{\bbV}{{\mathbb V}}

\newcommand{\bbZ}{{\mathbb Z}}

\newcommand{\sfa}{{\sf a}}
\newcommand{\sfb}{{\sf b}}

\def\sfr{{\sf r}}
\def\sfp{{\sf p}}
\def\sfd{{\sf d}}
\def\sfm{{\sf m}}

%% file: 20230122_arXiv_nonidentifiablemodel.bbl
\def\cprime{$'$} \def\cprime{$'$}
\begin{thebibliography}{10}
\providecommand{\url}[1]{{#1}}
\providecommand{\urlprefix}{URL }
\expandafter\ifx\csname urlstyle\endcsname\relax
  \providecommand{\doi}[1]{DOI~\discretionary{}{}{}#1}\else
  \providecommand{\doi}{DOI~\discretionary{}{}{}\begingroup
  \urlstyle{rm}\Url}\fi

\bibitem{aalen1978nonparametric}
Aalen, O.: Nonparametric inference for a family of counting processes.
\newblock The Annals of Statistics pp. 701--726 (1978)

\bibitem{andrews1999estimation}
Andrews, D.W.: Estimation when a parameter is on a boundary.
\newblock Econometrica \textbf{67}(6), 1341--1383 (1999)

\bibitem{atiyah1970resolution}
Atiyah, M.F.: Resolution of singularities and division of distributions.
\newblock Communications on pure and applied mathematics \textbf{23}(2),
  145--150 (1970)

\bibitem{brazzale2022likelihood}
Brazzale, A.R., Mameli, V.: Likelihood asymptotics in nonregular settings: A
  review with emphasis on the likelihood ratio.
\newblock arXiv preprint arXiv:2206.15178  (2022)

\bibitem{chernoff1954distribution}
Chernoff, H.: On the distribution of the likelihood ratio.
\newblock The Annals of Mathematical Statistics \textbf{25}, 573--578 (1954)

\bibitem{chornoboy1988maximum}
Chornoboy, E., Schramm, L., Karr, A.: Maximum likelihood identification of
  neural point process systems.
\newblock Biological cybernetics \textbf{59}(4), 265--275 (1988)

\bibitem{de2012adaptive}
De~Gregorio, A., Iacus, S.M.: Adaptive lasso-type estimation for multivariate
  diffusion processes.
\newblock Econometric Theory \textbf{28}(4), 838--860 (2012)

\bibitem{frank1993statistical}
Frank, L.E., Friedman, J.H.: A statistical view of some chemometrics regression
  tools.
\newblock Technometrics \textbf{35}(2), 109--135 (1993)

\bibitem{fu2000asymptotics}
Fu, W., Knight, K.: Asymptotics for lasso-type estimators.
\newblock The Annals of statistics \textbf{28}(5), 1356--1378 (2000)

\bibitem{gaiffas2019sparse}
Ga{\"\i}ffas, S., Matulewicz, G.: Sparse inference of the drift of a
  high-dimensional ornstein--uhlenbeck process.
\newblock Journal of Multivariate Analysis \textbf{169}, 1--20 (2019)

\bibitem{hironaka1964resolution}
Hironaka, H.: Resolution of singularities of an algebraic variety over a field
  of characteristic zero: Ii.
\newblock Annals of Mathematics pp. 205--326 (1964)

\bibitem{ibragimov1972asymptotic}
Ibragimov, I.A., Khas'minskii, R.Z.: The asymptotic behavior of statistical
  estimators in the smooth case. {I}. {S}tudy of the likelihood ratio.
\newblock Theory of Probability and its Applications \textbf{17}, 445--462
  (1973)

\bibitem{IbragimovHascprimeminskiui1981}
Ibragimov, I.A., Khas'minskii, R.Z.: Statistical estimation, \emph{Applications
  of Mathematics}, vol.~16.
\newblock Springer-Verlag, New York (1981).
\newblock Asymptotic theory, Translated from the Russian by Samuel Kotz

\bibitem{kalbfleisch2011statistical}
Kalbfleisch, J.D., Prentice, R.L.: The statistical analysis of failure time
  data.
\newblock John Wiley \& Sons (2011)

\bibitem{kinoshita2019penalized}
Kinoshita, Y., Yoshida, N.: Penalized quasi likelihood estimation for variable
  selection.
\newblock arXiv preprint arXiv:1910.12871  (2019)

\bibitem{lawless2011statistical}
Lawless, J.F.: Statistical models and methods for lifetime data.
\newblock John Wiley \& Sons (2011)

\bibitem{lindsey1995fitting}
Lindsey, J.: Fitting parametric counting processes by using log-linear models.
\newblock Journal of the Royal Statistical Society: Series C (Applied
  Statistics) \textbf{44}(2), 201--212 (1995)

\bibitem{liu2003asymptotics}
Liu, X., Shao, Y.: Asymptotics for likelihood ratio tests under loss of
  identifiability.
\newblock The Annals of Statistics \textbf{31}(3), 807--832 (2003)

\bibitem{masuda2017moment}
Masuda, H., Shimizu, Y.: Moment convergence in regularized estimation under
  multiple and mixed-rates asymptotics.
\newblock Mathematical Methods of Statistics \textbf{26}(2), 81--110 (2017)

\bibitem{mun2013superposed}
Mun, B.M., Bae, S.J., Kvam, P.H.: A superposed log-linear failure intensity
  model for repairable artillery systems.
\newblock Journal of Quality Technology \textbf{45}(1), 100--115 (2013)

\bibitem{pulcini2001modeling}
Pulcini, G.: Modeling the failure data of a repairable equipment with bathtub
  type failure intensity.
\newblock Reliability Engineering \& System Safety \textbf{71}(2), 209--218
  (2001)

\bibitem{tibshirani1996regression}
Tibshirani, R.: Regression shrinkage and selection via the lasso.
\newblock Journal of the Royal Statistical Society: Series B (Methodological)
  \textbf{58}(1), 267--288 (1996)

\bibitem{umezu2019aic}
Umezu, Y., Shimizu, Y., Masuda, H., Ninomiya, Y.: Aic for the non-concave
  penalized likelihood method.
\newblock Annals of the Institute of Statistical Mathematics \textbf{71}(2),
  247--274 (2019)

\bibitem{watanabe2009algebraic}
Watanabe, S.: Algebraic geometry and statistical learning theory, vol.~25.
\newblock Cambridge university press (2009)

\bibitem{xu2018superposed}
Xu, Q., Qiang, Z., Chen, Q., Liu, K., Cao, N.: A superposed model for the pipe
  failure assessment of water distribution networks and uncertainty analysis: a
  case study.
\newblock Water resources management \textbf{32}(5), 1713--1723 (2018)

\bibitem{yoshida2022quasi}
Yoshida, J., Yoshida, N.: Quasi-maximum likelihood estimation and penalized
  estimation under non-standard conditions.
\newblock arXiv preprint arXiv:2211.13871  (2022)

\bibitem{zou2006adaptive}
Zou, H.: The adaptive lasso and its oracle properties.
\newblock Journal of the American statistical association \textbf{101}(476),
  1418--1429 (2006)

\bibitem{zou2016application}
Zou, Y., Henrickson, K., Lord, D., Wang, Y., Xu, K.: Application of finite
  mixture models for analysing freeway incident clearance time.
\newblock Transportmetrica A: Transport Science \textbf{12}(2), 99--115 (2016)

\end{thebibliography}
